\documentclass[]{elsarticle}

\usepackage{lineno,hyperref}

\usepackage{graphicx}
\usepackage{bm}
\usepackage{multirow}
\usepackage{amsmath}
\usepackage{amsfonts}
\usepackage{amssymb}
\usepackage{amsthm}
\usepackage{secdot}

\journal{Computers \& Mathematics with Applications}

\theoremstyle{plain}
\newtheorem{theorem}{Theorem}[section]
\newtheorem{lemma}{Lemma}
\newtheorem{corollary}[theorem]{Corollary}

\theoremstyle{definition}

\theoremstyle{remark}
\newtheorem{remark}{Remark}

\numberwithin{equation}{section}
\numberwithin{theorem}{section}
\numberwithin{lemma}{section}
\numberwithin{remark}{section}

\newcommand{\Tbold}{\bm{T}}
\newcommand{\nbold}{\bm{n}}

\newcommand{\ubold}{\bm{u}}
\newcommand{\zbold}{\bm{z}}
\newcommand{\vbold}{\bm{v}}
\newcommand{\wbold}{\bm{w}}

\newcommand{\xbold}{\bm{x}}

\newcommand{\ipt}[2]{\left(#1,#2\right)_{\mathcal{T}_h}}

\newcommand{\ipbt}[2]{\left\langle#1,#2\right\rangle_{\partial \mathcal{T}_h}}

\newcommand{\ipbf}[2]{\left\langle#1,#2\right\rangle_{\mathcal{F}_h}}
\newcommand{\iipbf}[2]{\left\langle#1,#2\right\rangle_{\mathcal{F}_h^i}}

\newcommand{\llbracket}{\left[\!\left[}
\newcommand{\rrbracket}{\right] \! \right]}

\newcommand{\llcurve}{\left\{\!\left\{}
\newcommand{\rrcurve}{\right\} \! \right\}}
\newcommand{\vertiii}[1]{ \left\| #1 \right\|}

\newcommand{\vertiiisup}[1]{{\left\vert\kern-0.25ex\left\vert\kern-0.25ex\left\vert #1 
    \right\vert\kern-0.25ex\right\vert\kern-0.25ex\right\vert}}

\begin{document}

\begin{frontmatter}

\title{Versatile Mixed Methods for the Incompressible Navier-Stokes Equations}

\author{Xi Chen}
\author{David M. Williams \corref{mycorrespondingauthor}}
\address{Department of Mechanical Engineering, The Pennsylvania State University, University Park, Pennsylvania 16802}

\cortext[mycorrespondingauthor]{Corresponding author}
\ead{david.m.williams@psu.edu}

\begin{abstract}
In the spirit of the ``Principle of Equipresence" introduced by~Truesdell \& Toupin, \emph{The Classical Field Theories} (1960), we use the full version of the viscous stress tensor $\nu \left( \nabla \ubold + \nabla \ubold^T - \frac{2}{3} \left(\nabla \cdot \ubold \right) \mathbb{I} \right)$ which was originally derived for compressible flows, instead of the classical incompressible stress tensor $\nu \nabla \ubold$. (Note that, here $\nu$ is the dynamic viscosity coefficient, and $\ubold$ is the velocity field.) In our approach, the divergence-free constraint for the viscous stress term is not enforced ahead of discretization. Instead, our formulation allows the scheme itself to ``choose" a consistent way to interpret the divergence-free constraint: i.e., the divergence-free constraint is interpreted (or enforced) in a consistent fashion in both the mass conservation equation \emph{and} the stress tensor term (in the momentum equation). Furthermore, our approach preserves the original symmetrical properties of the stress tensor, e.g.~its rotational invariance, and it remains physically correct in the context of compressible flows. As a result, our approach facilitates versatility and code reuse. In this paper, we introduce our approach and establish some important mathematical properties for the resulting class of finite element schemes. More precisely, for general mixed methods, which are not necessarily pointwise divergence-free, we establish the existence of a new norm induced by the full, viscous bilinear form. Thereafter, we prove the coercivity of the viscous bilinear form and the semi-coercivity of a convective trilinear form. In addition, we demonstrate L2-stability of the discrete velocity fields for the general class of methods and (by deduction) the H(div)-conforming methods. Finally, we run some numerical experiments to illustrate the behavior of the versatile mixed methods, and we make careful comparisons with a conventional H(div)-conforming scheme. 
\end{abstract}


\end{frontmatter}

\section{Introduction} 
In this paper, we discuss the discretization of the incompressible Navier-Stokes equations using mixed finite element methods. It is important to note that mixed methods are not all the same, and that there are some important differences that distinguish them from each other. Standard mixed methods lack pressure robustness in the sense that the error estimate of velocity will be affected by the pressure approximation scaled by the inverse of the viscosity coefficient~\cite{john2017divergence}. This scaling will lead to poor convergence behavior, especially in convection-dominated flows. Another closely-related issue with many mixed methods is the weak enforcement of the divergence-free constraint. As a result, one loses the conservation of mass that is associated with the continuous system. There are many ways to address this issue for non-divergence-free methods. For example, for the standard Taylor-Hood velocity-pressure pair (which is H1-conforming in both velocity and pressure fields), one of the popular approaches is to add a grad-div stabilization term originally proposed by Franca and Hughes~\cite{franca1988two}. This term effectively penalizes the lack of mass conservation. As a result, one can improve solution accuracy by reducing the effect of the pressure error on the velocity error~\cite{olshanskii2004grad,olshanskii2009grad}. One thing worth noticing is that, although the grad-div term penalizes the failure to conserve mass, the resulting method may still be far from divergence free~\cite{john2017divergence}. For practical problems which require strict mass conservation, this remains an issue. (For some more recent work related to grad-div stabilization, see for example \cite{gelhard2005stabilized,braack2007stabilized,case2011connection,jenkins2014parameter,john2016finite}). One way to completely remedy the poor mass conservation is to use Scott-Vogelius elements~\cite{burman2008stabilized}, which maintain the H1-conforming nature of the velocity field, while allowing the pressure field to become discontinuous. For these methods, one can obtain pressure robustness, as now the weakly divergence-free velocity space is pointwise divergence-free~\cite{burman2008stabilized}. Note: in this context and throughout the remainder of the paper, whenever we say pointwise divergence-free, we mean pointwise divergence-free within each element. 
Finally, a typical way to remedy mass conservation issues for a discontinuous Galerkin method is to perform a post-processing procedure on the velocity field, or to penalize the jumps in the normal components of the velocity field. The post-processing procedure involves projecting the discrete velocity field into an exactly incompressible space~\cite{cockburn2005locally}, whereas the penalization procedure is enforced through the introduction of a dissipative numerical flux function~\cite{guzman2016h} which controls the jumps in the normal components. Similar procedures are discussed in~\cite{linke2014role,linke2016pressure}.

The key point is that, for many H(div)-conforming methods, we can completely omit the penalization and post-processing procedures described above. The principal advantage of these H(div)-conforming methods is that, by carefully choosing the velocity and pressure pair~\cite{boffi2013mixed}, one can make the velocity space to be pointwise divergence-free, while simultaneously enforcing an inf-sup condition on the velocity and pressure spaces. 
These two properties can be used to (naturally) ensure mass conservation and pressure robustness~\cite{schroeder2018divergence}.
There are many H(div)-conforming methods of this type, including Raviart-Thomas and Brezzi-Douglas-Marini based methods~\cite{guzman2016h,schroeder2018towards}, and related methods~\cite{zhang2005new,falk2013stokes,guzman2014conforming,lehrenfeld2016high}. This list of methods is far from exhaustive, but is merely meant to summarize some of the recent work in this area. 

In addition to pressure robustness and conservation of mass, finite element schemes should possess other desirable properties, such as conservation of linear momentum, angular momentum, and kinetic energy. Note that the conservation of kinetic energy only holds in the limit of vanishing viscosity, (see the discussions in \cite{majda2002vorticity,palha2017mass,coppola2019discrete} for details). It turns out that all of the important conservation properties (mentioned above) are exactly satisfied at the discrete level for pointwise divergence-free, H1-conforming methods. In addition, conservation of kinetic energy is guaranteed for pointwise divergence-free, H(div)-conforming methods that are equipped with an appropriate numerical flux, (e.g.~the central flux of~\cite{guzman2016h}). However, these properties are not guaranteed to hold for general mixed methods. For these methods, discrete conservation of linear momentum, angular momentum, and kinetic energy are closely linked to the discretization of the nonlinear convective term $\left(\ubold \cdot \nabla \right) \ubold$ that appears in the momentum equation (as discussed in \cite{schroeder2017pressure}). The conventional discretizations of this term, such as the direct discretization of $\left(\ubold \cdot \nabla \right) \ubold$, or the discretization of the skew-symmetric formulation, $\left(\ubold \cdot \nabla \right) \ubold + \left(1/2\right) \left(\nabla \cdot \ubold \right) \ubold$, fail to enforce discrete conservation. Instead, it is necessary to introduce the ``energy momentum and angular momentum conserving" (EMAC) formulation~\cite{charnyi2017conservation,charnyi2018efficient,charnyi2018emac,lehmkuhl2019low} which takes the following form: $(\nabla\ubold+\nabla\ubold^{T})\ubold+(\nabla\cdot\ubold)\ubold$. This form is interesting because it rewrites the convective term as a function of the symmetric gradient. We choose not to focus on EMAC methods in the remainder of this work, but we mention them here due to their recent popularity in finite element discussions, and their superficial similarity to our proposed approach.

In light of our previous discussion, which focused on the convective term, we now turn our attention to the classical viscous term, $\nabla \cdot \left( \nu \nabla \ubold \right)$. 
Naturally, this term is obtained by applying the divergence-free constraint to the compressible, symmetric stress tensor $\nu \left( \nabla \ubold + \nabla \ubold^T - \frac{2}{3} \left(\nabla \cdot \ubold \right) \mathbb{I} \right)$. Therefore, finite element methods that use the non-symmetric form of the stress tensor, i.e.~$\nu \nabla \ubold$, effectively assume that the divergence-free constraint should be enforced prior to discretization. Here, we propose that the divergence-free constraint should be enforced after discretization, and that the original symmetric formulation of the stress tensor should be retained. This facilitates philosophical consistency between the treatment of the divergence-free constraint in the mass equation \emph{and} the momentum equation. Furthermore, we prefer the full compressible tensor due to its symmetry, which enables rotational invariance, and its versatility, which facilitates the application of our methods to compressible flows.

Now, the idea of reformulating the stress tensor term by using a symmetric stress tensor instead of a non-symmetric stress tensor is not new. However, most finite element methods which utilize a symmetric stress tensor have been designed for solving elasticity problems, rather than fluids problems. 
Since the divergence of the Cauchy stress tensor contributes to the governing equations of elasticity,
the appropriate space for the stress tensor is $\bm{H}(\text{div}; \Omega; \mathsf{S})$, where $\mathsf{S}$ denotes a symmetric tensor. If one is able to work inside this space at the discrete level, then one obtains H(div)-conforming approximations and symmetry~\cite{arnold2002mixed,arnold2008finite,hu2014family,hu2014simple}. However, the resulting methods have some issues with computational efficiency~\cite{cockburn2017devising}, and therefore, some researchers relax the H(div)-conforming requirement while maintaining symmetry~\cite{arnold2003nonconforming,gopalakrishnan2011symmetric,cockburn2017devising}. 

The construction of mixed methods that preserve the symmetry of the stress tensor for fluids problems has only recently garnered significant attention. In particular, for Stokes flow, \cite{cockburn2017note} and \cite{giacomini2018superconvergent} constructed H(div)-nonconforming and symmetric methods. These methods are more expensive than standard mixed methods, as they introduce additional unknowns for the stress tensor components, (although this yields better accuracy in some cases). In contrast, one may consider less expensive mixed methods that omit the extra unknowns, and directly introduce the symmetric tensor into the primal formulation. Examples of these methods include the H1-conforming method of~\cite{tezduyar1991stabilized}, and the H(div)-conforming method of~\cite{hong2016robust,hong2016uniformly}. In the remainder of this paper, we propose mixed finite element methods that are natural extensions of these `primal methods'. Our work differs from the previous work in this area, as all previous efforts have utilized the symmetric tensor formulation $\nu \left( \nabla \ubold + \nabla \ubold^T \right)$, whereas we use the full symmetric tensor formulation $\nu \left( \nabla \ubold + \nabla \ubold^T - \frac{2}{3} \left(\nabla \cdot \ubold \right) \mathbb{I} \right)$. Furthermore, almost all prior work has focused on the Stokes equations, whereas we focus on the complete, incompressible Navier-Stokes equations. 

Our paper is organized as follows. In section \ref{prelim_section} we introduce some relevant notation and give the motivation for our formulation. In sections \ref{gen_method_sec} and \ref{div_method_sec} we present a general class of mixed finite element methods, and then introduce a specific class of pointwise divergence-free, H(div)-conforming methods. In section \ref{connections} we explore the key differences between our general approach and several alternative methods. In section \ref{Analysis_in_general} we prove some important properties of the bilinear and trilinear forms for the general case.
In section \ref{stability} we prove L2-stability of the velocity field for the general case. Next, in section \ref{numerical_simulation} we carry out some canonical simulations to show the behavior of our schemes. Finally, in section \ref{conclusion} we summarize our work, and in the Appendices we provide a detailed derivation of the schemes.

\label{sec;introduction}


\section{Preliminaries} \label{prelim_section}

Our objective is to solve the incompressible Navier-Stokes (NS) equations in a bounded, $\left(d+1\right)$-dimensional domain $\left(0, t_n \right) \times \Omega$, with boundary $\left(0, t_n \right) \times \partial \Omega$, where $t_n > 0$ is an arbitrary final time and $d=2$ or 3. Towards this end, we first define the pressure field $p\left(t,\xbold\right)$ and the velocity field $\ubold \left(t, \xbold \right)$. Throughout this discussion, we assume that the density is constant in both space and time, i.e., $\rho = \rho_0$, $\rho  \neq \rho \left(t, \xbold \right)$. It is common practice to specify $\rho_0$ as a unit value to simplify the presentation, and then to denote $p/\rho$ as simply $p$. We will not do this here, as this practice is not (strictly speaking) dimensionally consistent and it obfuscates the connection between the incompressible and compressible NS equations. Instead, we will define a new pressure variable $\widetilde{p} = p/\rho$. In what follows, $\widetilde{p}$ and $\ubold$ will serve as the primary parameters of interest. 

Now, we can formally introduce the incompressible NS equations, which take the form
\begin{align}
\partial_t \, \ubold + \nabla \cdot \left( \ubold \otimes \ubold + \widetilde{p} \, \mathbb{I} \right) - \nabla \cdot \widetilde{\bm{\tau}}  = \widetilde{\bm{f}},  \qquad & \text{in} \quad \left(0, t_n\right) \times \Omega, \label{moment_cons}\\[1.5ex]
\nabla \cdot \ubold = 0, \qquad & \text{in} \quad \left(0, t_n\right) \times \Omega, \label{mass_cons} \\[1.5ex]
\ubold=0,\qquad& \text{on} \quad \left(0, t_n\right) \times \partial \Omega,\\[1.5ex]
\ubold(0,\xbold)=\ubold_0(\xbold),\qquad& \text{in} \quad \Omega, 
\end{align}
where $\partial_t \left( \cdot \right)$ is the temporal derivative operator, $\nabla \left( \cdot \right)$ is the spatial gradient operator, $\widetilde{\bm{f}}$ is the density-weighted forcing function, and $\widetilde{\bm{\tau}}$ is the density-weighted viscous stress tensor
\begin{align}
\widetilde{\bm{\tau}} = \nu \left( \nabla \ubold + \nabla \ubold^T - \frac{2}{3} \left(\nabla \cdot \ubold \right) \mathbb{I} \right). \label{stress_tensor}
\end{align}
Here, $\nu = \mu/\rho_0$ is the kinematic viscosity coefficient, and $\mu$ is the dynamic viscosity coefficient. 

Our presentation of the incompressible NS equations differs from the classical presentation as we use a `compressible stress tensor' in Eq.~\eqref{stress_tensor} that contains the velocity gradient, the velocity gradient transpose, and the velocity divergence; whereas conversely, the classical `incompressible stress tensor' only contains the velocity gradient, as follows
\begin{align}
\widetilde{\bm{\tau}} = \nu \nabla \ubold. \label{stress_tensor_simp}
\end{align}
Of course, if our objective is to calculate $\nabla \cdot \widetilde{\bm{\tau}}$ in Eq.~\eqref{moment_cons}, the stress tensor formulations in Eqs.~\eqref{stress_tensor} and~\eqref{stress_tensor_simp} are equivalent, as the following identities hold at the continuous level
\begin{align}
    \nabla \cdot \left(\nu \nabla \ubold^T \right) = \nu \nabla \cdot \left(\nabla \ubold^T \right) = \nu \nabla \left(\nabla \cdot \ubold \right) & = 0, \label{cont_id_one} \\[1.5ex]
    -\nabla \cdot \left( \frac{2}{3} \nu \left( \nabla \cdot \ubold \right) \mathbb{I} \right) = - \frac{2}{3} \nu \nabla \cdot \left(  \left( \nabla \cdot \ubold \right) \mathbb{I} \right) &= 0,
\end{align}
because $\nabla \cdot \ubold$ vanishes pointwise in accordance with Eq.~\eqref{mass_cons}. However, the stress tensor formulations are different at the discrete level, as $\nabla \cdot \ubold$ does not necessarily vanish pointwise when a generic finite element method is applied to Eq.~\eqref{mass_cons}. Furthermore, even if we successfully choose a finite element method for which $\nabla \cdot \ubold$ vanishes pointwise, the velocity gradient transpose ($\nabla \ubold^T$) term in Eq.~\eqref{cont_id_one} does not necessarily vanish in the weak formulation of the equations. Because of these considerations, the formulation in Eq.~\eqref{stress_tensor} is inherently distinct from the formulation in Eq.~\eqref{stress_tensor_simp}. 

Based on the discussion above, it is not immediately clear which stress tensor formulation, Eq.~\eqref{stress_tensor} or \eqref{stress_tensor_simp}, should be utilized in practice. In what follows, we claim that the formulation in Eq.~\eqref{stress_tensor} is preferable, due to its superior versatility and flexibility. In particular:

\begin{enumerate}
    \item The formulation in Eq.~\eqref{stress_tensor} easily applies to both incompressible and \emph{compressible} fluids, whereas the formulation in Eq.~\eqref{stress_tensor_simp} only applies to incompressible fluids. We anticipate that the formulation in Eq.~\eqref{stress_tensor} will facilitate code-reuse, and allow for a more unified treatment of compressible and incompressible flows.

    \item The formulation in Eq.~\eqref{stress_tensor} is rotationally invariant, whereas the formulation in Eq.~\eqref{stress_tensor_simp} is rotationally \emph{variant}. This follows from the fact that Eq.~\eqref{stress_tensor_simp} contains the full velocity gradient tensor $\nabla \ubold$, which (by construction) is composed from symmetric \emph{and} antisymmetric parts,
    \begin{align*}
        \nabla \ubold &= \underbrace{\frac{1}{2} \left( \nabla \ubold + \nabla \ubold^T \right)}_{\text{symmetric}} + \underbrace{\frac{1}{2} \left( \nabla \ubold - \nabla \ubold^T \right)}_{\text{antisymmetric}} \\[1.5ex]
        &= \bm{\epsilon} \left(\ubold \right) + \bm{\omega} \left(\ubold \right).
    \end{align*}
    Conversely, Eq.~\eqref{stress_tensor} only contains the symmetric part. In order to highlight the significance of this distinction, we review the following basic result from the field of continuum mechanics. Let us consider two reference frames that are related by a proper orthogonal tensor $\bm{Q} \in \text{SO}(3)$, (i.e.~a rotation tensor), such that $\bm{x}^{\ast} - \bm{x}_{0}^{\ast} = \bm{Q} \left( \bm{x} - \bm{x}_0 \right)$. Then, we have rotational invariance for the symmetric part
    \begin{align*}
        \bm{\epsilon} \left(\ubold \right)^{\ast} = \bm{Q} \, \bm{\epsilon} \left(\ubold \right) \bm{Q}^T,
    \end{align*}
    but not for the antisymmetric part
    \begin{align*}
         \bm{\omega} \left(\ubold \right)^{\ast} = \bm{Q} \, \bm{\omega} \left(\ubold \right) \bm{Q}^T + \dot{\bm{Q}} \bm{Q}^T.
    \end{align*}

\end{enumerate}
%
%
In light of the above discussion, one may view Eqs.~\eqref{moment_cons} and \eqref{mass_cons} as a system of equations for unknowns $\widetilde{p}$ and $\ubold$, where Eq.~\eqref{stress_tensor} is a general constitutive relation.

Next, we will prepare to solve Eqs.~\eqref{moment_cons} and \eqref{mass_cons} at the discrete level by introducing the necessary mathematical machinery which consists of the following: a suitable subdivision of the domain $\Omega$ and notations for inner products, jumps, and function spaces. Towards this end, we introduce a mesh $\mathcal{T}_h$ that discretizes the domain $\Omega$ and contains elements $K$. We assume that the boundary of the domain, $\partial \Omega$, is composed from straight edges (or faces), and that the mesh conforms to the domain. In addition, each element in the mesh has a boundary $\partial K$ that is composed from a collection of faces denoted by $\mathcal{F}_{K}$. The total collection of faces in the mesh is denoted by $\mathcal{F}_h$. Therefore, if we denote an individual face by $F$, then the collection of faces for an element are defined such that: $\mathcal{F}_{K} = \left\{ F \in \mathcal{F}_h : F \subset \partial K \right\}$. Each interior face $F$ is shared by two elements, and we assume that the elements do not overlap, and that the mesh does not contain any hanging nodes. The collection of all interior faces is denoted by $\mathcal{F}_h^i = \{F \in \mathcal{F}_h : F \cap \partial\Omega = \emptyset \}$, and the collection of boundary faces by $\mathcal{F}_h^{\partial} = \{F \in \mathcal{F}_h : F \cap \partial\Omega \neq \emptyset \}$. Finally, an outward-pointing normal vector for an arbitrary element is denoted by $\nbold$, and a normal vector associated with a face $F$ is denoted by $\nbold_F$. The vector $\nbold_F$ is assumed to point from the positive (+) side of a face towards the negative (-) side.

We can now define inner products of vector-valued functions $\vbold$ and $\wbold$ over the elements and faces of the mesh as follows
\begin{align*}
\ipt{\vbold}{\wbold} & = \sum_{K \in \mathcal{T}_h} \int_{K} \vbold \cdot \wbold \, dV, \qquad \ipbt{\vbold}{\wbold}  = \sum_{K \in \mathcal{T}_h} \int_{\partial K} \vbold \cdot \wbold \, dA, \\[1.5ex]
\ipbf{\vbold}{\wbold} & = \sum_{F \in \mathcal{F}_h} \int_{F} \vbold \cdot \wbold \, dA.
\end{align*}
In addition, for a scalar-valued function $\phi$, a vector-valued function $\vbold$, and a second-order-tensor-valued function $\bm{T}$, we have the following integration by parts formulas
\begin{align*}
\int_{\partial K} \phi \left( \vbold \cdot \nbold \right) dA & = \int_{K} \left(\phi \left(\nabla\cdot \vbold \right) + \vbold \cdot \nabla \phi  \right) dV, \\[1.5ex]
\int_{\partial K} \vbold \cdot \Tbold \nbold \, dA &= \int_{K} \left(\vbold \cdot \left( \nabla \cdot \Tbold \right) + \Tbold : \nabla \vbold \right) dV,
\end{align*}
or equivalently
\begin{align*}
\left\langle \phi \vbold, \nbold \right\rangle_{\partial K} &= \left( \phi, \nabla \cdot \vbold \right)_{K} + \left( \vbold, \nabla \phi \right)_{K}, \\[1.5ex]
\left\langle \vbold, \bm{T} \nbold \right\rangle_{\partial K} &= \left(\vbold, \nabla \cdot \bm{T} \right)_{K} + \left(\bm{T}, \nabla \vbold \right)_{K}.
\end{align*}
For each of the integration formulas above, we have assumed that the integrands are well-defined such that the integrations make sense.

Next, we will introduce jump $\llbracket \cdot \rrbracket$ and average $\llcurve \cdot \rrcurve$ operators for an interface $F$ as follows
\begin{align*}
\llbracket \phi \rrbracket &= \phi_{+} - \phi_{-}, \qquad \llbracket \phi \nbold \rrbracket = \phi_{+} \nbold_{+} + \phi_{-} \nbold_{-}, \qquad \llcurve \phi \rrcurve = \frac{1}{2} \left( \phi_{+} + \phi_{-} \right), \\[1.5ex]
\llbracket \vbold \rrbracket &= \vbold_{+} - \vbold_{-}, \qquad \llbracket \vbold \otimes \nbold \rrbracket = \vbold_{+} \otimes \nbold_{+} + \vbold_{-} \otimes \nbold_{-}, \qquad  \llcurve \vbold \rrcurve = \frac{1}{2} \left( \vbold_{+} + \vbold_{-} \right),
\end{align*}
when $F\in\mathcal{F}^{i}_{h}$, and we define
\begin{align*}
\llbracket \phi \rrbracket &= \phi, \qquad \llbracket \phi \nbold \rrbracket = \phi \nbold, \qquad \llcurve \phi \rrcurve = \phi, \\[1.5ex]
\llbracket \vbold \rrbracket &= \vbold, \qquad \llbracket \vbold \otimes \nbold \rrbracket = \vbold \otimes \nbold, \qquad  \llcurve \vbold \rrcurve = \vbold,
\end{align*}
when $F\in\mathcal{F}^{\partial}_{h}$. We conclude this section by introducing some standard function spaces that are necessary for constructing mixed finite element methods. We begin by defining the following Hilbert spaces
\begin{align*}
   &\bm{H}_{0}(\text{div};\Omega)= \left\{ \wbold  :  \wbold\in \bm{L}^{2}(\Omega),~\nabla\cdot\wbold \in L^{2}(\Omega),~\wbold \cdot\nbold|_{\partial\Omega}=0\right\},\\[1.5ex] 
    &\bm{H}_{0}^{1}(\Omega)= \left\{ \wbold  :  \wbold\in \bm{H}^{1}(\Omega),~\wbold|_{\partial\Omega}=0\right\},\\[1.5ex] 
    &\bm{H}^{m}(\mathcal{T}_h) = \left\{\wbold  \in \bm{L}^{2}(\Omega), \wbold|_{K} \in\bm{H}^{m}(K),~\forall K \in \mathcal{T}_h \right\},
\end{align*}
where $\bm{H}^{1} \left(\Omega\right) = \left(H^{1} \left(\Omega\right) \right)^d$. Next, we can define the finite element spaces
\begin{align*}
&Q_h^{DC} = \left\{ q_h  : q_h \in L^2_{\ast} \left( \Omega \right), q_{h} |_{K} \in P_{k} \left( K \right), \forall K \in \mathcal{T}_h \right\}, \\[1.5ex]
&\bm{W}_h^{RT}  = \left\{ \wbold_h : \wbold_h \in \bm{H}_{0}\left(\text{div}; \Omega \right), \wbold_h |_{K} \in  \bm{RT}_k \left( K \right), \forall K \in \mathcal{T}_h \right\},
\end{align*}
where $L^2_{\ast} \left(\Omega\right)$ is the space of $L^2$ functions with zero mean, $P_{k} \left(K \right)$ is the space of polynomials of degree $\leq k$, and $ \bm{RT}_k \left(K \right)$ is the Raviart-Thomas space of degree~$k$
\begin{align*}
\bm{RT}_k \left(K \right) = \left(P_k \left(K \right) \right)^d \oplus P_k \left(K \right) \xbold.
\end{align*}
Note that the Raviart-Thomas space does not completely span $\left(P_{k+1} \left(K \right)\right)^d$. This limits the accuracy of the resulting methods, and as a result, one may instead consider the Brezzi-Douglas-Marini (BDM) space which maintains H(div)-conformity, while still spanning $\left(P_{k+1} \left(K \right)\right)^d$. We denote the BDM space by $\bm{W}_h^{BDM}$, and we refer the reader to~\cite{boffi2013mixed} for an explicit definition.

Finally, we define the following continuous function spaces that are necessary for constructing Taylor-Hood elements
\begin{align*}
    &Q_h^{TH} = \left\{ q_h  : q_h \in C^{0} \left( \Omega \right), q_{h} |_{K} \in P_{k} \left( K \right), \forall K \in \mathcal{T}_h \right\} \cap L_{\ast}^2 \left(\Omega \right), \\[1.5ex]
    &\bm{W}_h^{TH} = \left\{ \wbold_h : \wbold_h \in \bm{C}^{0} \left(\Omega\right), \wbold_h |_{K} \in  \left(P_{k+1} \left( K \right) \right)^d, \forall K \in \mathcal{T}_h \right\} \cap \bm{H}_{0}^{1}(\Omega),
\end{align*}
where $\bm{C}^{0} \left(\Omega\right) = \left(C^{0} \left(\Omega\right) \right)^d$. 





\section{General Mixed Methods} \label{gen_method_sec}

In this section, we introduce a new class of mixed finite element methods for discretizing Eqs.~\eqref{moment_cons} and \eqref{mass_cons}. This class of methods is very general, and it includes methods that are not necessarily pointwise divergence-free. We will utilize this class of methods to construct a particular class of pointwise divergence-free methods in the next section. For now, we consider the general framework of methods, which can be formally stated as follows: 1) identify function spaces $Q_h \subset L^2_{\ast} \left(\Omega \right)$ and $\bm{W}_h \subset \bm{H}_{0}(\text{div};\Omega)$; 2) choose test functions $\left(q_h, \wbold_h\right)$ that span $Q_h \times \bm{W}_h$; and 3) find unknowns $\left(\widetilde{p}_h, \ubold_h \right)$ in $Q_h  \times \bm{W}_h$ that satisfy 
%
%
\begin{align}
& \ipt{ \nabla \cdot \ubold_h}{q_h} = 0, \label{mass_cons_disc}
\end{align}
\begin{align}
\nonumber & \ipt{\partial_t \ubold_h }{\wbold_h} - \ipt{ \ubold_h \otimes \ubold_h}{\nabla_h \wbold_h} - \ipt{\widetilde{p}_h}{\nabla \cdot \wbold_h} + \ipbt{\hat{\bm{\sigma}}_{\text{inv}} \, \nbold}{\wbold_h} \\[1.5ex] 
\nonumber & + \nu_h \bigg[ \ipt{\nabla_h \ubold_h + \nabla_h \ubold_{h}^{T} - \frac{2}{3} \left(\nabla \cdot \ubold_h \right) \mathbb{I}}{\nabla_h \wbold_h} - \ipbt{\hat{\bm{\sigma}}_{\text{vis}} \, \nbold}{\wbold_h} \\[1.5ex]
\nonumber & + \ipbt{\hat{\bm{\varphi}}_{\text{vis}} - \ubold_h}{\left( \nabla_h \wbold_h + \nabla_h \wbold_{h}^{T} - \frac{2}{3} \left(\nabla \cdot \wbold_h \right) \mathbb{I} \right) \nbold} \bigg] \\[1.5ex]
&-\frac{1}{2} \ipt{\left(\nabla \cdot \bm{u}_h \right) \ubold_h}{\wbold_h} = \ipt{\widetilde{\bm{f}}}{\wbold_h}, \label{moment_cons_disc}
\end{align}
where the quantities with hats (e.g.~$\hat{\bm{\sigma}}_{\text{inv}}$) denote numerical fluxes.
We recommend that the numerical fluxes are specified in the following fashion
\begin{align*}
\hat{\bm{\sigma}}_{\text{inv}} &= \llcurve \ubold_h \rrcurve \otimes \llcurve \ubold_h \rrcurve + \llcurve \widetilde{p}_h \rrcurve \mathbb{I} + \zeta \left| \ubold_h \cdot \nbold_F \right| \llbracket \ubold_h \otimes \nbold \rrbracket, \\[1.5ex]
\hat{\bm{\sigma}}_{\text{vis}} & = \llcurve \nabla_h \ubold_h + \nabla_h \ubold_{h}^{T} - \frac{2}{3} \left(\nabla \cdot \ubold_h \right) \mathbb{I} \rrcurve -\frac{\eta}{h_F} \llbracket \ubold_h \otimes \nbold \rrbracket, \\[1.5ex]
\hat{\bm{\varphi}}_{\text{vis}} &= \llcurve \ubold_h \rrcurve,
\end{align*}
where $\zeta$ and $\eta$ are user-specified parameters that control the amount of dissipation that is added to the scheme. 

Now, it is standard practice to rewrite Eqs.~\eqref{mass_cons_disc} and \eqref{moment_cons_disc} as the following compact system
\begin{align}
& b_h \left(\ubold_h, q_h \right) = 0, \label{incomp_form_one} \\[1.5ex]
\nonumber & \ipt{\partial_t \, \ubold_h}{\wbold_h} + c_h \left(\ubold_h; \ubold_h, \wbold_h \right) +\nu_h a_h \left(\ubold_h, \wbold_h \right)  - b_h \left( \wbold_h, \widetilde{p}_h \right) = \ipt{\widetilde{\bm{f}}}{\wbold_h},  
\end{align}
where $a_h$ is a viscous bilinear form, $b_h$ is a convective bilinear form, and $c_h$ is a convective trilinear form. Each of the bilinear and trilinear forms in Eq.~\eqref{incomp_form_one} can be explicitly defined in terms of scalar function $q_h \in L^2_{\ast} \left(\Omega \right)$, and vector functions $\vbold_h$, $\wbold_h$, and $\bm{\beta}_h \in \bm{H}_{0}(\text{div};\Omega)$ as follows
\begin{align}
b_h \left(\vbold_h, q_h \right) &= \ipt{\nabla \cdot \vbold_h}{q_h},
\label{bilinear_press_div} \\[1.5ex]
c_h \left(\bm{\beta}_h ; \vbold_h, \wbold_h \right) & = - \ipt{\vbold_h \otimes \bm{\beta}_h}{\nabla_h \wbold_h} -\frac{1}{2} \ipt{\left(\nabla \cdot \bm{\beta}_h \right) \vbold_h}{\wbold_h} \\[1.5ex]
\nonumber & + \ipbt{\left( \llcurve \vbold_h \rrcurve \otimes \llcurve \bm{\beta}_h \rrcurve + \zeta \left|  \bm{\beta}_h \cdot \nbold_F \right| \llbracket \vbold_h \otimes \nbold \rrbracket \right) \nbold}{\wbold_h},
\end{align}
\begin{align}
a_h \left(\vbold_h, \wbold_h \right) &=  \ipt{ \nabla_h \vbold_h + \nabla_h \vbold_{h}^{T} - \frac{2}{3} \left(\nabla \cdot \vbold_h \right) \mathbb{I}}{\nabla_h \wbold_h} \label{viscous_1}\\[1.5ex] 
\nonumber & -  \ipbt{\left( \llcurve \nabla_h \vbold_h + \nabla_h \vbold_{h}^{T} - \frac{2}{3} \left(\nabla \cdot \vbold_h \right) \mathbb{I} \rrcurve -\frac{\eta}{h_F} \llbracket \vbold_h \otimes \nbold \rrbracket \right) \nbold}{\wbold_h} \\[1.5ex]
\nonumber & + \ipbt{\llcurve \vbold_h \rrcurve - \vbold_h}{\left( \nabla_h \wbold_h + \nabla_h \wbold_{h}^{T} - \frac{2}{3} \left(\nabla \cdot \wbold_h \right) \mathbb{I} \right) \nbold}.
\end{align} 
In addition, we can perform integration by parts on the trilinear form $c_h$, and then rewrite $c_h$ and the bilinear form $a_h$ in terms of summations over faces in the mesh as follows
\begin{align}
c_h \left(\bm{\beta}_h; \vbold_h, \wbold_h \right) &= \ipt{\bm{\beta}_h \cdot \nabla_h \vbold_h}{\wbold_h} + \frac{1}{2} \ipt{\left(\nabla \cdot \bm{\beta}_h \right) \vbold_h}{\wbold_h} \label{trilinear}\\[1.5ex]
& \nonumber - \iipbf{ \left( \bm{\beta}_h \cdot \nbold_F \right) \llbracket \vbold_h \rrbracket}{\llcurve \wbold_h \rrcurve} + \zeta \iipbf{\left| \bm{\beta}_h \cdot \nbold_F \right|\llbracket \vbold_h \rrbracket}{\llbracket \wbold_h \rrbracket}, 
\end{align}

\begin{align}
a_h \left(\vbold_h, \wbold_h \right) &=  \ipt{ \nabla_h \vbold_h + \nabla_h \vbold_{h}^{T} - \frac{2}{3} \left(\nabla \cdot \vbold_h \right) \mathbb{I}}{\nabla_h \wbold_h} \label{diff_bilinear} \\[1.5ex]
\nonumber & -\ipbf{\llbracket \vbold_h \rrbracket}{\llcurve  \nabla_h \wbold_h + \nabla_h \wbold_{h}^{T} - \frac{2}{3} \left(\nabla \cdot \wbold_h \right) \mathbb{I} \rrcurve \nbold_F}  \\[1.5ex]
\nonumber & -\ipbf{\llbracket \wbold_h \rrbracket}{\llcurve  \nabla_h \vbold_h + \nabla_h \vbold_{h}^{T} - \frac{2}{3} \left(\nabla \cdot \vbold_h \right) \mathbb{I} \rrcurve \nbold_F} + \ipbf{\frac{\eta}{h_F} \llbracket \vbold_h \rrbracket}{\llbracket \wbold_h \rrbracket}. 
\end{align}
Here, we assumed that a consistent boundary term was added to Eq.~\eqref{viscous_1} in order to obtain Eq.~\eqref{diff_bilinear}. Generally speaking, the consistency of the formulation holds for exact solutions $\widetilde{p}(t)\in L^{2}_{\ast}(\Omega)$ and $\bm{u}(t)\in\bm{H}^{\frac{3}{2}+\epsilon}(\mathcal{T}_h)\cap\bm{H}^{1}_{0}(\Omega)$ where $\epsilon>0$.
%


\section{Pointwise Divergence-Free, H(div)-Conforming Methods} \label{div_method_sec}

In this section, we introduce a class of pointwise divergence-free, H(div)-conforming methods for discretizing Eqs.~\eqref{moment_cons} and \eqref{mass_cons}. These methods are derived from section~\ref{gen_method_sec}, and can be stated formally as follows: 1) set function spaces $Q_h = Q_h^{DC}$ and $\bm{W}_h = \bm{W}_h^{RT}$ or $\bm{W}_h = \bm{W}_h^{BDM}$; 2) choose test functions $\left(q_h, \wbold_h\right)$ that span $Q_h \times \bm{W}_h$; and 3) find unknowns $\left(\widetilde{p}_h, \ubold_h \right)$ in $Q_h  \times \bm{W}_h$ that satisfy 
\begin{align}
& \ipt{ \nabla \cdot \ubold_h}{q_h} = 0, \label{mass_cons_disc_div}
\end{align}
\begin{align}
\nonumber & \ipt{\partial_t \ubold_h }{\wbold_h} - \ipt{ \ubold_h \otimes \ubold_h}{\nabla_h \wbold_h} - \ipt{\widetilde{p}_h}{\nabla \cdot \wbold_h} + \ipbt{\hat{\bm{\sigma}}_{\text{inv}} \, \nbold}{\wbold_h} \\[1.5ex] 
\nonumber & + \nu_h \bigg[ \ipt{\nabla_h \ubold_h + \nabla_h \ubold_{h}^{T}}{\nabla_h \wbold_h} - \ipbt{\hat{\bm{\sigma}}_{\text{vis}} \, \nbold}{\wbold_h} \\[1.5ex]
& + \ipbt{\hat{\bm{\varphi}}_{\text{vis}} - \ubold_h}{\left( \nabla_h \wbold_h + \nabla_h \wbold_{h}^{T} \right) \nbold} \bigg] = \ipt{\widetilde{\bm{f}}}{\wbold_h}. \label{moment_cons_disc_div}
\end{align}
We can now rewrite Eqs.~\eqref{mass_cons_disc_div} and \eqref{moment_cons_disc_div} as the following compact system
\begin{align*}
& b_h \left(\ubold_h, q_h \right) = 0, \\[1.5ex]
\nonumber & \ipt{\partial_t \, \ubold_h}{\wbold_h} + c_h \left(\ubold_h; \ubold_h, \wbold_h \right) +\nu_h a_h \left(\ubold_h, \wbold_h \right)  - b_h \left( \wbold_h, \widetilde{p}_h \right) = \ipt{\widetilde{\bm{f}}}{\wbold_h}.
\end{align*}
The bilinear form $b_h$ was previously defined in Eq.~\eqref{bilinear_press_div}. In addition, the forms $c_h$ and $a_h$ can be written as follows
\begin{align}
c_h \left(\bm{\beta}_h; \vbold_h, \wbold_h \right) &= \ipt{\bm{\beta}_h \cdot \nabla_h \vbold_h}{\wbold_h}  \\[1.5ex]
& \nonumber - \iipbf{ \left( \bm{\beta}_h \cdot \nbold_F \right) \llbracket \vbold_h \rrbracket}{\llcurve \wbold_h \rrcurve} + \zeta \iipbf{\left| \bm{\beta}_h \cdot \nbold_F \right|\llbracket \vbold_h \rrbracket}{\llbracket \wbold_h \rrbracket}, 
\end{align}

\begin{align}
\nonumber a_h \left(\vbold_h, \wbold_h \right) &=  \ipt{ \nabla_h \vbold_h + \nabla_h \vbold_{h}^{T}}{\nabla_h \wbold_h}   -\ipbf{\llbracket \vbold_h \rrbracket}{\llcurve  \nabla_h \wbold_h + \nabla_h \wbold_{h}^{T} \rrcurve \nbold_F} \\[1.5ex]
& -\ipbf{\llbracket \wbold_h \rrbracket}{\llcurve  \nabla_h \vbold_h + \nabla_h \vbold_{h}^{T} \rrcurve \nbold_F} + \ipbf{\frac{\eta}{h_F} \llbracket \vbold_h \rrbracket}{\llbracket \wbold_h \rrbracket}. 
\label{H(div)_bilinear_form}
\end{align}
One should note that these methods are pointwise divergence-free because $\bm{W}_h$ and $Q_h$ are `divergence-conforming' in the following sense: $\nabla \cdot \bm{W}_h \subseteq Q_h$, (see~\cite{john2017divergence} for more details).

\section{Connections with Other Approaches} \label{connections}

In this section, we explore the connections between the general finite element methods in section~\ref{gen_method_sec}, and previously established methods. We begin by discussing the relationship between the viscous bilinear form $a_h$ in Eq.~\eqref{diff_bilinear} and the classical concept of grad-div stabilization. Thereafter, we demonstrate that the conventional analysis of grad-div stabilization does not immediately apply to the new viscous bilinear form. Finally, we explore the connections between the new approach and several additional stabilization procedures.

\subsection{Relationship with Grad-Div Stabilization} \label{grad_div_discuss}

The standard grad-div stabilization term can be written as follows
\begin{align}
    \varepsilon \ipt{\nabla \cdot \vbold_h}{\nabla \cdot \wbold_h},
    \label{grad_div_form}
\end{align}
where $\varepsilon \geq 0$ is a user-defined parameter. Note that the origin of the name `grad-div stabilization' can be illustrated upon integrating Eq.~\eqref{grad_div_form} by parts 
\begin{align}
    \varepsilon \ipt{\nabla \cdot \vbold_h}{\nabla \cdot \wbold_h} = \varepsilon \left[ \ipbt{\nabla \cdot \vbold_h}{\wbold_h \cdot \nbold} -\ipt{\nabla_h \left(\nabla \cdot \vbold_h \right)}{\wbold_h} \right]. \label{grad_div_exp}
\end{align}
If one neglects the first term on the RHS of Eq.~\eqref{grad_div_exp}, or equivalently, if $\nabla \cdot \vbold_h$ is continuous and $\wbold_h$ vanishes on the boundary, then
\begin{align*}
     \varepsilon \ipt{\nabla \cdot \vbold_h}{\nabla \cdot \wbold_h} = - \varepsilon \ipt{\nabla_h \left(\nabla \cdot \vbold_h \right)}{\wbold_h},
\end{align*}
and the stabilization only consists of a grad-div volumetric term.

Now, in order to help stabilize the finite element methods introduced in section~\ref{gen_method_sec}, the grad-div term (Eq.~\eqref{grad_div_form}) can be added to the LHS of Eq.~\eqref{moment_cons_disc}. Upon performing this operation, and setting $\vbold_h = \wbold_h = \ubold_h$ one obtains
\begin{align*}
    \ipt{\partial_t \ubold_h }{\ubold_h} + \varepsilon \ipt{\nabla \cdot \ubold_h}{\nabla \cdot \ubold_h} + \left(\text{Remaining Terms}\right) = 0,
\end{align*}
or equivalently
\begin{align*}
    \frac{1}{2} \frac{d}{dt} \left\| \ubold_h \right\|_{\bm{L}^2 \left(\Omega\right)}^2 + \left(\text{Remaining Terms}\right) = - \varepsilon \left\| \nabla \cdot \ubold_h \right\|_{L^2 \left(\Omega \right)}^2.
\end{align*}
Naturally, the grad-div term helps stabilize the resulting scheme by decreasing the growth rate of the kinetic energy. 

With this in mind, we can now highlight the key distinction between grad-div stabilization and the formulation of $a_h$ in Eq.~\eqref{diff_bilinear}. In particular, upon setting $\vbold_h = \wbold_h = \ubold_h$ in Eq.~\eqref{diff_bilinear} and multiplying by the viscosity coefficient $\nu_h$, one obtains
\begin{align}
    \nu_h a_h \left(\ubold_h, \ubold_h\right) &= \nu_h \Bigg[ \ipt{ \nabla_h \ubold_h + \nabla_h \ubold_{h}^{T}}{\nabla_h \ubold_h} \label{diff_bilinear_sub} \\[1.5ex]
\nonumber & -2 \ipbf{\llbracket \ubold_h \rrbracket}{\llcurve  \nabla_h \ubold_h + \nabla_h \ubold_{h}^{T} \rrcurve \nbold_F} + \ipbf{\frac{\eta}{h_F} \llbracket \ubold_h \rrbracket}{\llbracket \ubold_h \rrbracket}   \\[1.5ex]
\nonumber & +\frac{4}{3} \ipbf{\llbracket \ubold_h \rrbracket}{\llcurve \left(\nabla \cdot \ubold_h \right) \mathbb{I} \rrcurve \nbold_F} - \frac{2}{3} \ipt{\nabla \cdot \ubold_h}{\nabla \cdot \ubold_h} \Bigg]. 
\end{align}
Upon examining Eq.~\eqref{diff_bilinear_sub}, we see that all of the dilatational terms have been placed on the last line. As a result of these terms, some non-standard analysis is required. In particular, let us focus on the second term of the last line
\begin{align*}
    - \frac{2 \, \nu_h}{3} \ipt{\nabla \cdot \ubold_h}{\nabla \cdot \ubold_h}.
\end{align*}
This term is identical to the grad-div stabilization term of Eq.~\eqref{grad_div_form} if we set $\varepsilon = -2 \, \nu_h/3$. Unfortunately, this action is not permitted since we require $\varepsilon \geq 0$ by construction. Therefore, we have shown that the stability (or coercivity) of $a_h$ in Eq.~\eqref{diff_bilinear} does not immediately follow from the standard arguments for grad-div stabilization, as $\varepsilon$ has the wrong sign. In what follows, we will briefly review some attempts to address this issue.

\subsection{Alternative Perspectives}

In his classical textbook, John has suggested (see~\cite{john2016finite}, p.~219) modifying the stress tensor by replacing the coefficient of the dilatation term $(-2\nu/3)$ with $(\xi - 2\nu/3)$, such~that
\begin{align*}
    \widetilde{\bm{\tau}} = \nu \left( \nabla \ubold + \nabla \ubold^T \right) + \left( \xi - \frac{2 \nu}{3} \right) \left(\nabla \cdot \ubold \right) \mathbb{I}.
\end{align*}
Here, $\xi$ is the `bulk viscosity coefficient'. For sufficiently large values of $\xi$ (i.e.~$\xi > 2\nu/3$), the coefficient of the dilatational term becomes positive, and the resulting volumetric contribution can be treated as a grad-div stabilization term. Unfortunately, in accordance with Stokes' hypothesis, it is commonplace to assume that $\xi = 0$. This hypothesis is exactly correct for monatomic gases, and for polyatomic gases, it produces accurate predictions over a wide range of flow conditions~\cite{schlichting2016boundary}. Therefore, it is difficult to justify choosing $\xi \neq 0$, and the resulting argument is not sufficiently general. 

Some authors (e.g. Peterson et al.~\cite{peterson2018overview}), have suggested ignoring the volumetric dilatational term entirely, based on principles from residual-based stabilization theory. In particular, they suggest balancing the volumetric dilatational term by implicitly adding a term of the form
\begin{align}
            &\ipt{\tau_{\text{LSIC}} \, \nabla \cdot \ubold_h }{\nabla \cdot \wbold_h} = \tau_{\text{LSIC}} \ipt{\nabla \cdot \ubold_h }{\nabla \cdot \wbold_h}  \label{resid_stab} \\[1.5ex]
            \nonumber &=\frac{2 \nu}{3} \ipt{\nabla \cdot \ubold_h}{\nabla \cdot \wbold_h},
\end{align}
where $\tau_{\text{LSIC}} = 2\nu/3$ is a least-squares incompressibility condition (LSIC) coefficient. This stabilization term is merely a re-interpretation of grad-div stabilization. The resulting method is only guaranteed to be stable if the inter-element jump terms vanish. Evidently, this assumption is valid for H1-conforming methods, but not for discontinuous Galerkin or H(div)-conforming methods. Furthermore, we claim that the resulting schemes are unnecessarily dissipative. In particular, the stabilization term in Eq.~\eqref{resid_stab} is unnecessary, as we can prove coercivity in its absence (see Lemma~\ref{coercive_vis_lemma}, in the next section).

Finally, Allaire has shown (see~\cite{allaire2007numerical}, p.~138), that there is an alternative approach that applies to $\bm{H}_0^1$-conforming methods, (or more specifically, $\bm{H}_0^1 \cap \bm{C}^0$-conforming methods). In order to illustrate this approach, 
one may start with the gradient transpose term $\ipt{\nabla_h \zbold^{T}}{\nabla_h \zbold}$, 
with $\zbold\in\bm{C}_{c}^{\infty} \left(\Omega\right)$ and $\bm{C}_{c}^{\infty} \left(\Omega\right)$ the space of vector-valued, compactly supported, infinitely differentiable functions. Then, we have
%
\begin{align}
    \ipt{\nabla_h \zbold^{T}}{\nabla_h \zbold} &= \ipbt{\nabla_h \zbold}{\nbold \otimes \zbold} - \ipt{\nabla_h \left( \nabla \cdot \zbold \right)}{\zbold} \label{ibp_ident} \\[1.5ex]
    \nonumber & = \ipbt{\nabla_h \zbold}{\nbold \otimes \zbold} - \ipbt{\nabla \cdot \zbold}{\zbold \cdot \nbold} + \ipt{\nabla \cdot \zbold}{\nabla \cdot \zbold} \\[1.5ex]
    \nonumber & = \ipt{\nabla \cdot \zbold}{\nabla \cdot \zbold}.
\end{align}
Now, we observe that due to the density of $C_{c}^{\infty} \left(\Omega\right)$ in $H_0^1 \left(\Omega\right)$, and with an abuse of notation, Eq.~\eqref{ibp_ident} holds for all  $\zbold \in \bm{H}_0^1 \left(\Omega\right)$. Upon enforcing the inter-element continuity of $\ubold_h$, and substituting this identity into Eq.~\eqref{diff_bilinear_sub} with $\zbold=\ubold_h$ one obtains
\begin{align}
    \nu_h a_h \left(\ubold_h, \ubold_h\right) &= \nu_h \left[ \ipt{ \nabla_h \ubold_h}{\nabla_h \ubold_h} + \frac{1}{3} \ipt{\nabla \cdot \ubold_h}{\nabla \cdot \ubold_h} \right]. \label{bilinear_viscous_h1}
\end{align}
The second term on the RHS of Eq.~\eqref{bilinear_viscous_h1} is simply a grad-div stabilization term with $\varepsilon = \nu_h/3$. Therefore, we have shown that the negative dilatational contribution is completely cancelled by the gradient transpose term for $\bm{H}_0^1 \cap \bm{C}^0$-conforming methods. Of course, the assumption of $\bm{H}_0^1$-conformity means that the approach of Allaire does not apply to general H1-conforming methods with non-zero velocity boundary conditions, H(div)-conforming methods, or discontinuous Galerkin methods.

With all of the  limitations and restrictions (above) in mind, in the next section we explore a significantly more general approach.


\section{Analysis of the Bilinear and Trilinear Forms: General Case}\label{Analysis_in_general}

In this section, we prove that there is a norm associated with the viscous bilinear form $a_h$ in Eq.~\eqref{diff_bilinear}. Furthermore, we prove that the viscous bilinear form $a_h$ is coercive and show that the convective trilinear form $c_h$ in Eq.~\eqref{trilinear} is semi-coercive. 

\begin{lemma}[Establishing a New Norm]
Suppose $\wbold \in \bm{H}^{1}(\mathcal{T}_h)$, the space of piecewise H1 vector fields, and $d=2$ or 3, then
\begin{align}
\vertiii{\wbold}_{\text{sym}} = \left(\left\| \nabla_{h}\wbold+\nabla_h \wbold^{T}-\frac{2}{3} \left( \nabla_h \cdot\wbold \right) \mathbb{I} \right\|_{\bm{L}^2 \left(\Omega \right) \times \bm{L}^2 \left(\Omega \right)}^2 + \ipbf{\frac{1}{h_{F}}\llbracket\wbold\rrbracket}{\llbracket\wbold\rrbracket}  \right)^{1/2},\label{definition}
\end{align}
is a norm on $\Omega$.
\end{lemma}
\begin{proof}
The only nontrivial part of the proof is to show that if $\vertiii{\wbold}_{\text{sym}}=0$, then $\bm{w}=0$. With this in mind, let us first consider the case of $d=2$ and find the kernel of
\begin{align}
\nabla\wbold+\nabla \wbold^{T}-\frac{2}{3} \left( \nabla \cdot\wbold \right) \mathbb{I}, \label{kernel}
\end{align}
for a single element. Towards this end, we can obtain four equations, one for each of the tensor components as follows 
\begin{align*}
\frac{\partial w_i}{\partial x_j}+\frac{\partial w_j}{\partial x_i}-\frac{2}{3}\frac{\partial w_k}{\partial x_k}\delta_{ij}=0.
\end{align*}
If we consider only the diagonal components, we find that
\begin{align*}
\frac{\partial w_1(x_1, x_2)}{\partial x_1}=\frac{\partial w_2 (x_1,x_2)}{\partial x_2}=0.
\end{align*}
Therefore, $w_1 = w_1 \left(x_2\right)$ and $w_2 = w_2 \left(x_1\right)$.
Next, based on the off-diagonal components of the tensor, we conclude that
\begin{align*}
\frac{\partial w_1(x_2)}{\partial x_2}=-\frac{\partial w_2(x_1)}{\partial x_1}=k_1, \label{equation_1}
\end{align*}
where $k_1$ is a generic constant. Thus, the kernel space $\mathcal{W}_{ker}$ of Eq.~\eqref{kernel} can be expressed as follows: $\wbold \in \mathbb{R}^2$ such that
\begin{align}
\wbold 
= \begin{bmatrix}
          w_1 \\
          w_2
\end{bmatrix}
=\begin{bmatrix}
0 & k_1 \\
-k_1 & 0 
\end{bmatrix}\begin{bmatrix}
           x_1 \\
           x_2 \\
\end{bmatrix}+
\begin{bmatrix}
           k_2 \\
           k_3 \\
\end{bmatrix},
\end{align}
where $k_2$ and $k_3$ are constants. 

Now, since we have established an explicit definition for the kernel space on each element, we can assert that $\vertiii{\wbold}_{\text{sym}}=0$ implies that $\wbold \in \mathcal{W}_{ker}$ on each element. Furthermore, by the definition of $\vertiii{\cdot}_{\text{sym}}$, we note that $\vertiii{\wbold}_{\text{sym}}=0$ implies that $\wbold$ is continuous across interior edges, and that $\wbold = 0$ on the boundary of the domain. It turns out that if $\wbold \in \mathcal{W}_{ker}$ over an arbitrary element and $\wbold = 0$ on at least one edge of the element, then $\wbold$ vanishes over the entire element (see~\cite{allaire2007numerical}, p.~140). By this fact, and the fact that $\wbold$ vanishes on the boundary of the domain and is continuous over interior edges, then $\wbold$ vanishes over the entire domain. This completes our proof for the case of~$d =2$. \\[1.5ex]
Now, let us consider the case of $d=3$. In accordance with \cite{schirra2012new,breit2017trace}, the kernel space $\mathcal{W}_{ker}$ can be written as follows: $\wbold \in \mathbb{R}^3$ such that
\begin{align}
\label{w_ker}
\nonumber\bm{w} &= \begin{bmatrix}
          w_1 \\
          w_2 \\
          w_3
\end{bmatrix}
=\begin{bmatrix}
          0  & k_1 & k_2\\
          -k_1 & 0 & k_3 \\
          -k_2 & -k_3 & 0
\end{bmatrix}\begin{bmatrix}
          x_1 \\
          x_2 \\
          x_3
\end{bmatrix}+
\begin{bmatrix}
          k_4 \\
          k_5 \\
          k_6
\end{bmatrix}-(x_1^{2}+x_2^{2}+x_3^{2})\begin{bmatrix}
          k_8 \\
          k_9 \\
          k_{10}
\end{bmatrix}\\[1.5ex]
&+\big(2(k_8x_1+k_9x_2+k_{10}x_3)+k_7\big)
\begin{bmatrix}
          x_1 \\
          x_2 \\
          x_3
\end{bmatrix},
\end{align}
where $k_1, k_2, \ldots, k_{10}$ are constants. 

The key is to prove that if $\wbold \in \mathcal{W}_{ker}$ vanishes on a plane (e.g.~the face of an element), then $\wbold =0$ throughout the element. Without loss of generality, we begin by assuming that $\wbold$ vanishes on a plane of the form  
\begin{align}
x_1=bx_2+cx_3+d,    \label{plane}
\end{align}
%
where $b$, $c$, and $d$ are constants. Next, we substitute Eq.~\eqref{plane} into Eq.~\eqref{w_ker}, and rearrange the result in order to obtain
\begin{align}
    \nonumber &\left(k_4 + d k_7 + d^2 k_8 \right) + \left( k_1 + b k_7 + 2 b d k_8 + 2 d k_9 \right) x_2 \\
    \nonumber & + \left(- k_8  + b^2 k_8  + 2 b k_9 \right) x_{2}^2
    + \left( 2 d k_{10} + k_2 + c k_7 + 2 c d k_8 \right) x_3 \\
    & + \left( 2 b k_{10} + 2 b c k_8 + 2 c k_9 \right) x_2 x_3 + \left( 2 c k_{10} - k_8 + c^2 k_8 \right) x_{3}^{2} = 0,  \label{equ_1} \\[1.5ex]
\nonumber & \left(-d k_1 + k_5 - d^2 k_9 \right) +\left( - b k_1 + k_7 + 2 d k_8 - 2 b d k_9 \right) x_2 \\
& \nonumber + \left(
  2 b k_8 + k_9 - b^2 k_9 \right) x_2^2 
  + \left( - c k_1 + k_3 - 
  2 c d k_9 \right) x_3 \\
  & + \left( 2 k_{10} + 2 c k_8 - 2 b c k_9 \right) x_2 x_3 + \left( - k_9 -
   c^2 k_9 \right) x_3^2 = 0,  \label{equ_2} \\[1.5ex]
\nonumber & \left(-d^2 k_{10} - d k_2 + k_6 \right) + \left( - 2 b d k_{10} - b k_2 - k_3 \right) x_2 \\
\nonumber & + \left( - k_{10} - b^2 k_{10} \right) x_2^2 
+ \left( - 2 c d k_{10} - c k_2 + k_7 + 2 d k_8 \right) x_3 \\
& + \left( - 2 b c k_{10} + 2 b k_8 + 2 k_9 \right) x_2 x_3 + \left( k_{10} -  c^2 k_{10} + 2 c k_8 \right) x_3^2 = 0.  \label{equ_3}
\end{align}
In order for these equations to hold, the coefficients in front of the monomial terms $x_2, x_3, x_2x_3, x_2^{2}$ and $x_{3}^{2}$ must vanish. More specifically, by examining the coefficient of $x_3^2$ in Eq.~\eqref{equ_2} and the coefficient of $x_2^2$ in Eq.~\eqref{equ_3}, we obtain that $k_9 = k_{10} = 0$. Next, we divide the remainder of the proof into the following three cases:\\[1.5ex]
\emph{Case 1)} $b \neq 0, c \in \mathbb{R}, d \in \mathbb{R}$ \\[1.5ex]
Upon inspecting the coefficient of $x_2^2$ in Eq.~\eqref{equ_2}, we immediately obtain that $k_8=0$. The remaining nontrivial equations are 
\begin{align*}
k_5-dk_1&=0, & -bk_1+k_7&=0, & -ck_{1}+k_3&=0,\\
k_4+dk_7&=0, &  bk_7+k_1&=0, & ck_7+k_2&=0,\\
k_6-dk_2&=0, & -bk_2-k_3&=0, & -ck_2+k_7&=0. 
\end{align*}
Upon multiplying the third equation of the second row by $k_2$ and the third equation of the third row by $k_7$, and then adding the results together, we obtain $k_{2}^{2}+k_{7}^{2}=0$. Therefore, $k_2 = k_7 = 0$. Now, it is straightforward to obtain $k_1=k_3=k_4=k_5=k_6=0$ by inspecting the remaining equations in rows 1--3.\\[1.5ex]
\emph{Case 2)} $b=0, c \neq 0, d \in \mathbb{R}$ \\[1.5ex]
Upon inspecting the coefficient of $x_3^2$ in Eq.~\eqref{equ_3}, we immediately obtain that $k_8=0$. The remaining nontrivial equations are 
\begin{align*}
k_4+dk_7&=0, & k_1&=0, &ck_7+k_2&=0,  \\
k_5-dk_1&=0, & k_7&=0, & -ck_{1}+k_3&=0,\\
k_6-dk_2&=0, & k_3&=0, &  -ck_2+k_7&=0.
\end{align*}
It immediately follows that $k_1=k_2=k_3=k_4=k_5=k_6=k_7=0$ by examining these equations.\\[1.5ex]
\emph{Case 3)} $b=0, c=0, d \in \mathbb{R}$ \\[1.5ex]
Finally, upon inspecting the coefficient of $x_2^2$ in Eq.~\eqref{equ_1}, we immediately obtain that $k_8 = 0$. The remaining  nontrivial equations are 
\begin{align*}
k_4+dk_7&=0, & k_1&=0, & k_2&=0,  \\
k_5-dk_1&=0, & k_7&=0, & k_3&=0, \\
k_6-dk_2&=0, & k_3&=0, & k_7&=0.
\end{align*}
It immediately follows that $k_1=k_2=k_3=k_4=k_5=k_6=k_7=0$ by examining these equations. 

Therefore, we have proved that if $\wbold \in \mathcal{W}_{ker}$ vanishes on at least one face of an element, then $\wbold =0$ on the entire element. Finally, the proof for the case of $d=3$ is completed by observing that if $\vertiii{\wbold}_{\text{sym}}=0$, then we have $\wbold \in \mathcal{W}_{ker}$ on each element, $\wbold$ vanishes on the domain boundary, and $\wbold$ is continuous across the interior faces. These facts combine to ensure that $\wbold$ vanishes throughout the domain.
\end{proof}

\begin{lemma}[Coercivity of the Viscous Bilinear Form]
Suppose we choose generic test functions $\wbold_h, \vbold_h \in \bm{W}_h$, and we assume that $d = 2$ or 3. Furthermore, we choose $\eta > 2C_{\text{tr}}^2N_{\partial}$, where $C_{\text{tr}}$ and $N_{\partial}$ are constants which depend on the mesh topology. Then, the bilinear form $a_h$ in Eq.~\eqref{diff_bilinear} is coercive on $\bm{W}_h$, such that
\begin{align}
    \forall \wbold_h \in \bm{W}_h,\qquad a_h \left( \wbold_h, \wbold_h \right) \geq C \vertiii{\wbold_h}_{\text{sym}}^2, \label{coercive_res}
\end{align}
where $C$ is a positive constant independent of $h$.
\label{coercive_vis_lemma}
\end{lemma}

\begin{proof}
Consider the following identity
\begin{align*}
\nonumber & 2 \left( \nabla_{h}\wbold_{h}:\nabla_{h}\vbold_{h}+\nabla_{h}\wbold_{h}:\nabla_{h}\vbold_{h}^{T}-\frac{2}{3}\left( \nabla_h \cdot\wbold_{h} \right) \left(\nabla_h \cdot\vbold_{h} \right) \right) \\[1.5ex]
& = \left(\nabla_{h} \vbold_{h}+\nabla_{h}\vbold_{h}^T-\frac{2}{3} \left(\nabla_h \cdot\vbold_{h} \right)\mathbb{I} \right): \left(\nabla_{h}\wbold_{h}+\nabla_{h}\wbold_{h}^T-\frac{2}{3} \left( \nabla_h \cdot \wbold_h \right) \mathbb{I} \right) \\[1.5ex]
\nonumber &+\frac{4 \left(3-d\right)}{9} \left(\nabla_h \cdot \wbold_h \right) \left(\nabla_h \cdot \vbold_h \right).
\end{align*}
Using this result, the bilinear form in Eq.~\eqref{diff_bilinear} becomes
\begin{align}
a_{h}(\vbold_{h},\wbold_{h})&=\frac{1}{2} \ipt{\nabla_{h} \vbold_{h}+\nabla_{h}\vbold_{h}^T-\frac{2}{3} \left(\nabla_h \cdot\vbold_{h} \right)\mathbb{I}}{\nabla_{h} \wbold_{h}+\nabla_{h}\wbold_{h}^T-\frac{2}{3} \left(\nabla_h \cdot\wbold_{h}\right)\mathbb{I}} \label{diff_bilinear_one} \\[1.5ex]
\nonumber &+ \frac{2 \left(3-d\right)}{9} \ipt{\nabla_h \cdot \vbold_h}{\nabla_h \cdot \wbold_h} - \ipbf{ \llbracket\vbold_{h}\rrbracket}{\llcurve \nabla_{h}\wbold_{h}+\nabla_{h}\wbold_{h}^{T}-\frac{2}{3} \left(\nabla_h \cdot\wbold_{h} \right) \mathbb{I} \rrcurve\nbold_F} \\[1.5ex]
\nonumber &- \ipbf{\llbracket\wbold_{h}\rrbracket}{\llcurve\nabla_{h}\vbold_{h}+\nabla_{h}\vbold_{h}^{T}-\frac{2}{3} \left(\nabla_h \cdot\vbold_{h} \right) \mathbb{I} \rrcurve\nbold_F} + \ipbf{\frac{\eta}{h_{F}}\llbracket\vbold_{h} \rrbracket}{\llbracket\wbold_{h} \rrbracket}.
\end{align}
By setting $\vbold_{h}=\wbold_{h}$ in Eq.~\eqref{diff_bilinear_one}, we have
\begin{align}
 a_{h}(\wbold_{h},\wbold_{h}) 
& = \frac{1}{2} \left\| \nabla_{h}\wbold_{h}+\nabla_h \wbold_{h}^{T}-\frac{2}{3} \left( \nabla_h \cdot\wbold_{h} \right) \mathbb{I} \right\|_{\bm{L}^2 \left(\Omega \right) \times \bm{L}^2 \left(\Omega \right)}^2 + \frac{2 \left(3-d\right)}{9} \left\| \nabla_h \cdot \wbold_h \right\|_{L^2 \left(\Omega \right)}^2 \label{diff_bilinear_two} \\[1.5ex] 
\nonumber & - 2 \ipbf{\llbracket\wbold_{h}\rrbracket}{\llcurve\nabla_{h}\wbold_{h}+\nabla_{h}\wbold_{h}^{T}-\frac{2}{3} \left(\nabla_h \cdot\wbold_{h} \right) \mathbb{I} \rrcurve\nbold_F} + \eta \left| \wbold_{h} \right|_{J}^2,
\end{align}
where $\left| \wbold_{h} \right|_{J}$ is defined to be
\begin{align*}
\left| \wbold_{h} \right|_{J} = \left( \sum_{F \in \mathcal{F}_{h}}h_{F}^{-1} \left\| \llbracket \wbold_{h}\rrbracket| \right|^{2}_{\bm{L}^2 \left(F\right)} \right)^{1/2}.
\end{align*}
After examining Eq.~\eqref{diff_bilinear_two}, we find that the key is to bound the third term on the RHS. In what follows, we prove that for all $(\wbold_{h},\wbold_{h}) \in \bm{W}_{h} \times \bm{W}_{h}$, 
\begin{align}
& \left|  \ipbf{\llbracket\wbold_{h}\rrbracket}{\llcurve \nabla_{h}\wbold_{h}+\nabla_{h}\wbold_{h}^{T}-\frac{2}{3} \left(\nabla_h \cdot\wbold_{h} \right) \mathbb{I} \rrcurve\nbold_F} \right| \label{coerce_zero}  \\[1.5ex]
\nonumber & \leq \left(\sum_{K \in \mathcal{T}_h}\sum_{F \in \mathcal{F}_{K}} h_{F} \left\| \left( \nabla_{h}\wbold_{h}+\nabla_{h}\wbold_{h}^{T}-\frac{2}{3} \left(\nabla_h \cdot\wbold_{h} \right) \mathbb{I} \right)\bigg|_{K} \cdot\nbold_{F} \right\|^2_{\bm{L}^2(F)} \right)^{1/2} |\wbold_{h}|_{J}.
\end{align}
In accordance with a similar presentation in~\cite{DiPietro11}, we begin by identifying neighboring elements $K_1$ and $K_2$ that share a face $F = \partial K_1 \cap \partial K_2$. Then, we define $a_i=(\nabla_{h}\wbold_{h}+\nabla_{h}\wbold_{h}^{T}-\frac{2}{3} \left(\nabla_h \cdot\wbold_{h} \right)\mathbb{I})|_{K_{i}}\cdot\nbold_{F}$ and let $i =\{1,2\}$, in order to obtain
\begin{align*}
&\left\langle\llbracket\wbold_{h}\rrbracket,\llcurve\nabla_{h}\wbold_{h}+\nabla_{h}\wbold_{h}^{T}-\frac{2}{3} \left( \nabla_h \cdot\wbold_{h}\right) \mathbb{I} \rrcurve\nbold_F\right\rangle_{F} = \left\langle\frac{1}{2}(a_{1}+a_{2}),\llbracket\wbold_{h}\rrbracket\right\rangle_{F}\\[1.5ex]
\leq \nonumber & \left( \int_{F}  \frac{1}{4}(a_{1}+a_{2}) \cdot (a_{1}+a_{2}) \,  dA \right)^{1/2} \left( \int_{F}  \llbracket \wbold_h \rrbracket \cdot \llbracket \wbold_h \rrbracket  dA \right)^{1/2} \\[1.5ex]
\leq \nonumber & \frac{1}{2} \left( \left( \int_{F} a_1 \cdot a_1 \,  dA \right)^{1/2} + \left( \int_{F} a_2 \cdot a_2 \, dA \right)^{1/2} \right) \left( \int_{F} \llbracket \wbold_h \rrbracket \cdot \llbracket \wbold_h \rrbracket dA \right)^{1/2} \\[1.5ex]
\leq \nonumber & \frac{1}{2} \left( \left\| a_1 \right\|_{\bm{L}^2(F)} + \left\| a_2 \right\|_{\bm{L}^2(F)} \right) \left\| \llbracket \wbold_{h}\rrbracket \right\|_{\bm{L}^2(F)} \\[1.5ex]
\leq \nonumber &\left(\frac{1}{2}h_{F} \left( \left\| a_1 \right\|^2_{\bm{L}^2(F)}+ \left\| a_{2} \right\|^2_{\bm{L}^2(F)} \right) \right)^{1/2} h_{F}^{-1/2} \left\| \llbracket \wbold_{h}\rrbracket \right\|_{\bm{L}^2(F)},
\end{align*}
for all interior faces $\mathcal{F}_h^i$. Note that H\"{o}lder's inequality, the Triangle inequality, and the root-mean-square arithmetic-mean inequality have been used above. 

We can obtain a similar result for all boundary faces $\mathcal{F}_h^{\partial}$, 
\begin{align*}
& \left\langle\llbracket\wbold_{h}\rrbracket,\llcurve\nabla_{h}\wbold_{h}+\nabla_{h}\wbold_{h}^{T}-\frac{2}{3} \left(\nabla_h \cdot\wbold_{h} \right) \mathbb{I} \rrcurve\nbold_F\right\rangle_{F} \\[1.5ex]
&\nonumber \leq h_{F}^{1/2} \left\| \left( \nabla_{h}\wbold_{h}+\nabla_{h}\wbold_{h}^{T}-\frac{2}{3} \left(\nabla_h \cdot\wbold_{h} \right) \mathbb{I} \right)\bigg|_{K}\cdot\nbold_{F} \right\|_{\bm{L}^2(F)}h_{F}^{-1/2} \left\| \llbracket \wbold_{h}\rrbracket\right\|_{\bm{L}^2(F)}.
\end{align*}
Summing over all mesh faces, using the Cauchy-Schwarz inequality, and regrouping the face contributions for each element, we get the desired result in Eq.~\eqref{coerce_zero}.

Now, we can continue by rewriting the RHS of Eq.~\eqref{coerce_zero} as follows
\begin{align}
& \sum_{K \in \mathcal{T}_h}\sum_{F \in \mathcal{F}_{K}} h_F \left\| \left(\nabla_{h}\wbold_{h}+\nabla_{h}\wbold_{h}^{T}-\frac{2}{3} \left(\nabla_h \cdot\wbold_{h} \right) \mathbb{I} \right)\bigg|_{K}\cdot \nbold_{F} \right\|^2_{\bm{L}^2(F)} \label{diff_trace} \\[1.5ex]
\nonumber &\leq\sum_{K \in \mathcal{T}_{h}} h_{K} \left\| \left( \nabla_{h}\wbold_{h}+\nabla_h \wbold_{h}^{T}-\frac{2}{3} \left( \nabla_h \cdot\wbold_{h} \right) \mathbb{I} \right) \bigg|_{K}\cdot \nbold_{F} \right\|^2_{\bm{L}^2(\partial K)} \\[1.5ex]
\nonumber &\leq C_{\text{tr}}^2N_{\partial} \left\| \nabla_{h}\wbold_{h}+\nabla_h \wbold_{h}^{T}-\frac{2}{3} \left( \nabla_h \cdot\wbold_{h} \right) \mathbb{I} \right\|^2_{\bm{L}^2 \left(\Omega \right) \times \bm{L}^2 \left(\Omega \right)},
\end{align}
where a discrete trace inequality has been used in the last line (cf.~\cite{DiPietro11}, p. 27). Here, the constant $C_{\text{tr}}$ depends on the shape parameters of elements in the mesh, and $N_{\partial}$ characterizes the number of boundary faces. The result in Eq.~\eqref{diff_trace} leads to 
\begin{align}
& \left|  \ipbf{\llbracket\wbold_{h}\rrbracket}{\llcurve \nabla_{h}\wbold_{h}+\nabla_{h}\wbold_{h}^{T}-\frac{2}{3} \left(\nabla_h \cdot\wbold_{h} \right) \mathbb{I} \rrcurve\nbold_F} \right| \label{coerce_two} \\[1.5ex]
\nonumber &\leq C_{\text{tr}} N_{\partial}^{1/2} \left\| \nabla_{h}\wbold_{h}+\nabla_h \wbold_{h}^{T}-\frac{2}{3} \left( \nabla_h \cdot\wbold_{h} \right) \mathbb{I} \right\|_{\bm{L}^2 \left(\Omega \right) \times \bm{L}^2 \left(\Omega \right)} \left|\wbold_{h} \right|_{J}.
\end{align}
Next, in preparation for the final coercivity result, we will define
\begin{align*}
a &= \left\| \nabla_{h}\wbold_{h}+\nabla_h \wbold_{h}^{T}-\frac{2}{3} \left( \nabla_h \cdot\wbold_{h} \right) \mathbb{I} \right\|_{\bm{L}^2 \left(\Omega \right) \times \bm{L}^2 \left(\Omega \right)},  \\[1.5ex]
b &=\ipbf{\frac{1}{h_{F}}\llbracket\wbold_{h}\rrbracket}{\llbracket\wbold_{h}\rrbracket}^{1/2}, \qquad c = \left\| \nabla_h \cdot \wbold_h \right\|_{L^2 \left(\Omega \right)}.
\end{align*}
Now, combining these definitions with Eqs.~\eqref{diff_bilinear_two} and \eqref{coerce_two} we have
\begin{align}
a_{h}(\wbold_{h},\wbold_{h}) &\geq \frac{1}{2}(a^2-4C_{\text{tr}}N_{\partial}^{1/2} ab+2\eta b^2) + \frac{2\left(3-d\right)}{9} c^2 \label{coercive_ineq} \\[1.5ex]
\nonumber & \geq \frac{1}{2}(a^2-4C_{\text{tr}}N_{\partial}^{1/2} ab+2\eta b^2)\geq \left(\frac{\eta-2C_{\text{tr}}^2N_{\partial}}{1+2\eta} \right)(a^2+b^2),
\end{align}
where we have assumed that $d\leq 3$ in the first line. If we use the definition of the norm $\vertiii{ \cdot }_{\text{sym}}$, we get the following condition on the coercivity of the bilinear form 
\begin{align*}
& a_{h}(\wbold_{h},\wbold_{h}) \\[1.5ex]
\nonumber & \geq \left( \frac{\eta-2C_{\text{tr}}^2N_{\partial}}{1+2\eta} \right) \left(\left\| \nabla_{h}\wbold_{h}+\nabla_h \wbold_{h}^{T}-\frac{2}{3} \left( \nabla_h \cdot\wbold_{h} \right) \mathbb{I} \right\|_{\bm{L}^2 \left(\Omega \right) \times \bm{L}^2 \left(\Omega \right)}^2 + \ipbf{\frac{1}{h_{F}}\llbracket\wbold_{h}\rrbracket}{\llbracket\wbold_{h}\rrbracket}  \right) \\[1.5ex]
& \nonumber \geq \left(\frac{\eta-2C_{\text{tr}}^2N_{\partial}}{1+2\eta} \right) \vertiii{ \wbold_h }_{\text{sym}}^2. 
\end{align*}
%
%
Upon setting $C = \left(\eta-2C_{\text{tr}}^2N_{\partial} \right)/\left(1+2\eta \right)$, we obtain the desired result (see Eq.~\eqref{coercive_res}).
\end{proof}

\begin{remark}
In principle, Lemma~\ref{coercive_vis_lemma} holds for a much broader class of schemes, and is not merely limited to the mixed methods of section~\ref{gen_method_sec}. For example, it remains valid for general discontinuous Galerkin schemes. 
\end{remark}

Now, having established Lemma~\ref{coercive_vis_lemma}, we will construct a relatively straightforward corollary which establishes the coercivity of $a_h$ in a slightly stronger norm (at least for the 2D case). This result is not used in the present paper, but we provide it for the sake of completeness, and to help facilitate future analysis. 

\begin{corollary}
    Under the assumptions of Lemma~\ref{coercive_vis_lemma}, the bilinear form $a_h$ in Eq.~\eqref{diff_bilinear} is coercive on $\bm{W}_h$, such that
    \begin{align}
        \forall \wbold_h \in \bm{W}_h,\qquad a_h \left( \wbold_h, \wbold_h \right) \geq C \vertiiisup{\wbold_h}_{\text{sym}}^2, \label{coercive_res_alt}
    \end{align}
    where $C$ is a positive constant independent of $h$, and where
    \begin{align*}
        \vertiiisup{ \wbold_h }_{\text{sym}} = \left( \vertiii{ \wbold_h }_{\text{sym}}^2 + \frac{4}{9} \left(3 -d\right) \left\| \nabla_h \cdot \wbold_h \right\|_{L^2 \left(\Omega \right)}^2 \right)^{1/2},
    \end{align*}
    is a norm on $\Omega$. The present corollary is identical to Lemma~\ref{coercive_vis_lemma} when $d =3$, and differs from it when $d = 2$.
\end{corollary}

\begin{proof}
One may begin the proof by rewriting Eq.~\eqref{coercive_ineq} of Lemma~\ref{coercive_vis_lemma}, as follows
\begin{align*}
    a_{h}(\wbold_{h},\wbold_{h}) &\geq \frac{1}{2} ( a^2-4C_{\text{tr}}N_{\partial}^{1/2} ab+2\eta b^2) + \frac{2}{9} \left(3 -d\right) c^2 \\[1.5ex]
    & \geq  \left(\frac{\eta-2C_{\text{tr}}^2N_{\partial}}{1+2\eta} \right)(a^2+b^2) + \frac{1}{2} \cdot \frac{4}{9} \left(3 -d\right) c^2 \\[1.5ex]
    & \geq \left(\frac{\eta-2C_{\text{tr}}^2N_{\partial}}{1+2\eta} \right) \left(a^2 + b^2 + \frac{4}{9} \left(3 -d\right) c^2 \right).
\end{align*}
Here, we have used the fact that $\left(\eta-2C_{\text{tr}}^2N_{\partial} \right)/\left(1+2\eta \right)$ resides on the interval $\left(0, \frac{1}{2} \right)$ for all $\eta > 2C_{\text{tr}}^2N_{\partial}$. Upon rewriting the expression above, we find that
\begin{align*}
    a_{h}(\wbold_{h},\wbold_{h}) &\geq \left(\frac{\eta-2C_{\text{tr}}^2N_{\partial}}{1+2\eta} \right) \left( \vertiii{ \wbold_h }_{\text{sym}}^2 + \frac{4}{9} \left(3 -d\right) \left\| \nabla_h \cdot \wbold_h \right\|_{L^2 \left(\Omega \right)}^2 \right) \\[1.5ex]
    & \geq \left(\frac{\eta-2C_{\text{tr}}^2N_{\partial}}{1+2\eta} \right) \vertiiisup{ \wbold_h }_{\text{sym}}^2,
\end{align*}
where we have used the definition of $\vertiiisup{\cdot}_{\text{sym}}$. Upon setting $C = \left(\eta-2C_{\text{tr}}^2N_{\partial} \right)/\left(1+2\eta \right)$, we obtain the desired result (see Eq.~\eqref{coercive_res_alt}).
\end{proof}

\begin{lemma}[Semi-Coercivity of the Convective Trilinear Form]
Consider test functions $\bm{\beta}_h, \wbold_h \in \bm{W}_h$. Then, the trilinear form $c_h$ in Eq.~\eqref{trilinear} is semi-coercive on $\bm{W}_h$, such that
\begin{align}
    \forall \left(\bm{\beta}_h, \wbold_h \right) \in \bm{W}_h \times \bm{W}_h, \qquad c_h \left(\bm{\beta}_h; \wbold_h, \wbold_h \right) = \left| \wbold_h \right|_{\bm{\beta}_h}^2,
\end{align}
where
\begin{align}
    \left| \wbold_h \right|_{\bm{\beta}_h} = \left( \zeta \iipbf{\left| \bm{\beta}_h \cdot \nbold_F \right|\llbracket \wbold_h \rrbracket}{\llbracket \wbold_h \rrbracket} \right)^{1/2},
\end{align}
is a seminorm on $\Omega$.
\label{coercive_tri_lemma}
\end{lemma}

\begin{proof}
We begin by substituting $\vbold_h = \wbold_h$ into Eq.~\eqref{trilinear} in order to obtain the following
\begin{align}
c_h \left(\bm{\beta}_h; \wbold_h, \wbold_h \right) &= \ipt{\bm{\beta}_h \cdot \nabla_h \wbold_h}{\wbold_h} + \frac{1}{2} \ipt{\left(\nabla_h \cdot \bm{\beta}_h \right) \wbold_h}{\wbold_h} \label{conv_coercive_one}\\[1.5ex]
& \nonumber - \iipbf{ \left(\bm{\beta}_h \cdot \nbold_F \right) \llbracket \wbold_h \rrbracket}{\llcurve \wbold_h \rrcurve} + \zeta \iipbf{\left| \bm{\beta}_h \cdot \nbold_F \right|\llbracket \wbold_h \rrbracket}{\llbracket \wbold_h \rrbracket}.
\end{align}
Next, we can substitute Eq.~\eqref{temam_id_spec} from the lemma in~\ref{jump_lemma} into Eq.~\eqref{conv_coercive_one}
\begin{align}
c_h \left(\bm{\beta}_h; \wbold_h, \wbold_h \right) &= \zeta \iipbf{\left| \bm{\beta}_h \cdot \nbold_F \right|\llbracket \wbold_h \rrbracket}{\llbracket \wbold_h \rrbracket}. \label{conv_coercive_three}
\end{align}
The proof of semi-coercivity is completed by substituting the definition of the appropriate seminorm into the RHS of Eq.~\eqref{conv_coercive_three}.
\end{proof}

\begin{remark}
    One should note that Lemma~\ref{coercive_tri_lemma} can be generalized to non-H(div) conforming methods, such as discontinuous Galerkin methods. This can be achieved by replacing the boundary terms in $c_h$ with 
    \begin{align*}
       &- \iipbf{ \left( \llcurve \bm{\beta}_h \rrcurve \cdot \nbold_F \right) \llbracket \vbold_h \rrbracket}{\llcurve \wbold_h \rrcurve} + \zeta \iipbf{\left| \llcurve \bm{\beta}_h \rrcurve \cdot \nbold_F \right|\llbracket \vbold_h \rrbracket}{\llbracket \wbold_h \rrbracket} \\[1.5ex]
       & - \frac{1}{2} \ipbf{\llbracket \bm{\beta}_h \rrbracket \cdot \nbold_F}{\llcurve \vbold_h \cdot \wbold_h \rrcurve}.
    \end{align*}
    The corresponding proof of semi-coercivity is discussed in~\cite{DiPietro11}, p.~272.
\end{remark}


\section{Stability: General Case}\label{stability}

\begin{theorem}[Kinetic Energy Estimate] Consider a forcing function $\widetilde{\bm{f}} \in L^1 \left(t_0, t_n, \bm{L}^2 \left(\Omega \right)\right)$ and an initial condition $\ubold_h \left(t_0 \right) \in \bm{W}_h \subset \bm{H}_{0}(\text{div};\Omega)$. Subject to these assumptions, the discrete kinetic energy of the general mixed methods (formulated in Eqs.~\eqref{mass_cons_disc} and \eqref{moment_cons_disc}) is governed by the following equation at time $t_n \geq t_0$
\begin{align}
        &\frac{1}{2} \left\| \ubold_h \left(t_n \right) \right\|_{\bm{L}^2 \left(\Omega\right)}^2 + \int_{t_0}^{t_n} \left( \left| \ubold_h \left(s \right) \right|_{\ubold_h}^2 +\nu_h \, C \left\| \ubold_h \left(s \right) \right\|_{\text{sym}}^2  \right) ds \label{kinetic_res} \\[1.5ex]
        \nonumber & \leq  \left\| \ubold_h \left(t_0 \right) \right\|_{\bm{L}^2 \left(\Omega\right)}^2 + \frac{3}{2} \left\| \widetilde{\bm{f}} \right\|_{L^1 \left(t_0, t_n, \bm{L}^2 \left(\Omega\right) \right)}^2, 
\end{align}
where 
\begin{align*}
    \left\| \bm{f} \right\|_{L^1 \left(t_0, t_n, \bm{L}^2 \left(\Omega\right) \right)} = \int_{t_0}^{t_n} \left\| \bm{f} \left(s\right) \right\|_{\bm{L}^2 \left(\Omega\right)} ds,
\end{align*}
is a space-time norm on $\left(t_0, t_n\right) \times \Omega$.

\end{theorem}

\begin{proof}
We begin by substituting $q_h = \widetilde{p}_h$ and $\wbold_h = \ubold_h$ into Eq.~\eqref{incomp_form_one}, in order to obtain
\begin{align*}
& b_h \left(\ubold_h, \widetilde{p}_h \right) = 0,  \\[1.5ex]
\nonumber & \ipt{\partial_t \, \ubold_h}{\ubold_h} + c_h \left(\ubold_h; \ubold_h, \ubold_h \right) +\nu_h a_h \left(\ubold_h, \ubold_h \right)  - b_h \left( \ubold_h, \widetilde{p}_h \right) = \ipt{\widetilde{\bm{f}}}{\ubold_h}.  
\end{align*}
Upon summing these equations together we obtain
\begin{align*}
\ipt{\partial_t \, \ubold_h}{\ubold_h} + c_h \left(\ubold_h; \ubold_h, \ubold_h \right) +\nu_h a_h \left(\ubold_h, \ubold_h \right) = \ipt{\widetilde{\bm{f}}}{\ubold_h}, 
\end{align*}
or equivalently
\begin{align*}
   \frac{1}{2} \frac{d}{dt} \left\| \ubold_h \right\|_{\bm{L}^2 \left(\Omega\right)}^2 + c_h \left(\ubold_h; \ubold_h, \ubold_h \right) +\nu_h a_h \left(\ubold_h, \ubold_h \right) = \ipt{\widetilde{\bm{f}}}{\ubold_h}.
\end{align*}
Next, we invoke the coercivity of $a_h$ (Lemma~\ref{coercive_vis_lemma}) and the semi-coercivity of $c_h$ (Lemma~\ref{coercive_tri_lemma}) in order to obtain
\begin{align}
    \frac{1}{2} \frac{d}{dt} \left\| \ubold_h \right\|_{\bm{L}^2 \left(\Omega\right)}^2 + \left| \ubold_h \right|_{\ubold_h}^2 +\nu_h \, C \left\| \ubold_h \right\|_{\text{sym}}^2  \leq \ipt{\widetilde{\bm{f}}}{\ubold_h}. \label{kbound_two}
\end{align}
In accordance with the approach of~\cite{Dallmann15}, we examine Eq.~\eqref{kbound_two} and note that 
\begin{align*}
       \frac{1}{2} \frac{d}{dt} \left\| \ubold_h \right\|_{\bm{L}^2 \left(\Omega\right)}^2  \leq \ipt{\widetilde{\bm{f}}}{\ubold_h}, 
\end{align*}
and by the Cauchy-Schwarz inequality
\begin{align}
    \nonumber \left\| \ubold_h \right\|_{\bm{L}^2 \left(\Omega\right)} \frac{d}{dt} \left\| \ubold_h \right\|_{\bm{L}^2 \left(\Omega\right)} & \leq \left\| \widetilde{\bm{f}} \right\|_{\bm{L}^2 \left(\Omega\right)} \left\| \ubold_h \right\|_{\bm{L}^2 \left(\Omega\right)} \\[1.5ex]
    \frac{d}{dt} \left\| \ubold_h \right\|_{\bm{L}^2 \left(\Omega\right)} & \leq \left\| \widetilde{\bm{f}} \right\|_{\bm{L}^2 \left(\Omega\right)}. \label{kbound_four}
\end{align}
Now, we integrate Eq.~\eqref{kbound_four} from $t = t_0$ to $t = t_n$ in order to obtain
\begin{align*}
    \left\| \ubold_h \left(t_n \right) \right\|_{\bm{L}^2 \left(\Omega\right)} -     \left\| \ubold_h \left(t_0 \right) \right\|_{\bm{L}^2 \left(\Omega\right)} \leq \int_{t_0}^{t_n} \left\| \widetilde{\bm{f}} \left( s\right) \right\|_{\bm{L}^2 \left(\Omega\right)} ds,
\end{align*}
or equivalently
\begin{align}
    \left\| \ubold_h \left(t_n \right) \right\|_{\bm{L}^2 \left(\Omega\right)} \leq \left\| \ubold_h \left(t_0 \right) \right\|_{\bm{L}^2 \left(\Omega\right)} + \left\| \widetilde{\bm{f}} \right\|_{L^1 \left(t_0, t_n, \bm{L}^2 \left(\Omega\right) \right)}. \label{kbound_five}
\end{align}
%
%
Setting aside Eq.~\eqref{kbound_five} for the moment, we return our attention to Eq.~\eqref{kbound_two}. Upon integrating Eq.~\eqref{kbound_two} from $t = t_0$ to $t = t_n$ we obtain
\begin{align}
    & \frac{1}{2} \left\| \ubold_h \left(t_n \right) \right\|_{\bm{L}^2 \left(\Omega\right)}^2 + \int_{t_0}^{t_n} \left( \left| \ubold_h \left(s \right) \right|_{\ubold_h}^2 +\nu_h \, C \left\| \ubold_h \left(s \right) \right\|_{\text{sym}}^2  \right) ds \label{kbound_six} \\[1.5ex]
    \nonumber & \leq \frac{1}{2} \left\| \ubold_h \left(t_0 \right) \right\|_{\bm{L}^2 \left(\Omega\right)}^2 + \int_{t_0}^{t_n} \ipt{\widetilde{\bm{f}} \left(s \right) }{\ubold_h \left(s \right)} ds. 
\end{align}
The last term on the RHS of Eq.~\eqref{kbound_six} can be rewritten by applying the Cauchy-Schwarz inequality, Eq.~\eqref{kbound_five}, and Young's inequality as follows
\begin{align}
\nonumber  \int_{t_0}^{t_n} \ipt{\widetilde{\bm{f}} \left(s \right) }{\ubold_h \left(s \right)} ds & \leq \int_{t_0}^{t_n} \left[ \left\| \widetilde{\bm{f}} \left(s \right) \right\|_{\bm{L}^2 \left(\Omega\right)} \left\| \ubold_h \left(s \right) \right\|_{\bm{L}^2 \left(\Omega\right)} \right] ds  \\[1.5ex]
\nonumber  & \leq \int_{t_0}^{t_n} \left[ \left\| \widetilde{\bm{f}} \left(s \right) \right\|_{\bm{L}^2 \left(\Omega\right)} \left( \left\| \ubold_h \left(t_0 \right) \right\|_{\bm{L}^2 \left(\Omega\right)} + \left\| \widetilde{\bm{f}} \right\|_{L^1 \left(t_0, s, \bm{L}^2 \left(\Omega\right) \right)} \right) \right] ds \\[1.5ex]
\nonumber & \leq  \int_{t_0}^{t_n} \left\| \widetilde{\bm{f}} \left(s \right) \right\|_{\bm{L}^2 \left(\Omega\right)} ds  \left( \left\| \ubold_h \left(t_0 \right) \right\|_{\bm{L}^2 \left(\Omega\right)} + \left\| \widetilde{\bm{f}} \right\|_{L^1 \left(t_0, t_n, \bm{L}^2 \left(\Omega\right) \right)} \right) \\[1.5ex]
\nonumber & = \left\| \widetilde{\bm{f}} \right\|_{L^1 \left(t_0, t_n, \bm{L}^2 \left(\Omega\right) \right)} \left( \left\| \ubold_h \left(t_0 \right) \right\|_{\bm{L}^2 \left(\Omega\right)} + \left\| \widetilde{\bm{f}} \right\|_{L^1 \left(t_0, t_n, \bm{L}^2 \left(\Omega\right) \right)} \right) \\[1.5ex]
& \leq \frac{1}{2} \left\| \ubold_h \left(t_0 \right) \right\|_{\bm{L}^2 \left(\Omega\right)}^2 + \frac{3}{2} \left\| \widetilde{\bm{f}} \right\|_{L^1 \left(t_0, t_n, \bm{L}^2 \left(\Omega\right) \right)}^2. \label{kbound_seven}
\end{align}
Upon combining Eqs.~\eqref{kbound_seven} and \eqref{kbound_six}, we obtain the desired bound on the discrete kinetic energy, (see Eq.~\eqref{kinetic_res}).
\end{proof}

\section{Numerical Experiments}\label{numerical_simulation}
In this section, we will present some numerical simulations to demonstrate the performance of the newly proposed, versatile mixed methods. We considered four mixed methods: i) a Taylor-Hood method based on the symmetric tensor formulation in Eqs.~\eqref{mass_cons_disc} and \eqref{moment_cons_disc}, with continuous Lagrangian elements for the velocity and pressure (denoted by TH-Symmetric); ii) a Taylor-Hood method based on the classical, non-symmetric tensor formulation (denoted by TH-Non-Symmetric), iii) a pointwise divergence-free, H(div)-conforming method based on the symmetric tensor formulation in Eqs.~\eqref{mass_cons_disc_div} and \eqref{moment_cons_disc_div},  with Brezzi-Douglas-Marini elements for the velocity and discontinuous Lagrangian elements for the pressure  (denoted by BDM-Symmetric); and iv) a pointwise divergence-free, H(div)-conforming, BDM method based on the classical, non-symmetric tensor formulation (denoted by BDM-Non-Symmetric) which was reviewed by~\cite{schroeder2018towards}. In each case, we considered polynomial degree $k \in\{1,2,3\}$, (omitting the case of $k=0$), and we imposed a zero integral mean condition for the pressure via a Lagrange multiplier. In addition, the computational domains were meshed with squares, and each square was split into two congruent triangles. The high-order BDF3 scheme was used for the time discretization.  The numerical results presented below were obtained using FEniCS, which is an open-source software package for solving partial differential equations, (see~\cite{logg2012automated,alnaes2015fenics} for details). 

For the first example, we follow \cite{guzman2016h,schroeder2018divergence} and consider the 2D Taylor-Green vortex case with periodic boundary conditions on all sides. The exact solution is given by
\begin{align}
 \nonumber   \ubold(t,\xbold)&=\bigg(\sin(x_{1})\cos(x_{2})e^{-2\nu t},-\cos(x_{1})\sin(x_{2})e^{-2\nu t}\bigg), \\[1.5ex]
    \nonumber \widetilde{p}(t,\xbold)&=\frac{1}{4}\bigg(\cos(2x_{1})+\cos(2x_{2})\bigg)e^{-4\nu t},
\end{align}
with domain $\Omega:=[0,2\pi]^2$, $\xbold:=(x_1,x_2)\in\Omega$, $t\in[0,1]$, and viscosity $\nu=0.01$. The Taylor-Green vortex can be simulated for much longer times $t \gg 1$ (see~\cite{schroeder2018divergence}), but our objective was to limit the total runtime in order to help control the temporal errors. This was necessary for calculating the spatial orders of accuracy.  In each of our simulations, we computed the approximations of $\ubold$ and $\widetilde{p}$ at $t=1.0$ s, with time-step $\Delta t=0.01$ s, on four uniform meshes ($h = 0.8886$, 0.4443, 0.2221, and 0.1777). For the TH-Symmetric and TH-Non-Symmetric schemes, we performed tests with and without a stabilization term of the form
\begin{align*}
\delta \ipt{|\ubold_h| \nabla\cdot\ubold_h}{\nabla\cdot\wbold_h}.
\end{align*}
For the tests with non-zero stabilization, we set $\delta = 10$. For the BDM-Symmetric and BDM-Non-Symmetric schemes, we computed the numerical fluxes with $\zeta=0.5$ for an upwind biased convective flux, $\zeta=0$ for a central unbiased convective flux, and $\eta=3(k+1)(k+2)$ for a viscous flux with extra dissipation that scales with $k^2$.

In table \ref{table_TH}, we present the results of the TH-Symmetric method with and without stabilization. We observe that, in almost all cases, the velocity and pressure errors converge at the expected rates. However, when $k = 1$, we obtain super-convergence of the velocity for the TH-Symmetric method without the stabilization term. This is an exceptional case, and as expected, the method recovers the standard rates of convergence for $k = 2$ and 3. The TH-Symmetric method with stabilization does not demonstrate superconvergence (even for $k=1$), but instead achieves the expected rate of convergence in all cases. Overall, despite the superconvergence results for $k =1$, the TH-Symmetric method with stabilization tends to outperform the method without stabilization, as it produces lower absolute values of error in most cases. 

In table \ref{table_TH_nonsymmetric}, we present the results of the TH-Non-Symmetric method with and without stabilization. The results are very similar to those of the TH-Symmetric method in table~\ref{table_TH}. Interestingly enough, we observe slightly lower absolute errors for both velocity and pressure with the symmetric formulation, especially when the stabilization term is omitted ($\delta=0$).  

In table \ref{table_our_up_cen}, we present a comparison of the upwind and central versions of the BDM-Symmetric method. For both versions of the method, we observe the expected convergence rates for velocity and pressure. We obtain similar results, upon comparing the upwind and central versions of the BDM-Non-Symmetric method in table~\ref{table_our_up_comparison}. 

Now, when we compare tables \ref{table_our_up_cen} and \ref{table_our_up_comparison} to one another, we see that the BDM-Symmetric and BDM-Non-Symmetric methods achieve roughly the same absolute levels of accuracy. Of course, this holds with the caveat that the BDM-Symmetric method is significantly more versatile than the BDM-Non-Symmetric method (as previously discussed). Finally, when we compare tables~\ref{table_TH} and \ref{table_TH_nonsymmetric} with tables \ref{table_our_up_cen} and \ref{table_our_up_comparison}, we observe that the pointwise divergence-free, H(div)-conforming methods behave better than the Taylor-Hood methods. This is not a surprise, as the latter methods are neither pointwise divergence-free, nor pressure-robust.

We conclude our discussion of example 1, by examining the qualitative behaviour of the vorticity and pressure for the BDM-Symmetric method as illustrated in figures~\ref{fig:vorticity_k1}--\ref{fig:pressure_k3}. Here, we observe that the approximations of vorticity and pressure become significantly more accurate when $k$ increases, and similarly, when $h$ decreases.



In example 2, we compare the behavior of the upwind and central versions of the BDM-Symmetric method on the Gresho-vortex problem (see \cite{schroeder2017pressure}). This problem is defined on a domain $\Omega=(-0.5,0.5)^2$.
The initial state of the fluid is described by 
\begin{align*}
\ubold_{\phi}(r,\phi) &=\begin{cases}
  5r   &\text{$0\leq r\leq 0.2$}  \\[1.5ex]   
  2-5r &\text{$0\leq r\leq 0.4$}  \\[1.5ex] 
  0   &\text{$0.4\leq r$} 
\end{cases} \\[1.5ex]
\ubold_{r}(r,\phi) &=0.
\end{align*}
Note here that polar coordinates $\left(r, \phi \right)$ are used. For this problem, we set $\nu=5\times 10^{-6}$, and performed simulations with $k \in \left\{1, 2 ,3 \right\}$ and $h=0.0354$. 
The total simulation time was set to be 14.0 s and the time-step was $\Delta t = 0.01$ s. 
Velocity magnitude contours for this example are shown in figures~\ref{fig:v_t4_k1}--\ref{fig:v_t4_k3}. We observe that, in general, the upwind flux behaves better than the central flux, especially when $k=1$, which implies the superiority of the upwind flux compared with the central flux when one considers convection-dominated flows.

%
\begin{center}
\makeatletter\def\@captype{table}\makeatother
\resizebox{\textwidth}{33mm}{
\begin{tabular}{|p{0.1cm}|c|c|l|l|l|l|}
\hline
\multicolumn{1}{|c|}{\multirow{3}{*}{$k$}} & \multirow{3}{*}{$h$} & \multirow{3}{*}{d.o.f} & \multicolumn{2}{c|}{$\delta=0$}                                                                                                     & \multicolumn{2}{c|}{$\delta=10$}                                                                                                   \\ \cline{4-7} 
\multicolumn{1}{|c|}{}                   &                    &                        & \begin{tabular}[c]{@{}l@{}}$\|\ubold-\ubold_{h}\|_{\bm{L}^{2}(\Omega)}$ \\ error \hspace{0.65cm} order\end{tabular} & \begin{tabular}[c]{@{}l@{}}$\|\widetilde{p}-\widetilde{p}_{h}\|_{L^{2}(\Omega)}$\\ error \hspace{0.65cm} order\end{tabular} & \begin{tabular}[c]{@{}l@{}}$\|\ubold-\ubold_{h}\|_{\bm{L}^{2}(\Omega)}$\\ error \hspace{0.65cm} order\end{tabular} & \begin{tabular}[c]{@{}l@{}}$\|\widetilde{p}-\widetilde{p}_{h}\|_{L^{2}(\Omega)}$\\ error \hspace{0.65cm} order\end{tabular} \\ \hline
\multirow{4}{*}{1} & 0.8886 & 901 & 2.85e-1   \quad    -----                   & 1.51e-1  \quad  ----- & 1.53e-1  \quad  -----& 1.07e-1  \quad  -----         \\ \cline{2-7} 
& 0.4443 & 3601 & 2.54e-2   \quad 3.49  & 2.36e-2  \quad  2.68                  & 1.87e-2  \quad  3.04 & 2.34e-2  \quad    2.20                                                  \\ \cline{2-7} 
& 0.2221 & 14401 & 1.54e-3  \quad  4.05                                     
& 5.62e-3  \quad    2.07 & 2.52e-3  \quad  2.89                                 & 5.64e-3  \quad  2.05                                                  \\ \cline{2-7} 
& 0.1777 & 22501 & 6.39e-4  \quad 3.93  & 3.58e-3  \quad  2.02                  & 1.26e-3  \quad  3.10  & 3.59e-3  \quad  2.03                                                  \\ \hline
\multirow{4}{*}{2} & 0.8886 & 2201 & 4.96e-2    \quad    -----                  & 2.57e-2  \quad  ----- & 1.18e-2  \quad  ----- & 1.42e-2  \quad  -----         \\ \cline{2-7} 
& 0.4443 & 8801 & 5.68e-3  \quad  3.13 & 3.33e-3  \quad  2.95                  & 8.71e-4  \quad  3.75 & 2.05e-3  \quad   2.80                                                  \\ \cline{2-7} 
& 0.2221 & 35201 & 3.79e-4  \quad 3.90 & 3.53e-4  \quad    3.24      
& 4.70e-5  \quad  4.21  & 2.66e-4  \quad  2.94                                                  \\ \cline{2-7} 
& 0.1777 & 55001 & 1.54e-4  \quad  4.05                                     
& 1.69e-4  \quad    3.30    & 1.78e-5  \quad  4.35                              & 1.37e-4  \quad  2.98                                                  \\ \hline
\multirow{4}{*}{3} & 0.8886 & 4101 & 1.22e-3    \quad    ----- 
& 8.98e-4  \quad  ----- & 1.94e-4  \quad  -----                                 & 8.05e-4  \quad  -----                                                 \\ \cline{2-7} 
& 0.4443 & 16401 & 2.33e-5  \quad 5.71  & 4.66e-5  \quad  4.27                  & 4.94e-6  \quad  5.30  & 4.63e-5  \quad  4.12                                                  \\ \cline{2-7} 
& 0.2221 & 65601 & 6.60e-7  \quad      5.14     & 2.84e-6  \quad  4.04          & 1.44e-7  \quad  5.10  & 2.83e-6 \quad   4.03                                                  \\ \cline{2-7} 
& 0.1777 & 102501 & 2.16e-7     \quad   5.00 & 1.16e-6  \quad  4.01             & 4.69e-8  \quad  5.04  & 1.16e-6  \quad  4.01                                                  \\ \hline
\end{tabular}}
\caption{A comparison of the TH-Symmetric method without stabilization ($\delta=0$), to the same method with stabilization ($\delta=10$) at $t = 1.0$. The stabilization term takes the form $\delta|\ubold_h| \left( \nabla\cdot\ubold_h \right) \left(\nabla\cdot\wbold_h \right)$.}
\label{table_TH}
\end{center}

\begin{center}
\makeatletter\def\@captype{table}\makeatother
\resizebox{\textwidth}{33mm}{
\begin{tabular}{|p{0.1cm}|c|c|l|l|l|l|}
\hline
\multicolumn{1}{|c|}{\multirow{3}{*}{$k$}} & \multirow{3}{*}{$h$} & \multirow{3}{*}{d.o.f} & \multicolumn{2}{c|}{$\delta=0$}                                                                                                     & \multicolumn{2}{c|}{$\delta=10$}                                                                                                   \\ \cline{4-7} 
\multicolumn{1}{|c|}{}                   &                    &                        & \begin{tabular}[c]{@{}l@{}}$\|\ubold-\ubold_{h}\|_{\bm{L}^{2}(\Omega)}$ \\ error \hspace{0.65cm} order\end{tabular} & \begin{tabular}[c]{@{}l@{}}$\|\widetilde{p}-\widetilde{p}_{h}\|_{L^{2}(\Omega)}$\\ error \hspace{0.65cm} order\end{tabular} & \begin{tabular}[c]{@{}l@{}}$\|\ubold-\ubold_{h}\|_{\bm{L}^{2}(\Omega)}$\\ error \hspace{0.65cm} order\end{tabular} & \begin{tabular}[c]{@{}l@{}}$\|\widetilde{p}-\widetilde{p}_{h}\|_{L^{2}(\Omega)}$\\ error \hspace{0.65cm} order\end{tabular} \\ \hline
\multirow{4}{*}{1} & 0.8886 & 901 & 3.09e-1   \quad    -----                   & 1.58e-1  \quad  ----- & 1.53e-1  \quad  -----& 1.07e-1  \quad  -----         \\ \cline{2-7} 
& 0.4443 & 3601 & 3.21e-2   \quad 3.27  & 2.40e-2  \quad  2.72                  & 1.87e-2  \quad  3.04 & 2.34e-2  \quad    2.20                                                  \\ \cline{2-7} 
& 0.2221 & 14401 & 1.98e-3  \quad  4.02                                     
& 5.62e-3  \quad    2.09 & 2.52e-3  \quad  2.89                                 & 5.64e-3  \quad  2.05                                                  \\ \cline{2-7} 
& 0.1777 & 22501 & 8.26e-4  \quad 3.91  & 3.58e-3  \quad  2.02                  & 1.26e-3  \quad  3.10  & 3.59e-3  \quad  2.03                                                  \\ \hline
\multirow{4}{*}{2} & 0.8886 & 2201 & 5.82e-2    \quad    -----                  & 2.91e-2  \quad  ----- & 1.18e-2  \quad  ----- & 1.42e-2  \quad  -----         \\ \cline{2-7} 
& 0.4443 & 8801 & 7.13e-3  \quad  3.03 & 3.85e-3  \quad  2.92                  & 8.70e-4  \quad  3.75 & 2.05e-3  \quad   2.80                                                  \\ \cline{2-7} 
& 0.2221 & 35201 & 4.98e-4  \quad 3.84 & 4.01e-4  \quad    3.26      
& 4.70e-5  \quad  4.21  & 2.66e-4  \quad  2.94                                                  \\ \cline{2-7} 
& 0.1777 & 55001 & 2.03e-4  \quad  4.02                                     
& 1.89e-4  \quad    3.37    & 1.78e-5  \quad  4.35                              & 1.37e-4  \quad  2.98                                                  \\ \hline
\multirow{4}{*}{3} & 0.8886 & 4101 & 1.47e-3    \quad    ----- 
& 9.50e-4  \quad  ----- & 1.94e-4  \quad  -----                                 & 8.05e-4  \quad  -----                                                 \\ \cline{2-7} 
& 0.4443 & 16401 & 2.76e-5  \quad 5.74  & 4.68e-5  \quad  4.34                  & 4.94e-6  \quad  5.30  & 4.63e-5  \quad  4.12                                                  \\ \cline{2-7} 
& 0.2221 & 65601 & 7.53e-7  \quad      5.20     & 2.84e-6  \quad  4.04          & 1.44e-7  \quad  5.10  & 2.83e-6 \quad   4.03                                                  \\ \cline{2-7} 
& 0.1777 & 102501 & 2.46e-7     \quad   5.02 & 1.16e-6  \quad  4.01             & 4.69e-8  \quad  5.04  & 1.16e-6  \quad  4.01                                                  \\ \hline
\end{tabular}}
\caption{A comparison of the TH-Non-Symmetric method without stabilization ($\delta=0$), to the same method with stabilization ($\delta=10$) at $t = 1.0$. The stabilization term takes the form $\delta|\ubold_h| \left( \nabla\cdot\ubold_h \right) \left(\nabla\cdot\wbold_h \right)$.}
\label{table_TH_nonsymmetric}
\end{center}

\begin{center}
\makeatletter\def\@captype{table}\makeatother
\resizebox{\textwidth}{33mm}{
\begin{tabular}{|p{0.1cm}|c|c|l|l|l|l|}
\hline
\multicolumn{1}{|c|}{\multirow{3}{*}{$k$}} & \multirow{3}{*}{$h$} & \multirow{3}{*}{d.o.f} & \multicolumn{2}{c|}{Upwind flux}                                                                                                     & \multicolumn{2}{c|}{Central flux}                                                                                                   \\ \cline{4-7} 
\multicolumn{1}{|c|}{}                   &                    &                        & \begin{tabular}[c]{@{}l@{}}$\|\ubold-\ubold_{h}\|_{\bm{L}^{2}(\Omega)}$ \\ error \hspace{0.65cm} order\end{tabular} & \begin{tabular}[c]{@{}l@{}}$\|\widetilde{p}-\widetilde{p}_{h}\|_{L^{2}(\Omega)}$\\ error \hspace{0.65cm} order\end{tabular} & \begin{tabular}[c]{@{}l@{}}$\|\ubold-\ubold_{h}\|_{\bm{L}^{2}(\Omega)}$\\ error \hspace{0.65cm} order\end{tabular} & \begin{tabular}[c]{@{}l@{}}$\|\widetilde{p}-\widetilde{p}_{h}\|_{L^{2}(\Omega)}$\\ error \hspace{0.65cm} order\end{tabular} \\ \hline
\multirow{4}{*}{1} & 0.8886 & 2101 & 1.98e-2   \quad    -----                   & 6.79e-2  \quad  ----- & 1.82e-2  \quad  -----& 6.77e-2  \quad  -----         \\ \cline{2-7} 
& 0.4443 & 8401 & 2.46e-3   \quad 3.01  & 1.72e-2  \quad  1.98                  & 2.29e-3  \quad  2.99 & 1.72e-2  \quad    1.98                                                  \\ \cline{2-7} 
& 0.2221 & 33601 & 2.95e-4  \quad  3.06                                     
& 4.30e-3  \quad    2.00 & 2.85e-4  \quad  3.01                                 & 4.30e-3  \quad  2.00                                                  \\ \cline{2-7} 
& 0.1777 & 52501 & 1.49e-4  \quad 3.05  & 2.76e-3  \quad  2.00                  & 1.46e-4  \quad  3.01  & 2.75e-3  \quad  2.00                                                  \\ \hline
\multirow{4}{*}{2} & 0.8886 & 4001 & 1.29e-3    \quad    -----                  & 7.04e-3  \quad  ----- & 1.28e-3  \quad  ----- & 7.03e-3  \quad  -----         \\ \cline{2-7} 
& 0.4443 & 16001 & 7.41e-5  \quad  4.12 & 8.89e-4  \quad  2.98                  & 7.37e-5  \quad  4.12 & 8.89e-4  \quad   2.98                                                  \\ \cline{2-7} 
& 0.2221 & 64001 & 4.55e-6  \quad 4.03 & 1.11e-4  \quad    3.00      
& 4.53e-6  \quad  4.02  & 1.11e-4  \quad  3.00                                                  \\ \cline{2-7} 
& 0.1777 & 100001 & 1.87e-6  \quad  3.99                                     
& 5.71e-5  \quad    3.00    & 1.86e-6  \quad  3.98                              & 5.71e-5  \quad  3.00                                                  \\ \hline
\multirow{4}{*}{3} & 0.8886 & 6501 & 8.55e-5    \quad    ----- 
& 5.50e-4  \quad  ----- & 8.26e-5  \quad  -----                                 & 5.50e-4  \quad  -----                                                 \\ \cline{2-7} 
& 0.4443 & 26001 & 2.78e-6  \quad 4.94  & 3.47e-5  \quad  3.99                  & 2.72e-6  \quad  4.92  & 3.47e-5  \quad  3.99                                                  \\ \cline{2-7} 
& 0.2221 & 104001 & 8.63e-8  \quad      5.01     & 2.17e-6  \quad  4.00          & 8.56e-8  \quad  4.99  & 2.17e-6 \quad   4.00                                                  \\ \cline{2-7} 
& 0.1777 & 162501 & 2.82e-8     \quad   5.01 & 8.91e-7  \quad  4.00             & 2.80e-8  \quad  5.00  & 8.91e-7  \quad  4.00                                                  \\ \hline
\end{tabular}}
\caption{A comparison of the BDM-Symmetric method with an upwind numerical flux, to the same method with a central numerical flux at $t = 1.0$.}
\label{table_our_up_cen}
\end{center}

\begin{center}
\makeatletter\def\@captype{table}\makeatother
\resizebox{\textwidth}{33mm}{
\begin{tabular}{|p{0.1cm}|c|c|l|l|l|l|}
\hline
\multicolumn{1}{|c|}{\multirow{3}{*}{$k$}} & \multirow{3}{*}{$h$} & \multirow{3}{*}{d.o.f} & \multicolumn{2}{c|}{Upwind flux}                                                                                                     & \multicolumn{2}{c|}{Central flux}                                                                                                   \\ \cline{4-7} 
\multicolumn{1}{|c|}{}                   &                    &                        & \begin{tabular}[c]{@{}l@{}}$\|\ubold-\ubold_{h}\|_{\bm{L}^{2}(\Omega)}$ \\ error \hspace{0.65cm} order\end{tabular} & \begin{tabular}[c]{@{}l@{}}$\|\widetilde{p}-\widetilde{p}_{h}\|_{L^{2}(\Omega)}$\\ error \hspace{0.65cm} order\end{tabular} & \begin{tabular}[c]{@{}l@{}}$\|\ubold-\ubold_{h}\|_{\bm{L}^{2}(\Omega)}$\\ error \hspace{0.65cm} order\end{tabular} & \begin{tabular}[c]{@{}l@{}}$\|\widetilde{p}-\widetilde{p}_{h}\|_{L^{2}(\Omega)}$\\ error \hspace{0.65cm} order\end{tabular} \\ \hline
\multirow{4}{*}{1} & 0.8886 & 2101 & 1.95e-2   \quad    -----                   & 6.80e-2  \quad  -----  & 1.77e-2  \quad  -----& 6.77e-2  \quad  -----         \\ \cline{2-7} 
& 0.4443 & 8401 & 2.40e-3   \quad 3.02  & 1.72e-2  \quad  1.98                  & 2.22e-3  \quad  2.99 & 1.72e-2  \quad    1.98                                                  \\ \cline{2-7} 
& 0.2221 & 33601 & 2.88e-4  \quad  3.06                                     
& 4.31e-3  \quad    2.00 & 2.78e-4  \quad  3.00                                 & 4.30e-3  \quad  2.00                                                  \\ \cline{2-7} 
& 0.1777 & 52501 & 1.46e-4  \quad 3.04  & 2.76e-3  \quad  2.00                  & 1.42e-4  \quad  3.00  & 2.76e-3  \quad  2.00                                                  \\ \hline
\multirow{4}{*}{2} & 0.8886 & 4001 & 1.31e-3    \quad    -----                  & 7.05e-3  \quad  ----- & 1.29e-3  \quad  ----- & 7.03e-3  \quad  -----         \\ \cline{2-7} 
& 0.4443 & 16001 & 7.52e-5  \quad  4.12 & 8.90e-4  \quad  2.99                  & 7.47e-5  \quad  4.11 & 8.89e-4  \quad   2.98                                                  \\ \cline{2-7} 
& 0.2221 & 64001 & 4.60e-6  \quad 4.03 & 1.11e-4  \quad    3.00      
& 4.58e-6  \quad  4.03  & 1.11e-4  \quad  3.00                                                  \\ \cline{2-7} 
& 0.1777 & 100001 & 1.89e-6  \quad  3.99                                     
& 5.71e-5  \quad    3.00    & 1.88e-6  \quad  3.99                              & 5.71e-5  \quad  3.00                                                  \\ \hline
\multirow{4}{*}{3} & 0.8886 & 6501 & 8.66e-5    \quad    ----- 
& 5.50e-4  \quad  ----- & 8.36e-5  \quad  -----                                 & 5.50e-4  \quad  -----                                                 \\ \cline{2-7} 
& 0.4443 & 26001 & 2.82e-6  \quad 4.94  & 3.47e-5  \quad  3.99                  & 2.77e-6  \quad  4.92  & 3.47e-5  \quad  3.99                                                  \\ \cline{2-7} 
& 0.2221 & 104001 & 8.77e-8  \quad      5.01     & 2.17e-6  \quad  4.00          & 8.70e-8  \quad  4.99  & 2.17e-6 \quad   4.00                                                  \\ \cline{2-7} 
& 0.1777 & 162501 & 2.87e-8     \quad   5.01 & 8.91e-7  \quad  4.00             & 2.85e-8  \quad  5.00  & 8.91e-7  \quad  4.00                                                  \\ \hline
\end{tabular}}
\caption{A comparison of the BDM-Non-Symmetric method with an upwind numerical flux, to the same method with a central numerical flux at $t = 1.0$.}
\label{table_our_up_comparison}
\end{center}

\begin{figure}[h]
\centering
\includegraphics[width=0.425\linewidth,trim={5cm 2cm 0cm 2cm},clip]{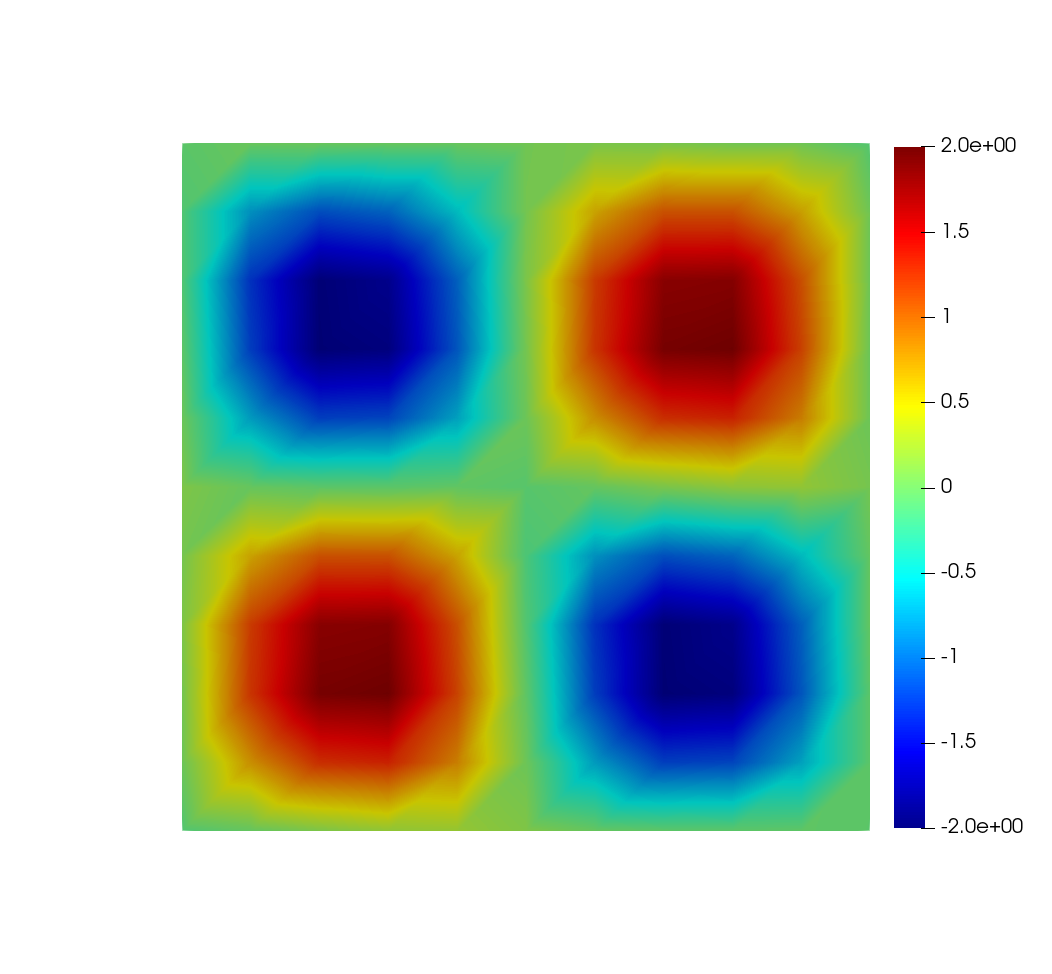}
\includegraphics[width=0.425\linewidth,trim={5cm 2cm 0cm 2cm},clip]{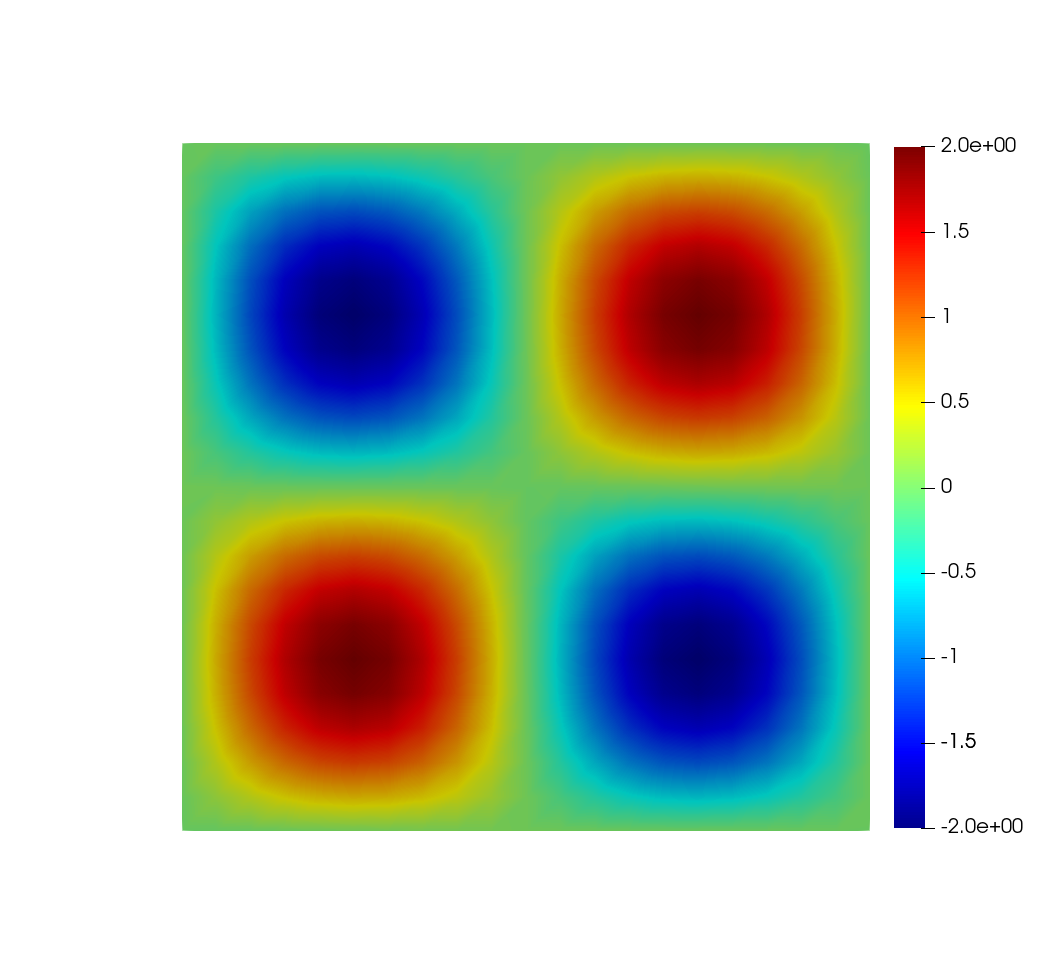}
\caption{Contours of vorticity for the BDM-Symmetric method with $k = 1$ and the upwind flux on meshes with $h = 0.8886$ (left) and $h = 0.4443$ (right) at $t = 1.0$ --- Example 1}
\label{fig:vorticity_k1}
\end{figure}
\begin{figure}[h]
\centering
\includegraphics[width=0.42\linewidth,trim={5cm 2cm 0cm 2cm},clip]{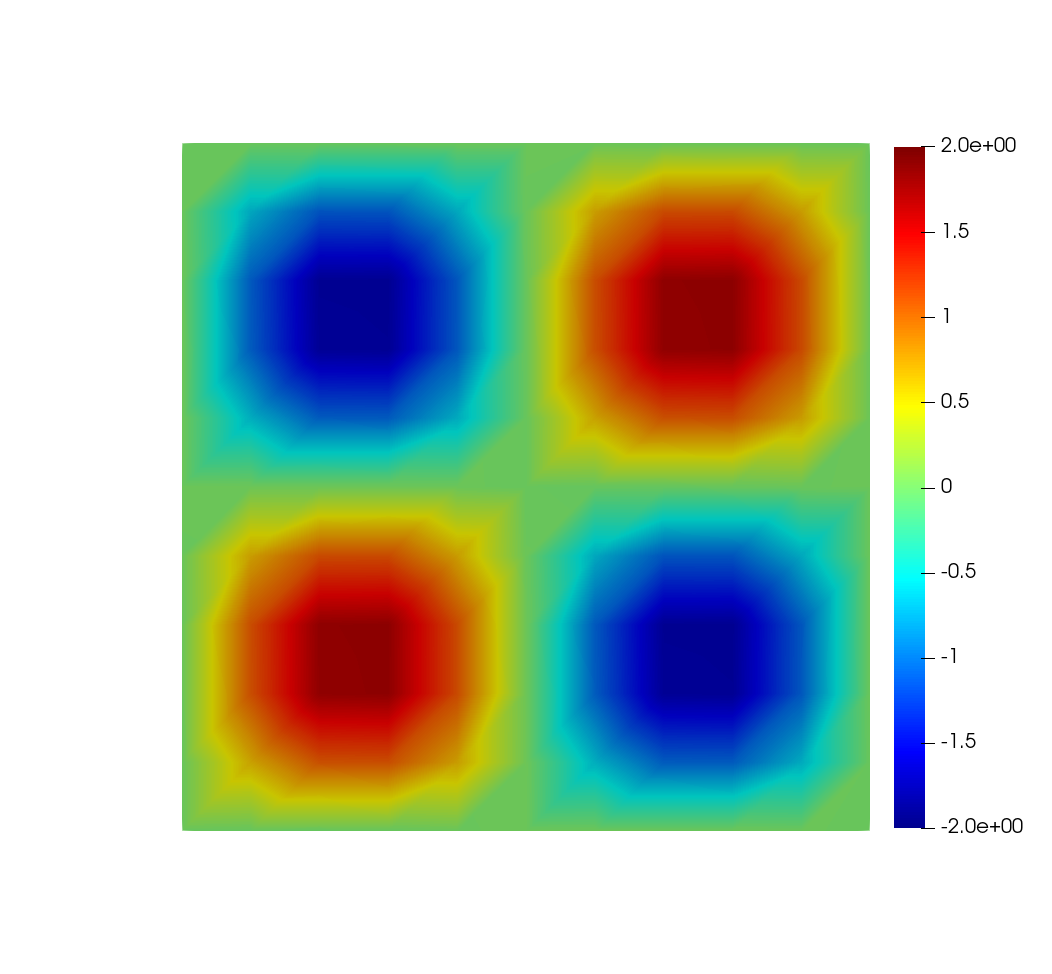}
\includegraphics[width=0.42\linewidth,trim={5cm 2cm 0cm 2cm},clip]{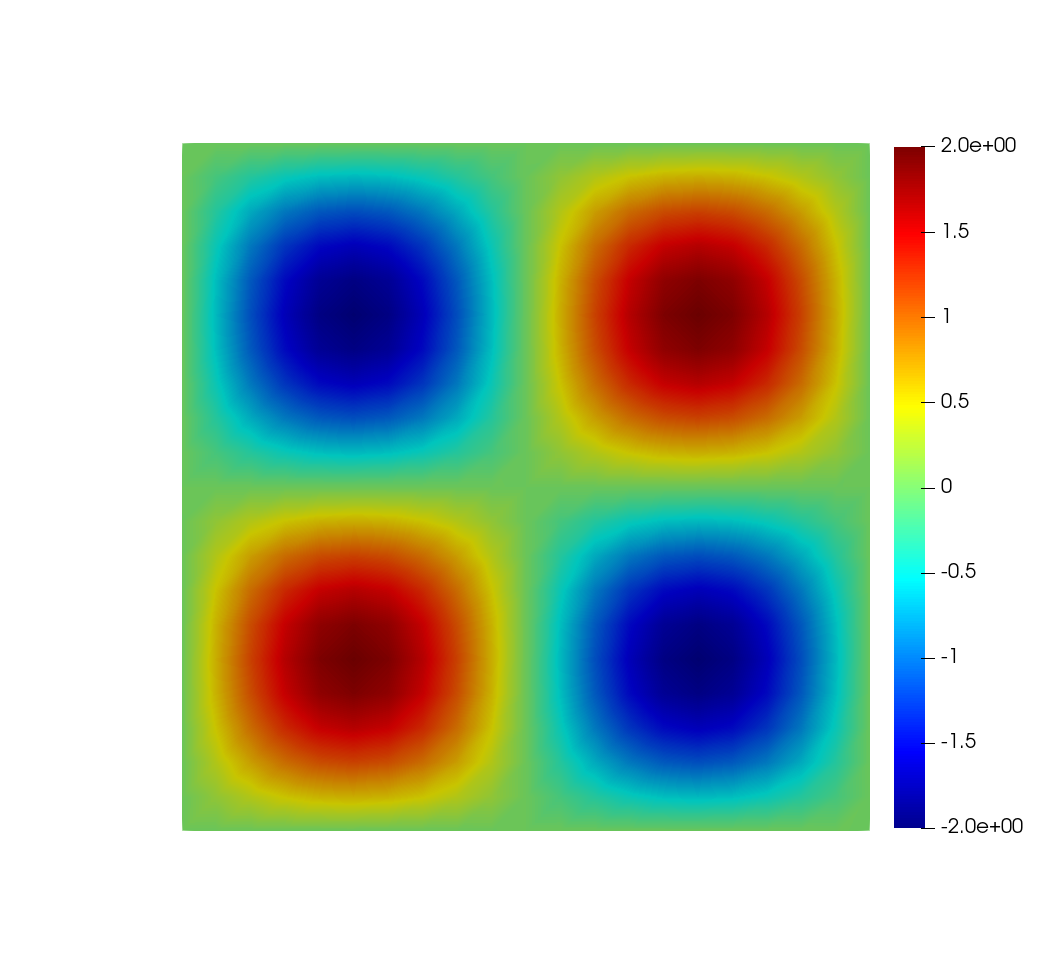}
\caption{Contours of vorticity for the BDM-Symmetric method with $k = 2$ and the upwind flux on meshes with $h = 0.8886$ (left) and $h = 0.4443$ (right) at $t = 1.0$ --- Example 1}
\label{fig:vorticity_k2}
\end{figure}

\begin{figure}[h]
\centering
\includegraphics[width=0.425\linewidth,trim={5cm 2cm 0cm 2cm},clip]{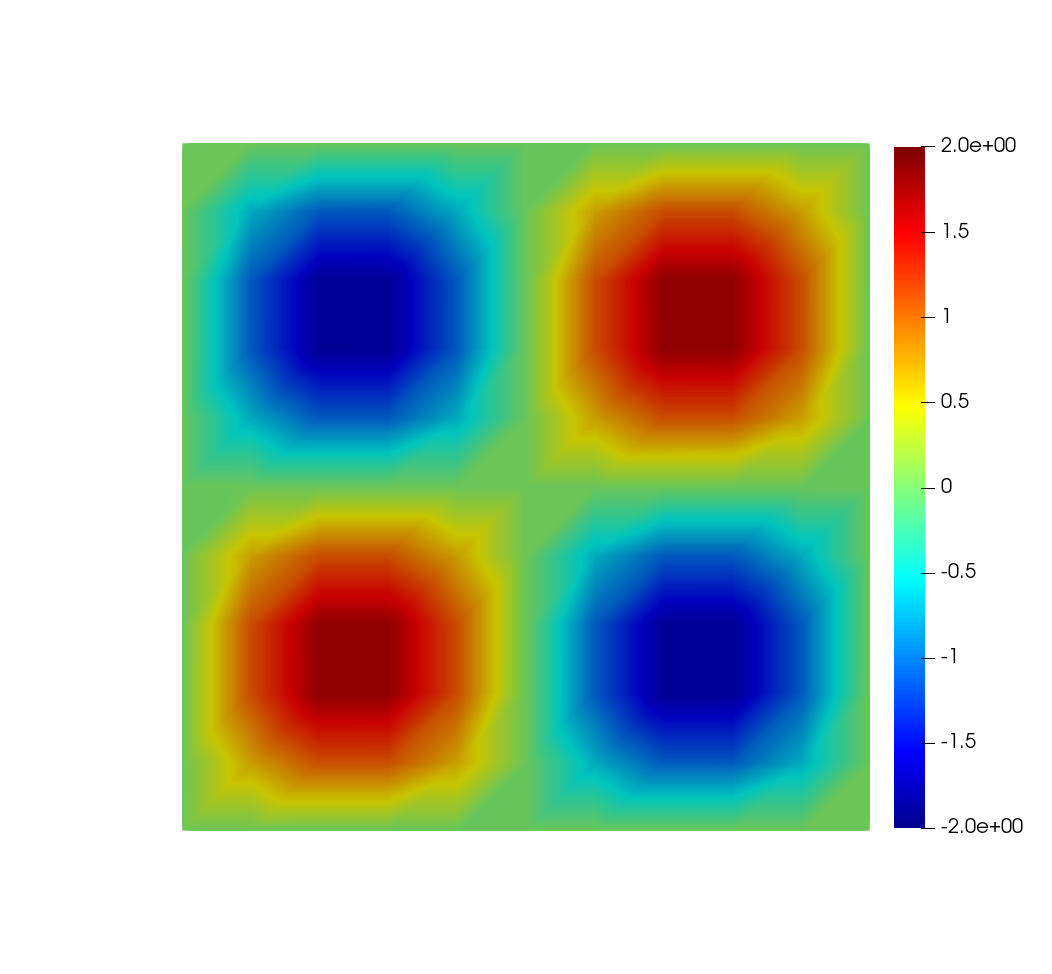}
\includegraphics[width=0.425\linewidth,trim={5cm 2cm 0cm 2cm},clip]{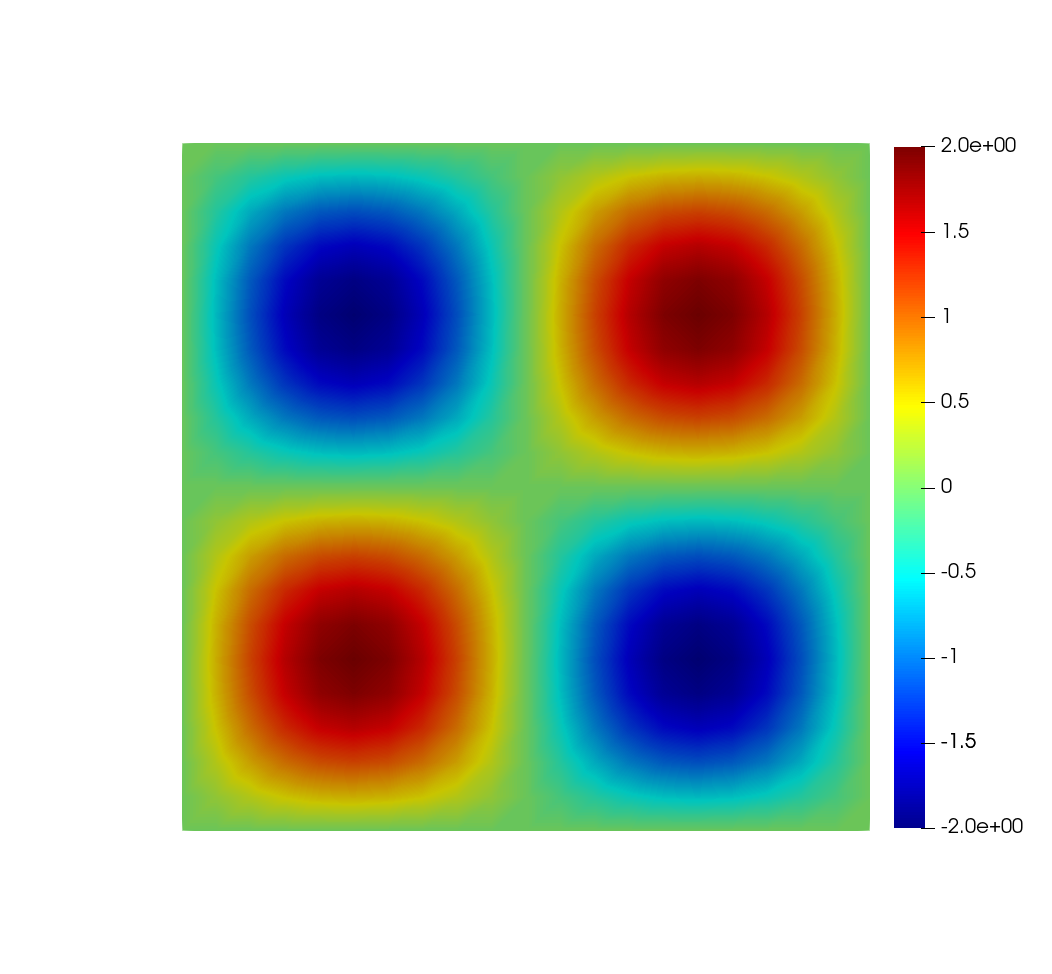}
\caption{Contours of vorticity for the BDM-Symmetric method with $k = 3$ and the upwind flux on meshes with $h = 0.8886$ (left) and $h = 0.4443$ (right) at $t = 1.0$ --- Example 1}
\label{fig:vorticity_k3}
\end{figure}

\begin{figure}[h]
\centering
\includegraphics[width=0.425\linewidth,trim={5cm 2cm 0cm 2cm},clip]{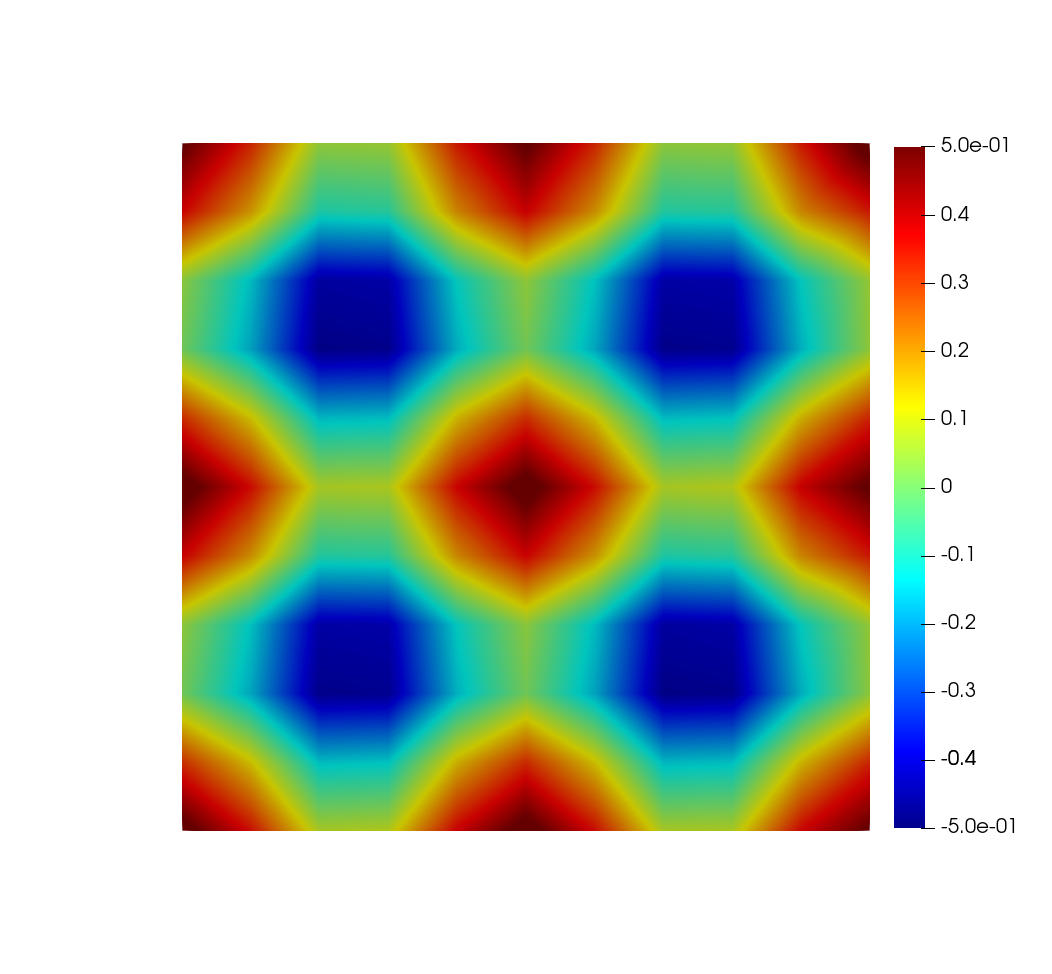}
\includegraphics[width=0.425\linewidth,trim={5cm 2cm 0cm 2cm},clip]{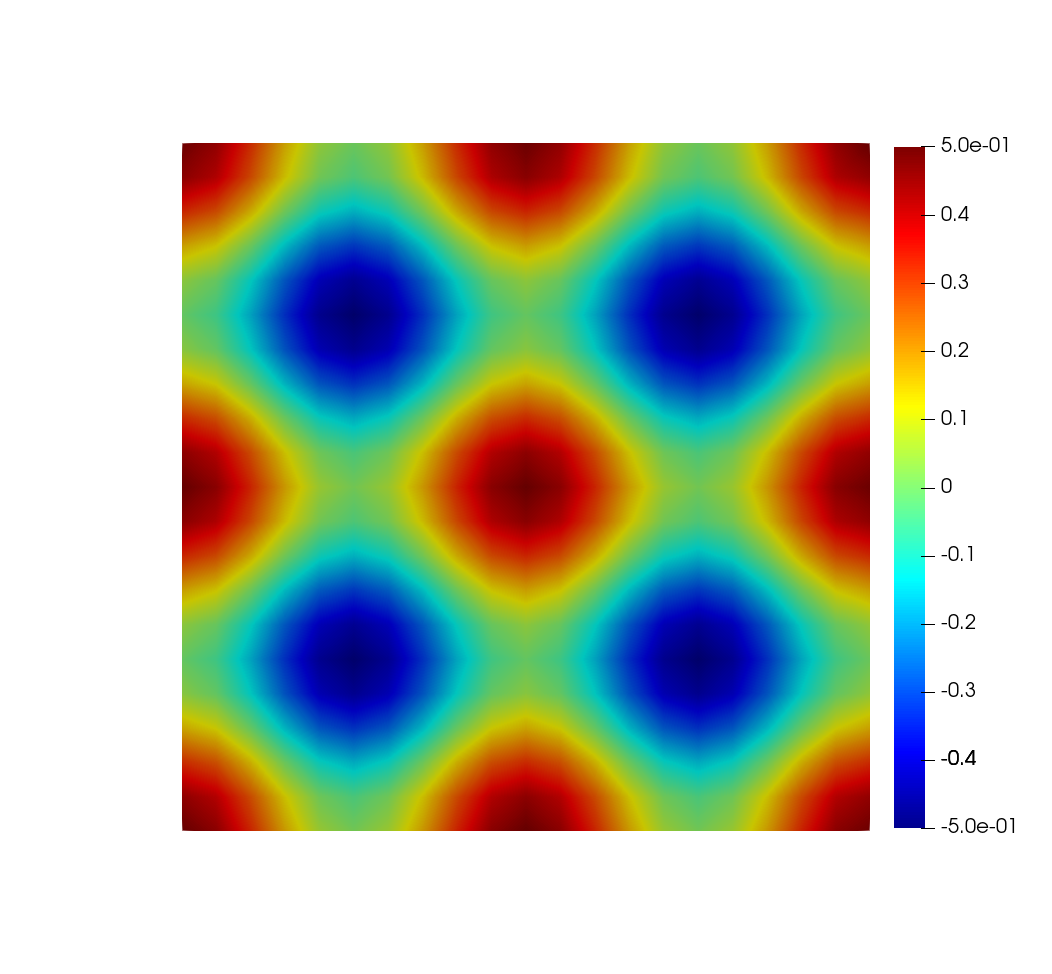}
\caption{Contours of pressure for the BDM-Symmetric method with $k = 1$ and the upwind flux on meshes with $h = 0.8886$ (left) and $h = 0.4443$ (right) at $t = 1.0$ --- Example 1}
\label{fig:pressure_k1}
\end{figure}

\begin{figure}[h]
\centering
\includegraphics[width=0.425\linewidth,trim={5cm 2cm 0cm 2cm},clip]{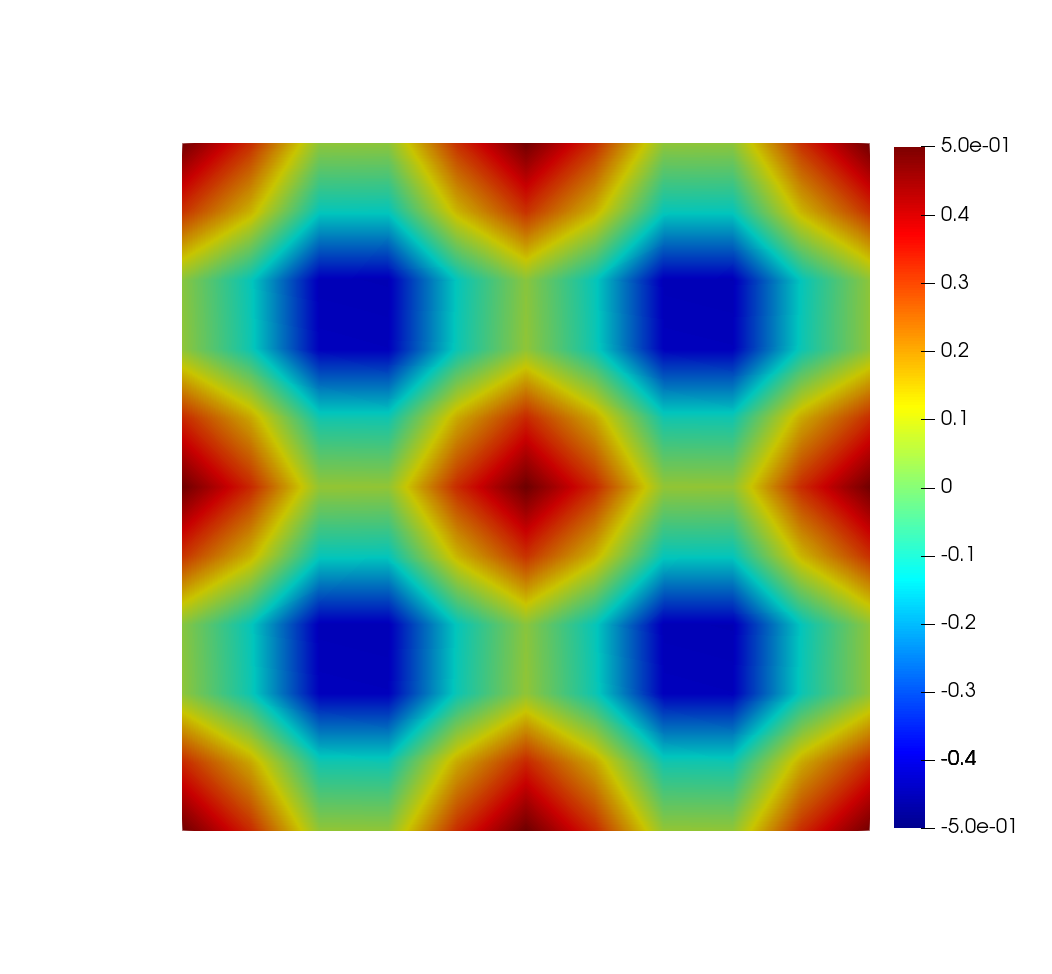}
\includegraphics[width=0.425\linewidth,trim={5cm 2cm 0cm 2cm},clip]{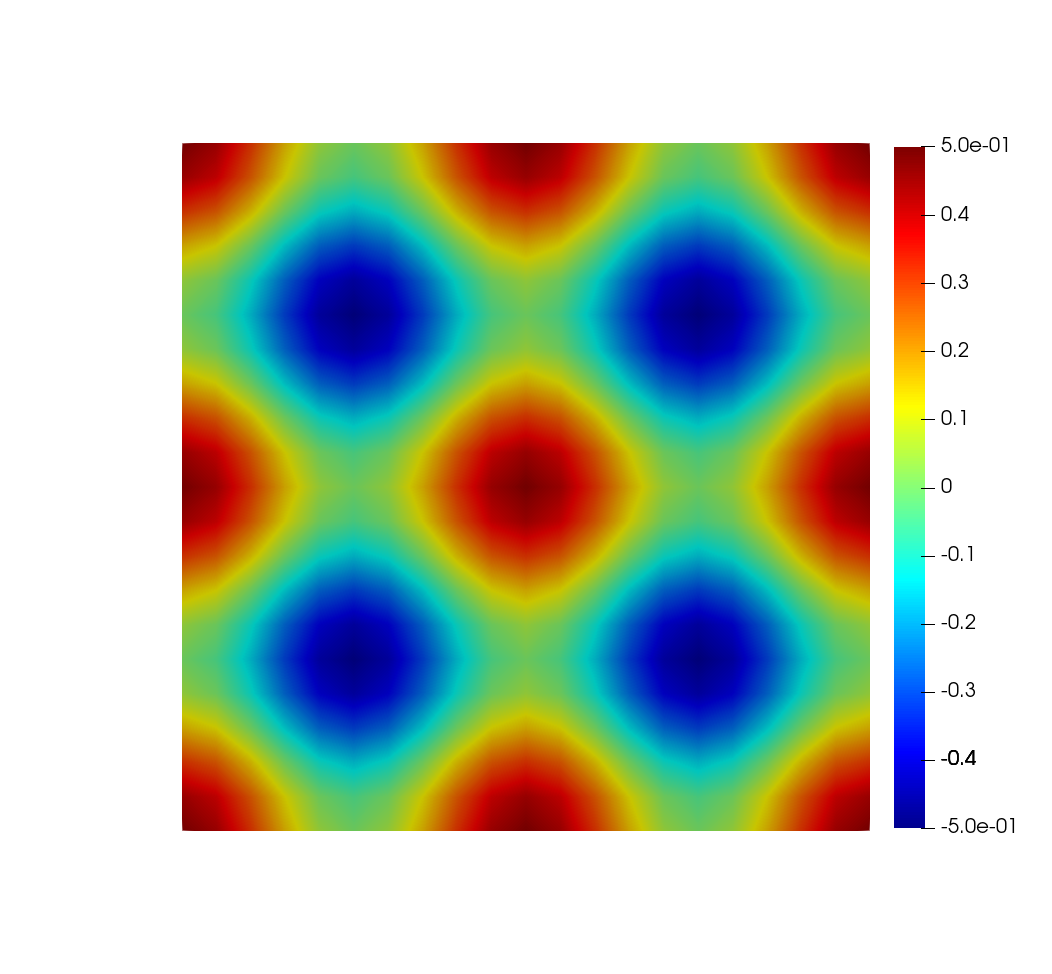}
\caption{Contours of pressure for the BDM-Symmetric method with $k = 2$ and the upwind flux on meshes with $h = 0.8886$ (left) and $h = 0.4443$ (right) at $t = 1.0$ --- Example 1}
\label{fig:pressure_k2}
\end{figure}
\begin{figure}[h]
\centering
\includegraphics[width=0.425\linewidth,trim={5cm 2cm 0cm 2cm},clip]{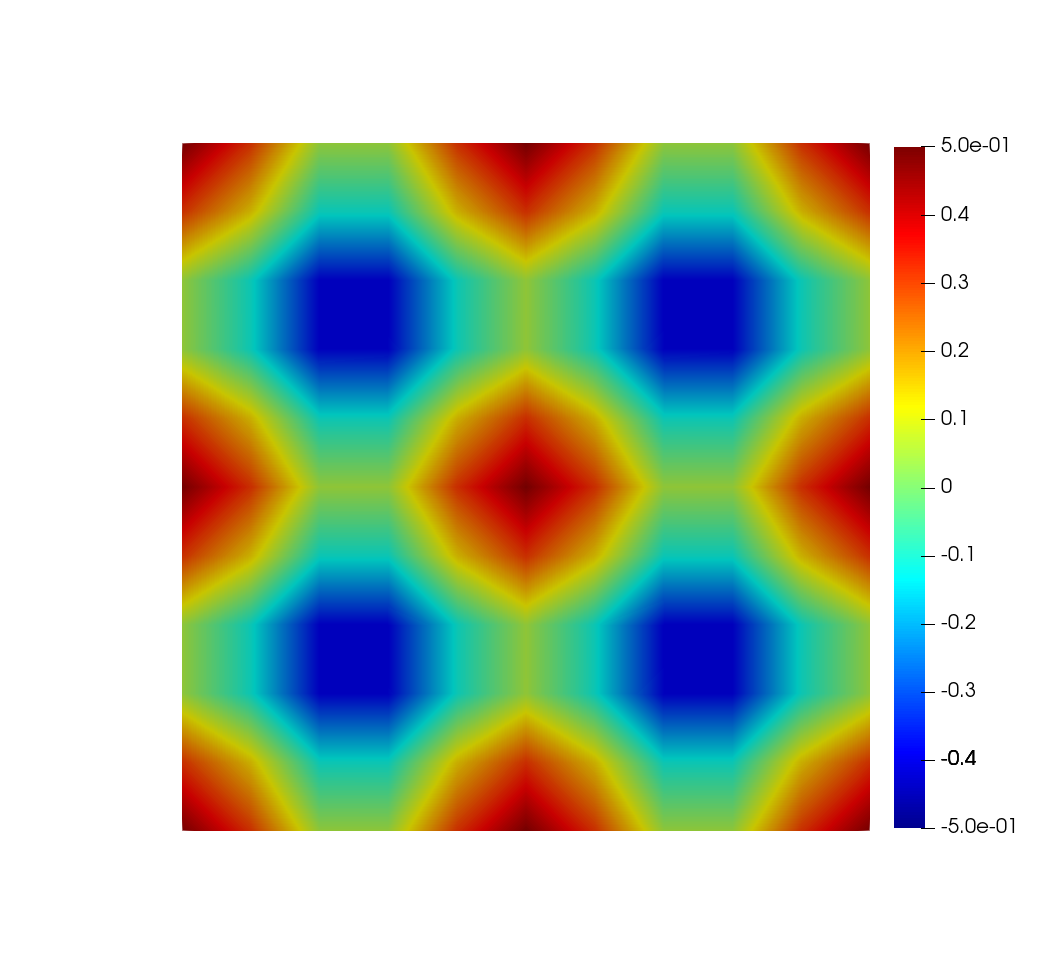}
\includegraphics[width=0.425\linewidth,trim={5cm 2cm 0cm 2cm},clip]{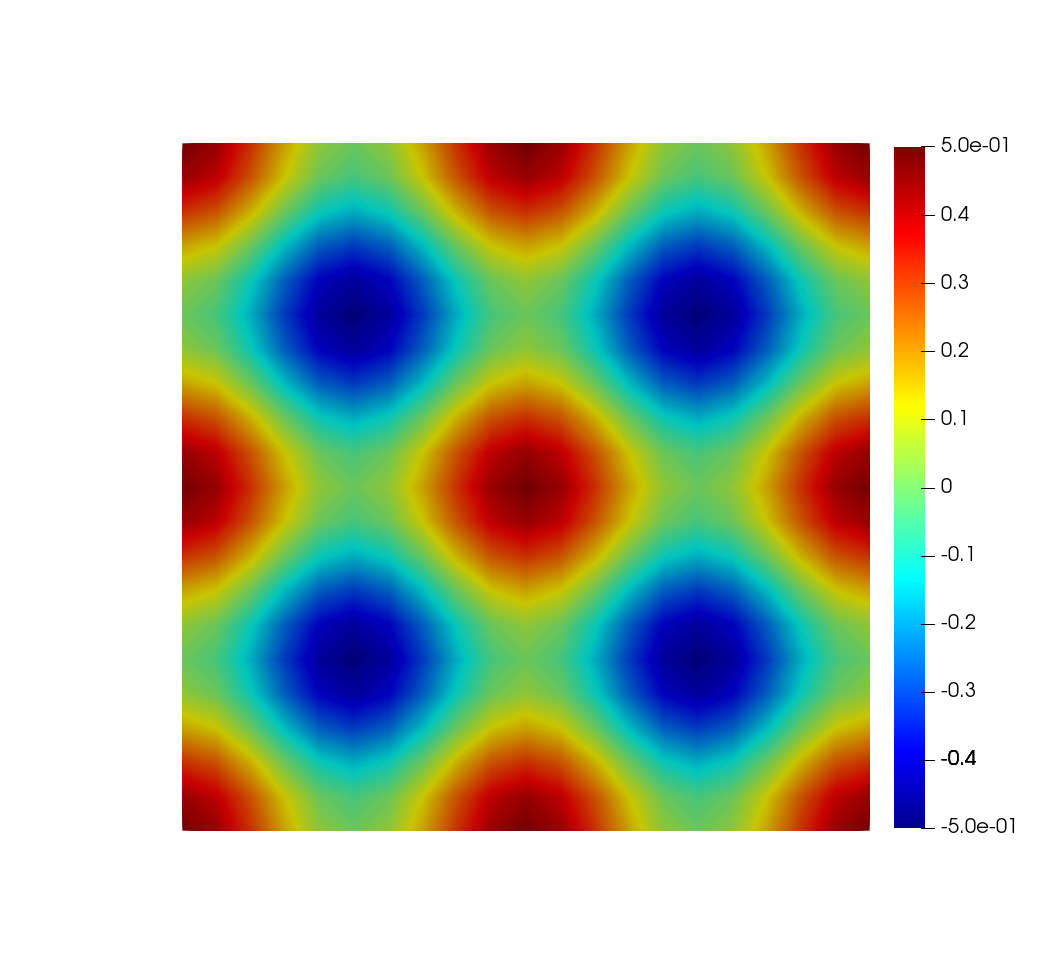}
\caption{Contours of pressure for the BDM-Symmetric method with $k = 3$ and the upwind flux on meshes with $h = 0.8886$ (left) and $h = 0.4443$ (right) at $t = 1.0$ --- Example 1}
\label{fig:pressure_k3}
\end{figure}

\begin{figure}[h]
\centering
\includegraphics[width=0.425\linewidth,trim={5cm 2cm 0cm 2cm},clip]{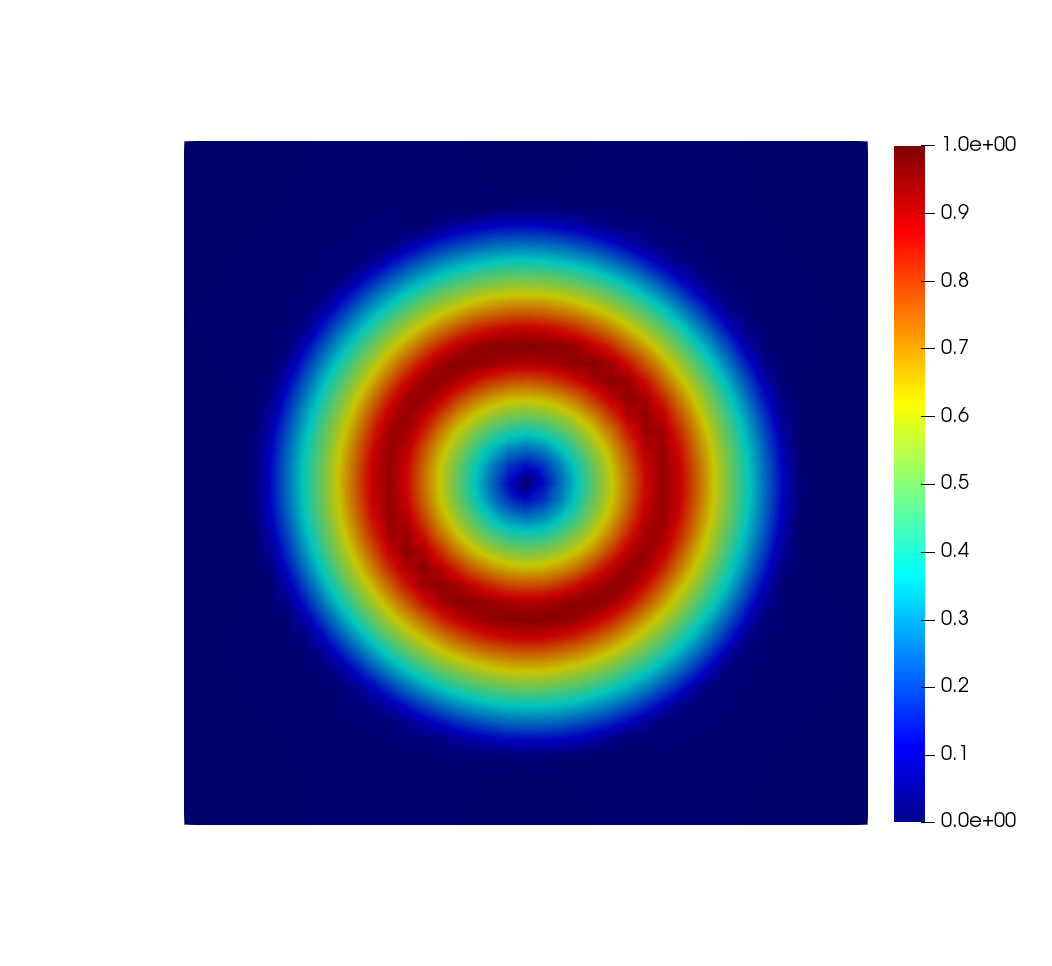}
\includegraphics[width=0.425\linewidth,trim={5cm 2cm 0cm 2cm},clip]{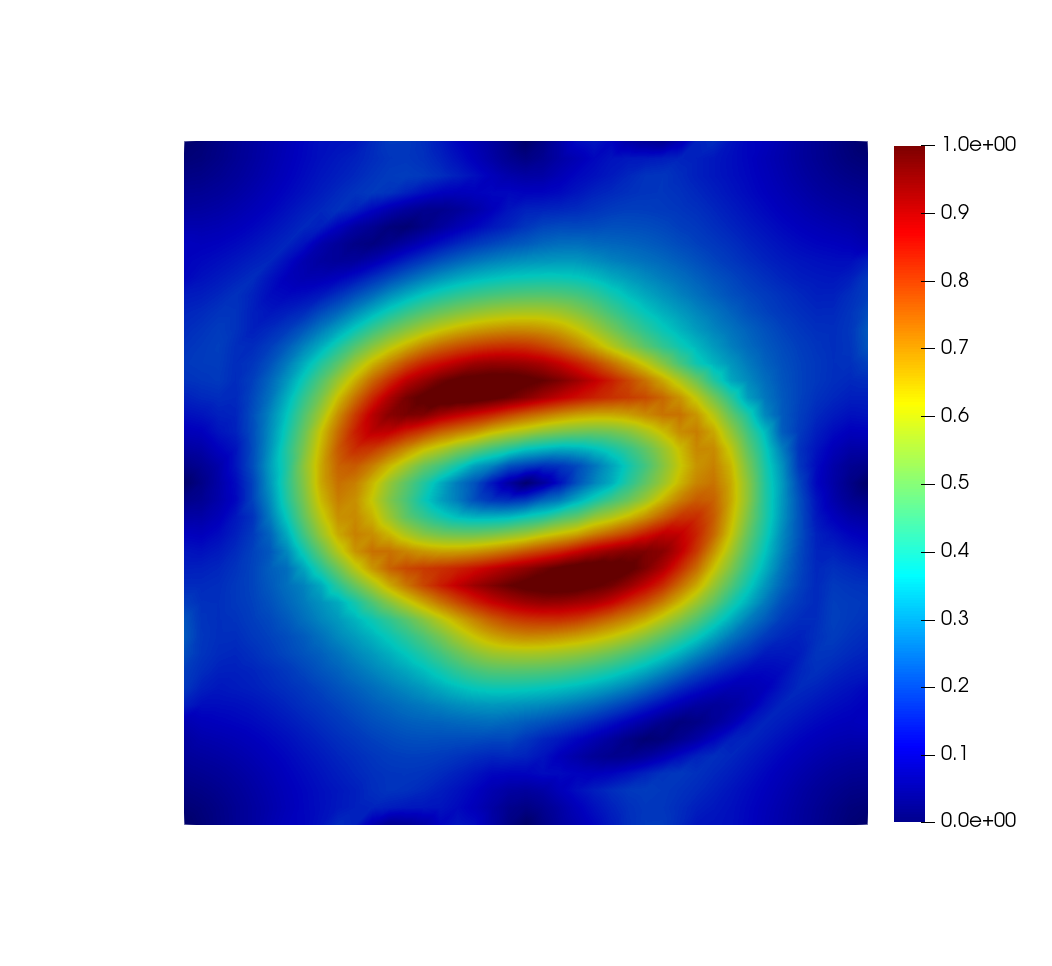}
\caption{Contours of velocity magnitude for the BDM-Symmetric method with upwind flux (left) and central flux (right) with $k = 1$ and $h=0.0354$ at $t = 14.0$ --- Example 2}
\label{fig:v_t4_k1}
\end{figure}

\begin{figure}[h]
\centering
\includegraphics[width=0.425\linewidth,trim={5cm 2cm 0cm 2cm},clip]{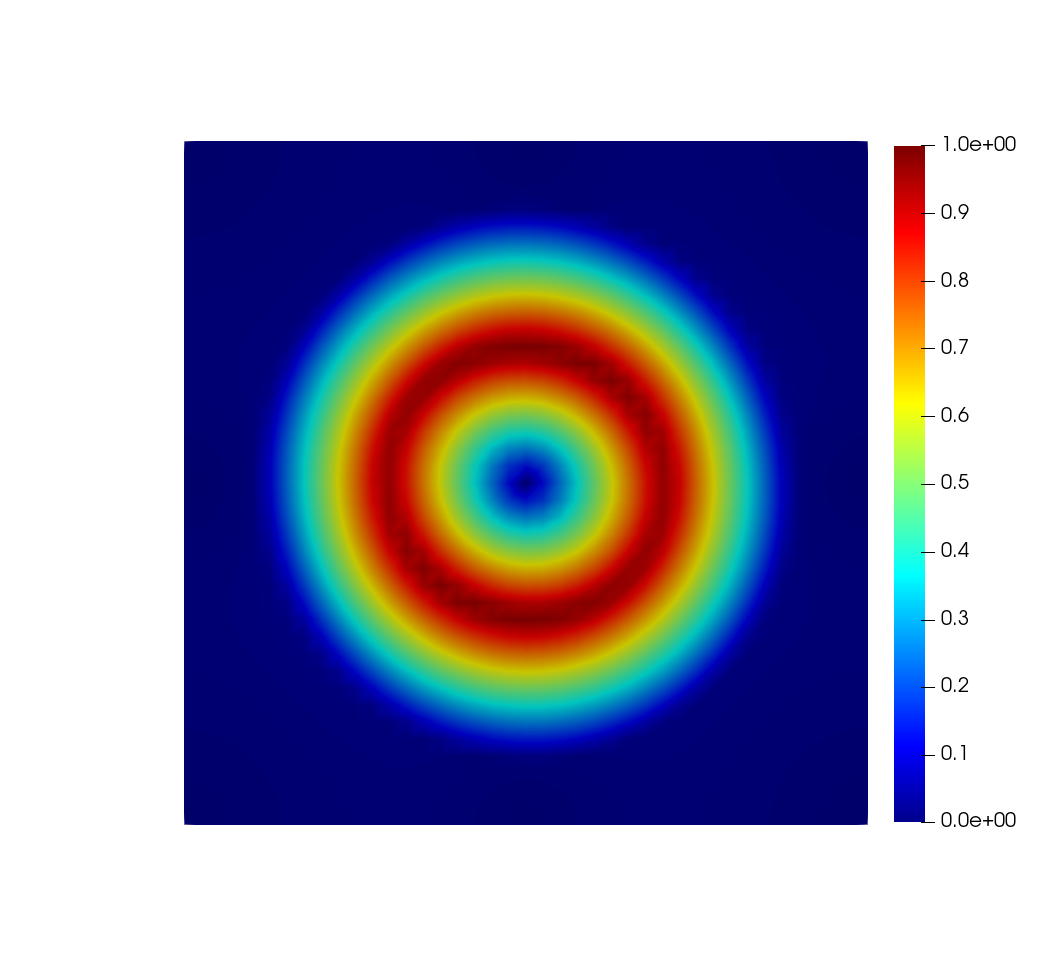}
\includegraphics[width=0.425\linewidth,trim={5cm 2cm 0cm 2cm},clip]{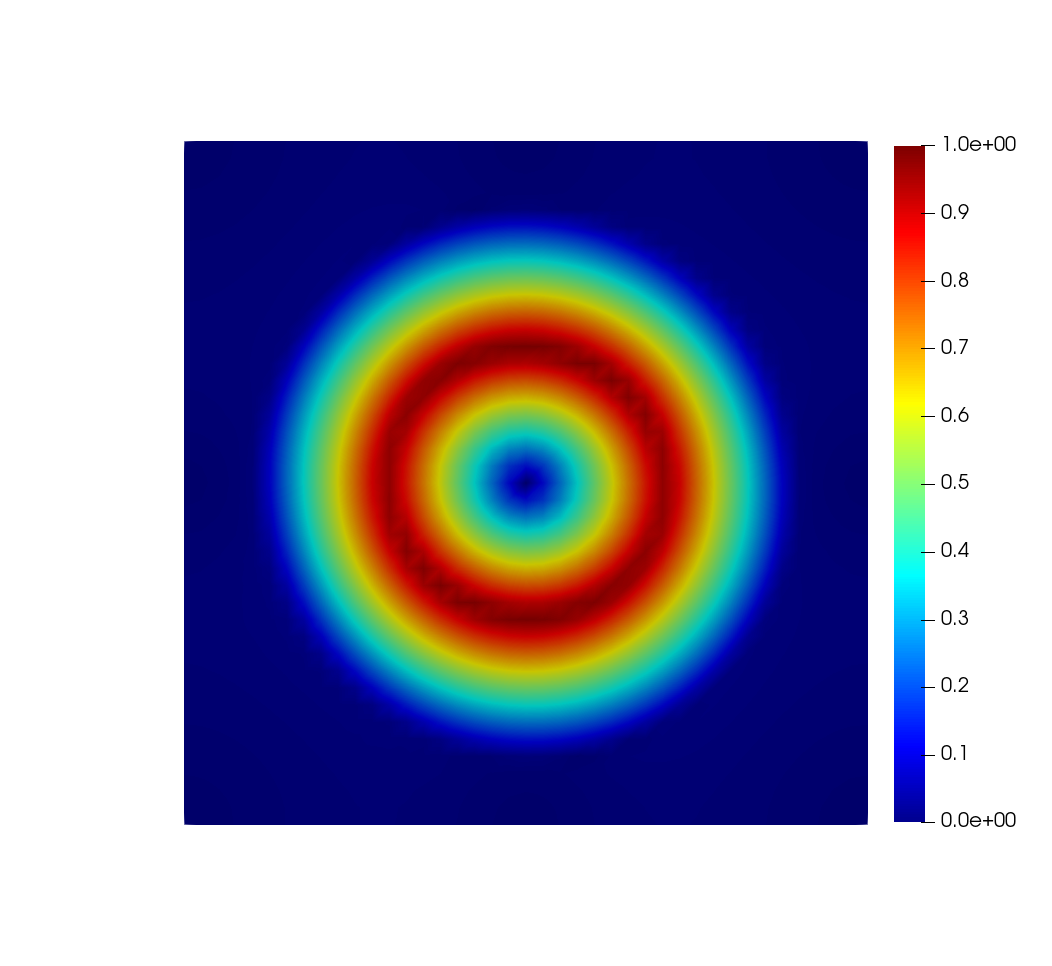}
\caption{Contours of velocity magnitude for the BDM-Symmetric method with upwind flux (left) and central flux (right) with $k = 2$ and $h=0.0354$ at $t = 14.0$ --- Example 2}
\label{fig:v_t4_k2}
\end{figure}

\begin{figure}[h]
\centering
\includegraphics[width=0.425\linewidth,trim={5cm 2cm 0cm 2cm},clip]{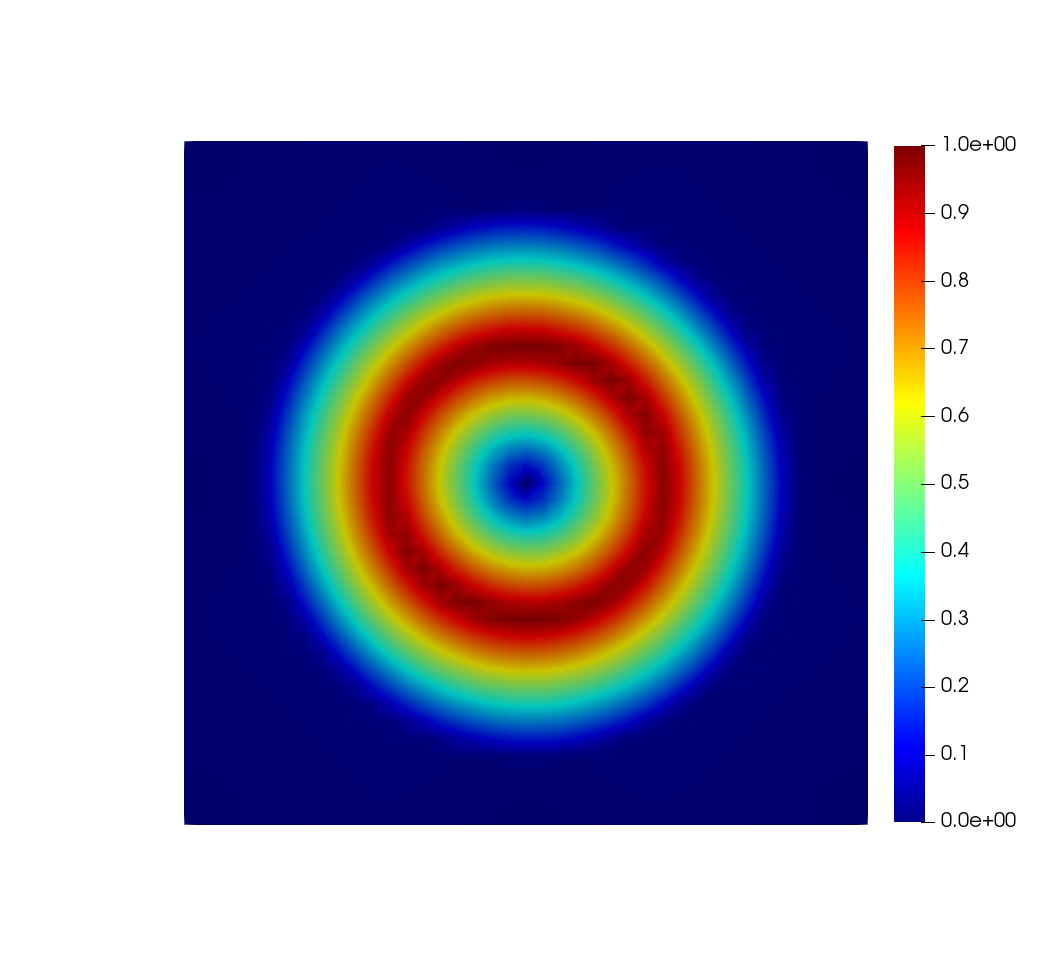}
\includegraphics[width=0.425\linewidth,trim={5cm 2cm 0cm 2cm},clip]{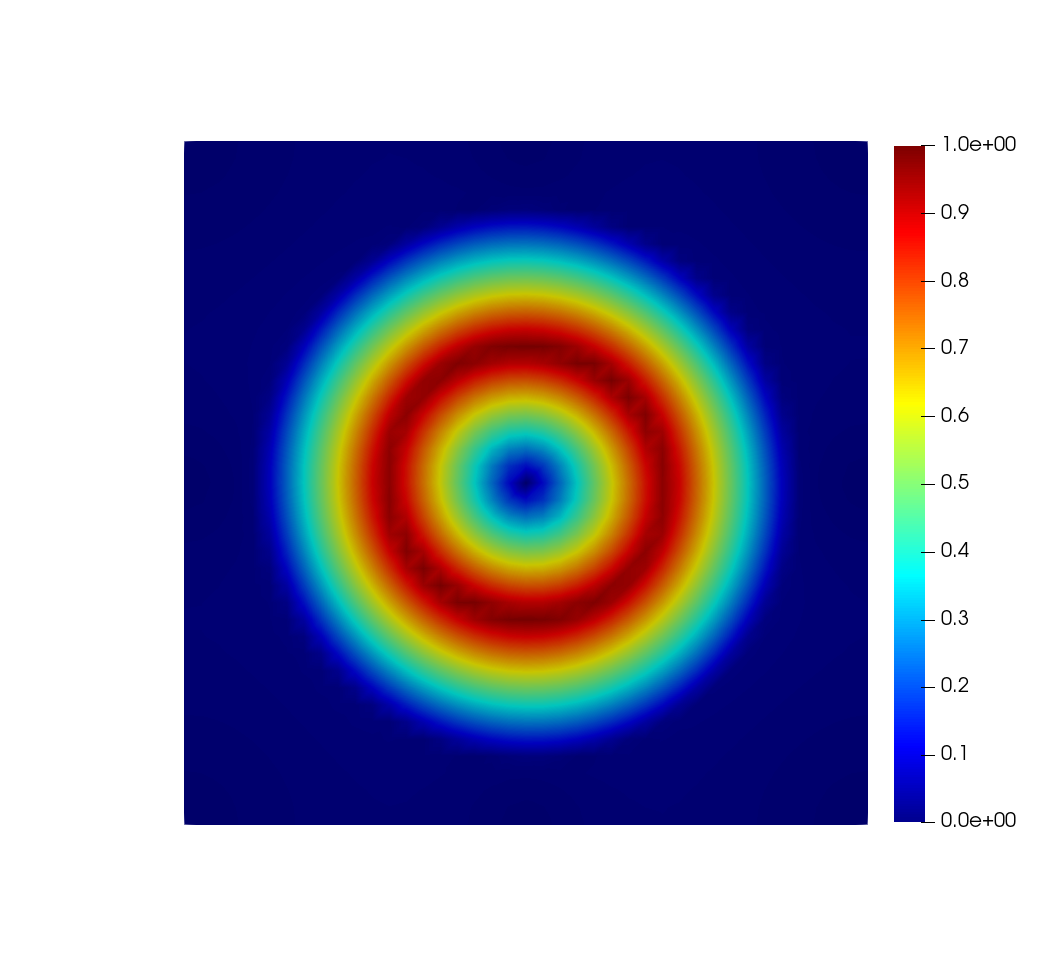}
\caption{Contours of velocity magnitude for the BDM-Symmetric method with upwind flux (left) and central flux (right) with $k = 3$ and $h=0.0354$ at $t = 14.0$ --- Example 2}
\label{fig:v_t4_k3}
\end{figure}

\clearpage

\section{Conclusion}\label{conclusion}

This paper presents a more consistent way to treat the divergence-free constraint that appears in both the mass and momentum conservation equations. The key idea is to use the full, compressible version of the stress tensor, and allow the finite element scheme to completely control the enforcement of the divergence-free constraint in the momentum equation. In this way, the divergence-free constraint in the momentum equation is not enforced \emph{a priori} (which is the case in most classical schemes), but rather is dictated by the scheme. The resulting schemes are versatile, in the sense that they are easily generalizable to compressible flows, and they enable the straightforward treatment of rotationally symmetric flows.

In order to help build a mathematical foundation for the schemes, we established the existence of a new norm associated with the viscous bilinear form. In addition, we presented some coercivity and semi-coercivity results that govern the bilinear and trilinear forms associated with the schemes. We also proved an L2-stability condition that governs the velocity fields for the general class of schemes. Finally, we constructed versatile, Taylor-Hood-based and BDM-based schemes, and then compared them with the standard BDM-based scheme of~\cite{schroeder2018towards}. The numerical simulations indicated that both sets of schemes achieve the same order of accuracy. Therefore, utilization of the symmetric tensor formulation increases the versatility of the schemes, while maintaining their accuracy. 

There are many promising paths for additional research on the proposed schemes. For example, it would be useful to extend the stability analysis in the present work, and develop L2-bounds for the pressure, and (perhaps) for secondary quantities, such as the enstrophy. Furthermore, it would be useful to construct rigorous error estimates for the schemes, in accordance with the analysis in (for example)~\cite{guzman2016h, schroeder2018divergence, schroeder2018towards, schroeder2017pressure}. In addition, it may be possible to extend the conventional  analysis techniques in order to prove stability and error estimates for the schemes in the context of weakly compressible flows, and in the singular limit, as the overall compressibility of a weakly compressible flow approaches zero.

From a more practical standpoint, the next step is to apply the schemes to compressible flow problems. Since the schemes are provably stable for incompressible flows, there is a strong possibility that they will remain stable in weakly compressible flows, regardless of whether or not a rigorous proof can be constructed. In addition, we anticipate that the high-order accuracy of the schemes will extend to weakly compressible flows, without any significant issues. Finally, there are potentially promising applications of the schemes to complex flows, which contain multiple regimes of incompressible and compressible flow. We hope to explore each of these potential applications in future work.


\section*{Acknowledgements}
The authors would like to thank Prof. Johnny Guzman (Brown University), Dr. Qingguo Hong (Pennsylvania State University), and Dr. Dmitry Kamenetskiy (The Boeing Company) for their participation in conversations that helped shape this work.

\section*{Disclosure Statement}

The authors certify that they have no affiliations with or involvement in any organization or entity with any financial interest, or non-financial interest in the subject matter or materials discussed in this manuscript.

\appendix

\section{Complementary Result(s)}\label{compelmentary}

\begin{lemma} Suppose that $\wbold_h$ and $\bm{\beta}_h\in\bm{H}_{0}(\text{div};\Omega)$. Then, the following identity holds:
%
\begin{align}
\ipt{\bm{\beta}_h \cdot \nabla_h \wbold_h}{\wbold_h} + \frac{1}{2} \ipt{\left(\nabla \cdot \bm{\beta}_h \right) \wbold_h}{\wbold_h} =  \iipbf{ \left(\bm{\beta}_h \cdot \nbold_F \right) \llbracket \wbold_h \rrbracket}{\llcurve \wbold_h \rrcurve}.
\label{temam_id_spec}
\end{align}
\label{jump_lemma}
\end{lemma}

\begin{proof}
We begin by using integration by parts to expand the second term on the LHS of Eq.~\eqref{temam_id_spec} 
\begin{align}
 \frac{1}{2} \ipt{\left(\nabla \cdot \bm{\beta}_h \right) \wbold_h}{\wbold_h} & = \frac{1}{2}  \ipbt{\wbold_h}{\wbold_h \left(\bm{\beta}_h \cdot \nbold \right)}  -  \ipt{\bm{\beta}_h \cdot \nabla_h \wbold_h}{\wbold_h}. \label{ip_sec}
\end{align}
Upon using Eq.~\eqref{ip_sec} to rewrite the LHS of  Eq.~\eqref{temam_id_spec}, we obtain the following
\begin{align}
& \ipt{\bm{\beta}_h \cdot \nabla_h \wbold_h}{\wbold_h} + \frac{1}{2} \ipt{\left(\nabla \cdot \bm{\beta}_h \right) \wbold_h}{\wbold_h} = \frac{1}{2}  \ipbt{\wbold_h}{\wbold_h \left(\bm{\beta}_h \cdot \nbold \right)}. \label{ip_third}
\end{align}
Next, we can rewrite the RHS of Eq.~\eqref{ip_third} in terms of summations over faces $F$ in the mesh
\begin{align}
\nonumber \frac{1}{2}  \ipbt{\wbold_h}{\wbold_h \left(\bm{\beta}_h \cdot \nbold \right)} & = \frac{1}{2} \iipbf{\llbracket \left( \wbold_h \cdot \wbold_h \right) \bm{\beta}_h \rrbracket}{\nbold_F} \\[1.5ex]
& = \frac{1}{2} \iipbf{\llbracket \bm{\beta}_h \rrbracket}{\nbold_F \llcurve \wbold_h \cdot \wbold_h \rrcurve} + \iipbf{ \left(\llcurve \bm{\beta}_h \rrcurve \cdot \nbold_F \right) \llbracket \wbold_h \rrbracket}{\llcurve \wbold_h \rrcurve}.\label{ip_fourth}
\end{align}
We complete the proof by substituting Eq.~\eqref{ip_fourth} into Eq.~\eqref{ip_third}, and noting that $\bm{\beta}_h$ has continuous normal components.
\end{proof}

\section{Derivation of the General Mixed Methods} \label{fem_deriv}

\subsection{Mass Equation Derivation}

One may substitute $\ubold_h$ into Eq.~\eqref{mass_cons}, multiply by a test function $q_h$, and integrate over the entire domain in order to obtain 
\begin{align}
\ipt{\nabla \cdot \ubold_h}{q_h} = 0, \label{mass_ibp_one}
\end{align}
which is identical to Eq.~\eqref{mass_cons_disc}.

%

\subsection{Linear Momentum Equation Derivation}

One may substitute $\widetilde{p}_h$ and $\ubold_h$ into Eq.~\eqref{moment_cons}, compute the dot product with a test function $\wbold_h$, and integrate over the entire domain in order to yield
\begin{align}
& \ipt{\partial_t \ubold_h}{\wbold_h} + \ipt{\nabla_h \cdot \left( \ubold_h \otimes \ubold_h + \widetilde{p}_h \mathbb{I} \right)}{\wbold_h} \label{moment_ibp_neg} \\[1.5ex]
\nonumber &- \nu_h \ipt{ \nabla_h \cdot \left( \nabla_h \ubold_h + \nabla_h \ubold_{h}^{T} - \frac{2}{3} \left(\nabla \cdot \ubold_h \right) \mathbb{I} \right)}{\wbold_h} = \ipt{\widetilde{\bm{f}}}{\wbold_h}. 
\end{align}
Upon integrating the second and third terms by parts and inserting numerical fluxes $\hat{\bm{\sigma}}_{\text{inv}}$ and $\hat{\bm{\sigma}}_{\text{vis}}$, one obtains
\begin{align}
\ipt{\nabla_h \cdot \left( \ubold_h \otimes \ubold_h + \widetilde{p}_h \mathbb{I} \right)}{\wbold_h} &= - \ipt{ \ubold_h \otimes \ubold_h + \widetilde{p}_h \mathbb{I}}{\nabla_h \wbold_h} + \ipbt{\left( \ubold_h \otimes \ubold_h + \widetilde{p}_h \mathbb{I} \right)\nbold}{\wbold_h} \label{moment_ibp_zero} \\[1.5ex]
\nonumber & \equiv - \ipt{  \ubold_h \otimes \ubold_h + \widetilde{p}_h \mathbb{I}}{\nabla_h \wbold_h} + \ipbt{ \hat{\bm{\sigma}}_{\text{inv}} \, \nbold}{\wbold_h} \\[1.5ex]
\nonumber & = - \ipt{\ubold_h \otimes \ubold_h}{\nabla_h \wbold_h} - \ipt{\widetilde{p}_h}{\nabla \cdot \wbold_h} + \ipbt{ \hat{\bm{\sigma}}_{\text{inv}} \, \nbold}{\wbold_h}.
\end{align}
\begin{align}
& -\ipt{ \nabla_h \cdot \left( \nabla_h \ubold_h + \nabla_h \ubold_{h}^{T} - \frac{2}{3} \left(\nabla \cdot \ubold_h \right) \mathbb{I} \right)}{\wbold_h}  \label{moment_ibp_two} \\[1.5ex]
\nonumber &= \ipt{\nabla_h \ubold_h + \nabla_h \ubold_{h}^{T} - \frac{2}{3} \left(\nabla \cdot \ubold_h \right) \mathbb{I}}{\nabla_h \wbold_h} - \ipbt{\left(\nabla_h \ubold_h + \nabla_h \ubold_{h}^{T} - \frac{2}{3} \left(\nabla \cdot \ubold_h \right) \mathbb{I} \right) \nbold}{\wbold_h} \\[1.5ex] 
\nonumber & \equiv \ipt{\nabla_h \ubold_h + \nabla_h \ubold_{h}^{T} - \frac{2}{3} \left(\nabla \cdot \ubold_h \right) \mathbb{I}}{\nabla_h \wbold_h} - \ipbt{ \hat{\bm{\sigma}}_{\text{vis}} \, \nbold}{\wbold_h}. 
\end{align}
One may expand each term in Eq.~\eqref{moment_ibp_two} by integrating by parts, inserting a numerical flux $\hat{\bm{\varphi}}_{\text{vis}}$, and integrating by parts again as follows
\begin{align}
\ipt{ \nabla_h \ubold_h}{\nabla_h \wbold_h} & = -\ipt{\ubold_h}{\nabla_h \cdot \left( \nabla_h \wbold_h \right)} + \ipbt{ \ubold_h}{ \left(\nabla_h \wbold_h \right) \nbold} \label{moment_ibp_three} \\[1.5ex]
\nonumber & \equiv -\ipt{\ubold_h}{\nabla_h \cdot \left( \nabla_h \wbold_h \right)} + \ipbt{\hat{\bm{\varphi}}_{\text{vis}}}{ \left(\nabla_h \wbold_h \right) \nbold} \\[1.5ex]
\nonumber & = \ipt{ \nabla_h \ubold_h}{\nabla_h \wbold_h} + \ipbt{\hat{\bm{\varphi}}_{\text{vis}} -  \ubold_h }{ \left(\nabla_h \wbold_h \right) \nbold}.
\end{align}
\begin{align}
\ipt{ \nabla_h \ubold_h^T}{\nabla_h \wbold_h} & = \ipt{\nabla_h \ubold_h}{ \nabla_h \wbold_h^T} \label{moment_ibp_four} \\[1.5ex]
\nonumber &= -\ipt{\ubold_h}{\nabla_h \cdot \left( \nabla_h \wbold_h^T \right)} + \ipbt{ \ubold_h}{ \left(\nabla_h \wbold_h^T \right) \nbold} \\[1.5ex]
\nonumber & \equiv -\ipt{\ubold_h}{\nabla_h \cdot \left( \nabla_h \wbold_h^T \right)} + \ipbt{\hat{\bm{\varphi}}_{\text{vis}} }{ \left(\nabla_h \wbold_h^T \right) \nbold} \\[1.5ex]
\nonumber & = \ipt{ \nabla_h \ubold_h^T}{\nabla_h \wbold_h} + \ipbt{\hat{\bm{\varphi}}_{\text{vis}} -  \ubold_h }{ \left(\nabla_h \wbold_h^T \right) \nbold}.
\end{align}
\begin{align}
\ipt{- \frac{2}{3}  \left(\nabla \cdot \ubold_h \right) \mathbb{I}}{\nabla_h \wbold_h} &= - \frac{2}{3} \ipt{\nabla \cdot \ubold_h }{ \nabla \cdot \wbold_h} \label{moment_ibp_five} \\[1.5ex]
\nonumber & = -\frac{2}{3} \left( -\ipt{\ubold_h}{ \nabla_h \left( \nabla \cdot \wbold_h \right)} + \ipbt{ \ubold_h}{ \left(\nabla \cdot \wbold_h \right) \nbold} \right) \\[1.5ex]
\nonumber & \equiv -\frac{2}{3} \left( -\ipt{\ubold_h}{ \nabla_h \left( \nabla \cdot \wbold_h \right)} + \ipbt{\hat{\bm{\varphi}}_{\text{vis}}}{ \left(\nabla \cdot \wbold_h \right) \nbold} \right) \\[1.5ex]
\nonumber & = \ipt{- \frac{2}{3}  \left(\nabla \cdot \ubold_h \right) \mathbb{I}}{\nabla_h \wbold_h} -\frac{2}{3} \ipbt{\hat{\bm{\varphi}}_{\text{vis}} -  \ubold_h}{ \left(\nabla \cdot \wbold_h \right) \nbold}.
\end{align}
Upon combining Eqs.~\eqref{moment_ibp_three} -- \eqref{moment_ibp_five}, one obtains
\begin{align}
& \ipt{\nabla_h \ubold_h + \nabla_h \ubold_{h}^{T} - \frac{2}{3} \left(\nabla \cdot \ubold_h \right) \mathbb{I}}{\nabla_h \wbold_h} \label{moment_ibp_six} \\[1.5ex]
\nonumber & \equiv \ipt{\nabla_h \ubold_h + \nabla_h \ubold_{h}^{T} - \frac{2}{3} \left(\nabla \cdot \ubold_h \right) \mathbb{I}}{\nabla_h \wbold_h} \\[1.5ex]
\nonumber & +  \ipbt{\hat{\bm{\varphi}}_{\text{vis}} -  \ubold_h }{ \left(\nabla_h \wbold_h + \nabla_h \wbold_h^T -\frac{2}{3} \left( \nabla \cdot \wbold_h  \right) \mathbb{I} \right) \nbold}. 
\end{align}
Finally, one may substitute Eqs.~\eqref{moment_ibp_zero}, \eqref{moment_ibp_two}, and \eqref{moment_ibp_six}, into Eq.~\eqref{moment_ibp_neg} in order to obtain Eq.~\eqref{moment_cons_disc}.


{\footnotesize\bibliography{technical-refs}}

\begin{thebibliography}{10}
\expandafter\ifx\csname url\endcsname\relax
  \def\url#1{\texttt{#1}}\fi
\expandafter\ifx\csname urlprefix\endcsname\relax\def\urlprefix{URL }\fi
\expandafter\ifx\csname href\endcsname\relax
  \def\href#1#2{#2} \def\path#1{#1}\fi

\bibitem{john2017divergence}
V.~John, A.~Linke, C.~Merdon, M.~Neilan, L.~G. Rebholz, On the divergence
  constraint in mixed finite element methods for incompressible flows, SIAM
  review 59~(3) (2017) 492--544.

\bibitem{franca1988two}
L.~P. Franca, T.~J. Hughes, Two classes of mixed finite element methods,
  Computer Methods in Applied Mechanics and Engineering 69~(1) (1988) 89--129.

\bibitem{olshanskii2004grad}
M.~Olshanskii, A.~Reusken, Grad-div stablilization for {Stokes} equations,
  Mathematics of Computation 73~(248) (2004) 1699--1718.

\bibitem{olshanskii2009grad}
M.~Olshanskii, G.~Lube, T.~Heister, J.~L{\"o}we, Grad--div stabilization and
  subgrid pressure models for the incompressible {Navier}--{Stokes} equations,
  Computer Methods in Applied Mechanics and Engineering 198~(49-52) (2009)
  3975--3988.

\bibitem{gelhard2005stabilized}
T.~Gelhard, G.~Lube, M.~A. Olshanskii, J.-H. Starcke, Stabilized finite element
  schemes with {LBB}-stable elements for incompressible flows, Journal of
  Computational and Applied Mathematics 177~(2) (2005) 243--267.

\bibitem{braack2007stabilized}
M.~Braack, E.~Burman, V.~John, G.~Lube, Stabilized finite element methods for
  the generalized {Oseen} problem, Computer Methods in Applied Mechanics and
  Engineering 196~(4-6) (2007) 853--866.

\bibitem{case2011connection}
M.~A. Case, V.~J. Ervin, A.~Linke, L.~G. Rebholz, A connection between
  {Scott}--{Vogelius} and grad-div stabilized {Taylor}--{Hood} {FE}
  approximations of the {Navier}--{Stokes} equations, SIAM Journal on Numerical
  Analysis 49~(4) (2011) 1461--1481.

\bibitem{jenkins2014parameter}
E.~W. Jenkins, V.~John, A.~Linke, L.~G. Rebholz, On the parameter choice in
  grad-div stabilization for the {Stokes} equations, Advances in Computational
  Mathematics 40~(2) (2014) 491--516.

\bibitem{john2016finite}
V.~John, Finite Element Methods for Incompressible Flow Problems, Springer,
  2016.

\bibitem{burman2008stabilized}
E.~Burman, A.~Linke, Stabilized finite element schemes for incompressible flow
  using {Scott}--{Vogelius} elements, Applied Numerical Mathematics 58~(11)
  (2008) 1704--1719.

\bibitem{cockburn2005locally}
B.~Cockburn, G.~Kanschat, D.~Sch{\"o}tzau, A locally conservative {LDG} method
  for the incompressible {Navier}-{Stokes} equations, Mathematics of
  Computation 74~(251) (2005) 1067--1095.

\bibitem{guzman2016h}
J.~Guzm{\'a}n, C.-W. Shu, F.~A. Sequeira, H-(div) conforming and {DG} methods
  for incompressible {Euler's} equations, IMA Journal of Numerical Analysis
  37~(4) (2016) 1733--1771.

\bibitem{linke2014role}
A.~Linke, On the role of the {Helmholtz} decomposition in mixed methods for
  incompressible flows and a new variational crime, Computer Methods in Applied
  Mechanics and Engineering 268 (2014) 782--800.

\bibitem{linke2016pressure}
A.~Linke, C.~Merdon, Pressure-robustness and discrete {Helmholtz} projectors in
  mixed finite element methods for the incompressible {Navier}--{Stokes}
  equations, Computer Methods in Applied Mechanics and Engineering 311 (2016)
  304--326.

\bibitem{boffi2013mixed}
D.~Boffi, F.~Brezzi, M.~Fortin, Mixed Finite Element Methods and Applications,
  Vol.~44, Springer, 2013.

\bibitem{schroeder2018divergence}
P.~W. Schroeder, G.~Lube, Divergence-free {H(div)-FEM} for time-dependent
  incompressible flows with applications to high {Reynolds} number vortex
  dynamics, Journal of Scientific Computing (2018) 1--29.

\bibitem{schroeder2018towards}
P.~W. Schroeder, C.~Lehrenfeld, A.~Linke, G.~Lube, Towards computable flows and
  robust estimates for inf-sup stable {FEM} applied to the time-dependent
  incompressible {Navier}--{Stokes} equations, SeMA Journal, Boletin de la
  Sociedad Española de Matemática Aplicada 75~(4) (2018) 629--653.

\bibitem{zhang2005new}
S.~Zhang, A new family of stable mixed finite elements for the {3D} {Stokes}
  equations, Mathematics of Computation 74~(250) (2005) 543--554.

\bibitem{falk2013stokes}
R.~S. Falk, M.~Neilan, Stokes complexes and the construction of stable finite
  elements with pointwise mass conservation, SIAM Journal on Numerical Analysis
  51~(2) (2013) 1308--1326.

\bibitem{guzman2014conforming}
J.~Guzm{\'a}n, M.~Neilan, Conforming and divergence-free {Stokes} elements on
  general triangular meshes, Mathematics of Computation 83~(285) (2014) 15--36.

\bibitem{lehrenfeld2016high}
C.~Lehrenfeld, J.~Sch{\"o}berl, High order exactly divergence-free hybrid
  discontinuous {Galerkin} methods for unsteady incompressible flows, Computer
  Methods in Applied Mechanics and Engineering 307 (2016) 339--361.

\bibitem{majda2002vorticity}
A.~J. Majda, A.~L. Bertozzi, Vorticity and Incompressible Flow, Vol.~27,
  Cambridge University Press, 2002.

\bibitem{palha2017mass}
A.~Palha, M.~Gerritsma, A mass, energy, enstrophy and vorticity conserving
  ({MEEVC}) mimetic spectral element discretization for the {2D} incompressible
  {Navier}--{Stokes} equations, Journal of Computational Physics 328 (2017)
  200--220.

\bibitem{coppola2019discrete}
G.~Coppola, F.~Capuano, L.~de~Luca, Discrete energy-conservation properties in
  the numerical simulation of the {Navier}--{Stokes} equations, Applied
  Mechanics Reviews 71~(1) (2019) 010803.

\bibitem{schroeder2017pressure}
P.~W. Schroeder, G.~Lube, Pressure-robust analysis of divergence-free and
  conforming {FEM} for evolutionary incompressible {Navier}-{Stokes} flows,
  Journal of Numerical Mathematics 25~(4) (2017) 249--276.

\bibitem{charnyi2017conservation}
S.~Charnyi, T.~Heister, M.~A. Olshanskii, L.~G. Rebholz, On conservation laws
  of {Navier}--{Stokes} {Galerkin} discretizations, Journal of Computational
  Physics 337 (2017) 289--308.

\bibitem{charnyi2018efficient}
S.~Charnyi, T.~Heister, M.~A. Olshanskii, L.~G. Rebholz, Efficient
  discretizations for the {EMAC} formulation of the incompressible
  {Navier}--{Stokes} equations, Applied Numerical Mathematics 141 (2018)
  220--233.

\bibitem{charnyi2018emac}
S.~Charnyi, The {EMAC} scheme for {Navier}-{Stokes} simulations, and
  application to flow past bluff bodies, Ph.D. thesis, Clemson University
  (2018).

\bibitem{lehmkuhl2019low}
O.~Lehmkuhl, G.~Houzeaux, H.~Owen, G.~Chrysokentis, I.~Rodriguez, A
  low-dissipation finite element scheme for scale resolving simulations of
  turbulent flows, Journal of Computational Physics 390~(1) (2019) 51--65.

\bibitem{arnold2002mixed}
D.~N. Arnold, R.~Winther, Mixed finite elements for elasticity, Numerische
  Mathematik 92~(3) (2002) 401--419.

\bibitem{arnold2008finite}
D.~Arnold, G.~Awanou, R.~Winther, Finite elements for symmetric tensors in
  three dimensions, Mathematics of Computation 77~(263) (2008) 1229--1251.

\bibitem{hu2014family}
J.~Hu, S.~Zhang, A family of conforming mixed finite elements for linear
  elasticity on triangular grids, arXiv preprint arXiv:1406.7457.

\bibitem{hu2014simple}
J.~Hu, H.~Man, S.~Zhang, A simple conforming mixed finite element for linear
  elasticity on rectangular grids in any space dimension, Journal of Scientific
  Computing 58~(2) (2014) 367--379.

\bibitem{cockburn2017devising}
B.~Cockburn, G.~Fu, Devising superconvergent {HDG} methods with symmetric
  approximate stresses for linear elasticity by {M}-decompositions, IMA Journal
  of Numerical Analysis 38~(2) (2017) 566--604.

\bibitem{arnold2003nonconforming}
D.~N. Arnold, R.~Winther, Nonconforming mixed elements for elasticity,
  Mathematical Models and Methods in Applied Sciences 13~(03) (2003) 295--307.

\bibitem{gopalakrishnan2011symmetric}
J.~Gopalakrishnan, J.~Guzm{\'a}n, Symmetric nonconforming mixed finite elements
  for linear elasticity, SIAM Journal on Numerical Analysis 49~(4) (2011)
  1504--1520.

\bibitem{cockburn2017note}
B.~Cockburn, G.~Fu, W.~Qiu, A note on the devising of superconvergent {HDG}
  methods for {Stokes} flow by {M}-decompositions, IMA Journal of Numerical
  Analysis 37~(2) (2017) 730--749.

\bibitem{giacomini2018superconvergent}
M.~Giacomini, A.~Karkoulias, R.~Sevilla, A.~Huerta, A superconvergent {HDG}
  method for {Stokes} flow with strongly enforced symmetry of the stress
  tensor, Journal of Scientific Computing 77~(3) (2018) 1679--1702.

\bibitem{tezduyar1991stabilized}
T.~E. Tezduyar, Stabilized finite element formulations for incompressible flow
  computations, in: Advances in Applied Mechanics, Vol.~28, Elsevier, 1991, pp.
  1--44.

\bibitem{hong2016robust}
Q.~Hong, J.~Kraus, J.~Xu, L.~Zikatanov, A robust multigrid method for
  discontinuous {Galerkin} discretizations of {Stokes} and linear elasticity
  equations, Numerische Mathematik 132~(1) (2016) 23--49.

\bibitem{hong2016uniformly}
Q.~Hong, J.~Kraus, Uniformly stable discontinuous {Galerkin} discretization and
  robust iterative solution methods for the {Brinkman} problem, SIAM Journal on
  Numerical Analysis 54~(5) (2016) 2750--2774.

\bibitem{schlichting2016boundary}
H.~Schlichting, K.~Gersten, Boundary-layer Theory, Springer, 2016.

\bibitem{peterson2018overview}
J.~W. Peterson, A.~D. Lindsay, F.~Kong, Overview of the incompressible
  {Navier}-{Stokes} simulation capabilities in the {MOOSE} framework, Advances
  in Engineering Software 119 (2018) 68--92.

\bibitem{allaire2007numerical}
G.~Allaire, Numerical Analysis and Optimization: An Introduction to
  Mathematical Modelling and Numerical Simulation, Oxford University Press,
  2007.

\bibitem{schirra2012new}
O.~D. Schirra, New {Korn}-type inequalities and regularity of solutions to
  linear elliptic systems and anisotropic variational problems involving the
  trace-free part of the symmetric gradient, Calculus of Variations and Partial
  Differential Equations 43~(1-2) (2012) 147--172.

\bibitem{breit2017trace}
D.~Breit, A.~Cianchi, L.~Diening, Trace-free {Korn} inequalities in {Orlicz}
  spaces, SIAM Journal on Mathematical Analysis 49~(4) (2017) 2496--2526.

\bibitem{DiPietro11}
D.~A. Di~Pietro, A.~Ern, Mathematical Aspects of Discontinuous Galerkin
  Methods, Vol.~69, Springer Science \& Business Media, Berlin Heidelberg,
  2011.

\bibitem{Dallmann15}
H.~Dallmann, D.~Arndt, G.~Lube, Local projection stabilization for the {Oseen}
  problem, IMA Journal of Numerical Analysis 36~(2) (2015) 796--823.

\bibitem{logg2012automated}
A.~Logg, K.-A. Mardal, G.~Wells, Automated solution of differential equations
  by the finite element method: {The FEniCS} book, Vol.~84, Springer Science \&
  Business Media, 2012.

\bibitem{alnaes2015fenics}
M.~S. Alnaes, J.~Blechta, J.~Hake, A.~Johansson, B.~Kehlet, A.~Logg,
  C.~Richardson, J.~Ring, M.~E. Rognes, G.~N. Wells, The {FEniCS} project
  version 1.5, Archive of Numerical Software 3~(100) (2015) 9--23.

\end{thebibliography}

\end{document}